\documentclass[11pt,amscd,combelow,anysize,amsxtra]{amsart} 
\usepackage[margin=1in]{geometry}

\usepackage{amsmath}
\usepackage{hyperref}
\usepackage[capitalize]{cleveref}
\usepackage{graphicx}
\usepackage{tikz}
\usepackage{dynkin-diagrams}
\usepackage{vcell}  
\usepackage{array} 
\usetikzlibrary{shapes.misc}

\tikzset{cross/.style={cross out, draw=black, minimum size=2*(#1-\pgflinewidth), inner sep=1pt, outer sep=1pt},
	cross/.default={5pt}}
\usepackage{adjustbox}

\usepackage{calligra}
\usepackage{amsthm,amssymb,mathrsfs,mathtools,bm,eucal,tensor} 
\usepackage{microtype} 
\usepackage[scaled]{beramono,berasans}
\usepackage{enumerate,comment,braket,xspace,tikz-cd} 
\usetikzlibrary{arrows}
\usepackage[all,cmtip]{xy} 
\usepackage[utf8]{inputenc} 
\usepackage[T1]{fontenc} 
\usepackage{lmodern}
\definecolor{linkcolor}{HTML}{005050}
\usepackage{xparse}
\usepackage{url}
\usepackage[toc,page]{appendix}
\usepackage{stmaryrd}
\usepackage{lscape}
\newcommand{\myfatslash}{\mathbin{\mkern-6mu\fatslash}}

\usepackage{subcaption}

\usepackage{pgfplots} 
\pgfplotsset{compat=newest}

\theoremstyle{plain}
\newtheorem{thm-intro}{Theorem}
\newtheorem{thm}{Theorem}[section]

\newtheorem*{thm*}{Theorem}
\newtheorem{lem}[thm]{Lemma}
\newtheorem{lemma}[thm]{Lemma}
\newtheorem*{lem*}{Lemma}
\newtheorem{prop}[thm]{Proposition}
\newtheorem{proposition}[thm]{Proposition}
\newtheorem{assumption}[thm]{Assumption}

\newtheorem{cor}[thm]{Corollary}
\newtheorem{corollary}[thm]{Corollary}
\theoremstyle{definition}
\newtheorem{definition}[thm]{Definition}
\newtheorem{defn}[thm]{Definition}

\theoremstyle{remark}
\newtheorem{rem}[thm]{Remark}
\newtheorem{remark}[thm]{Remark}
\newtheorem{example}[thm]{Example}
\numberwithin{equation}{section}

\newtheorem{ex}[thm]{Example}

\newcommand{\fc}{\mathfrak{c}}
\newcommand{\fl}{\mathfrak{l}}
\newcommand{\fh}{\mathfrak{h}}

\newcommand{\wt}{\widetilde}
\newcommand{\ol}{\overline}

\newcommand{\fsl}{\mathfrak{sl}}
\newcommand{\reg}{\mathrm{reg}}
\newcommand{\bG}{\mathbb{G}}
\newcommand{\canC}{\mathbf{C}}
\newcommand{\canT}{\mathbf{T}}
\newcommand{\canA}{\mathbf{A}}
\usepackage{bbold}
\newcommand{\canft}{\mathfrak{t}}
\newcommand{\canfa}{\mathfrak{a}}
\newcommand{\fz}{\mathfrak{z}}

\DeclareMathOperator{\SO}{SO}
\DeclareMathOperator{\SL}{SL}
\DeclareMathOperator{\GL}{GL}
\DeclareMathOperator{\Sp}{Sp}

\DeclareMathOperator{\PGL}{PGL}
\newcommand{\fa}{\mathfrak{a}}

\let\C\relax

\newcommand{\C}{\mathbb C}

\newcommand{\Z}{\mathbb Z}

\newcommand\git{{/\!\!/}}

\newcommand{\fD}{\mathfrak D}

\newcommand{\fs}{\mathfrak s}
\newcommand{\ft}{\mathfrak t}
\newcommand{\fp}{\mathfrak{p}}
\newcommand{\fg}{\mathfrak{g}}
\newcommand{\fk}{\mathfrak{k}}
\newcommand{\Norm}{\ol{\ft\times_{\fc_G}\fc}}

\newcommand{\cA}{\mathcal A}

\newcommand{\ad}{\mathrm{ad}}

\newcommand{\cD}{\mathcal D}

\newcommand{\cI}{\mathcal I}

\newcommand{\cM}{\mathcal M}
\newcommand{\cN}{\mathcal N}
\newcommand{\cO}{\mathcal O}
\newcommand{\cP}{\mathcal P}

\newcommand{\cS}{\mathcal S}
\newcommand{\cT}{\mathcal T}

\DeclareFontFamily{U}{BOONDOX-calo}{\skewchar\font=45 }
\DeclareFontShape{U}{BOONDOX-calo}{m}{n}{<-> s*[1.05] BOONDOX-r-calo}{}
\DeclareFontShape{U}{BOONDOX-calo}{b}{n}{<-> s*[1.05] BOONDOX-b-calo}{}
\DeclareMathAlphabet{\mathcalboondox}{U}{BOONDOX-calo}{m}{n}

\newcommand{\bbG}{\mathbb G}

\newcommand{\bA}{\mathbf A}

\newcommand{\ul}{\underline}

\makeatletter
\let\save@mathaccent\mathaccent
\newcommand*\if@single[3]{%
	\setbox0\hbox{${\mathaccent"0362{#1}}^H$}%
	\setbox2\hbox{${\mathaccent"0362{\kern0pt#1}}^H$}%
	\ifdim\ht0=\ht2 #3\else #2\fi
}
\newcommand*\rel@kern[1]{\kern#1\dimexpr\macc@kerna}
\newcommand*\widebar[1]{\@ifnextchar^{{\wide@bar{#1}{0}}}{\wide@bar{#1}{1}}}
\newcommand*\wide@bar[2]{\if@single{#1}{\wide@bar@{#1}{#2}{1}}{\wide@bar@{#1}{#2}{2}}}
\newcommand*\wide@bar@[3]{%
	\begingroup
	\def\mathaccent##1##2{%
		\let\mathaccent\save@mathaccent
		\if#32 \let\macc@nucleus\first@char \fi
		\setbox\z@\hbox{$\macc@style{\macc@nucleus}_{}$}%
		\setbox\tw@\hbox{$\macc@style{\macc@nucleus}{}_{}$}%
		\dimen@\wd\tw@
		\advance\dimen@-\wd\z@
		\divide\dimen@ 3
		\@tempdima\wd\tw@
		\advance\@tempdima-\scriptspace
		\divide\@tempdima 10
		\advance\dimen@-\@tempdima
		\ifdim\dimen@>\z@ \dimen@0pt\fi
		\rel@kern{0.6}\kern-\dimen@
		\if#31
		\overline{\rel@kern{-0.6}\kern\dimen@\macc@nucleus\rel@kern{0.4}\kern\dimen@}%
		\advance\dimen@0.4\dimexpr\macc@kerna
		\let\final@kern#2%
		\ifdim\dimen@<\z@ \let\final@kern1\fi
		\if\final@kern1 \kern-\dimen@\fi
		\else
		\overline{\rel@kern{-0.6}\kern\dimen@#1}%
		\fi
	}%
	\macc@depth\@ne
	\let\math@bgroup\@empty \let\math@egroup\macc@set@skewchar
	\mathsurround\z@ \frozen@everymath{\mathgroup\macc@group\relax}%
	\macc@set@skewchar\relax
	\let\mathaccentV\macc@nested@a
	\if#31
	\macc@nested@a\relax111{#1}%
	\else
	\def\gobble@till@marker##1\endmarker{}%
	\futurelet\first@char\gobble@till@marker#1\endmarker
	\ifcat\noexpand\first@char A\else
	\def\first@char{}%
	\fi
	\macc@nested@a\relax111{\first@char}%
	\fi
	\endgroup
}
\makeatother

\newcommand{\fR}{\mathfrak{R}}

\usetikzlibrary{decorations.markings} 
\tikzset{
	closed/.style = {decoration = {markings, mark = at position 0.5 with { \node[transform shape, xscale = .8, yscale=.4] {/}; } }, postaction = {decorate} },
	open/.style = {decoration = {markings, mark = at position 0.5 with { \node[transform shape, scale = .7] {$\circ$}; } }, postaction = {decorate} }
}

\DeclareMathOperator{\Aut}{Aut}

\DeclareMathOperator{\Hom}{Hom}

\DeclareMathOperator{\Pic}{Pic}

\DeclareMathOperator{\Spec}{Spec}

\DeclareMathOperator{\Sym}{Sym}

\newcommand{\fv}{\mathfrak{v}}

\AtBeginDocument{%
	\def\MR#1{}
}

\title{The Hitchin Fibration for Symmetric Pairs}
\author{Thomas Hameister}
\author{Benedict Morrissey}

\begin{document}

	\begin{abstract}
		We introduce and describe the ``regular quotient'' for the Hitchin fibration for symmetric spaces and explain some basic consequences for Higgs bundles. We include an invariant theoretic approach to spectral covers in this setting for the particular space $\GL_{2n}/\GL_n\times \GL_n$. We also include a study of the regular centralizer group scheme for quasisplit pairs, including a Galois description of a closely related group scheme. We collect some basic consequences for Hitchin systems associated to such pairs.
	\end{abstract}
	
	\maketitle

	\setcounter{tocdepth}{1}
	\tableofcontents

	\section{Introduction}
	
	Let $G$ be a reductive group over an algebraically closed field $k$. We assume that either the characteristic of $k$ is zero or that $G$ satisfies the ``standard hypotheses'' of Section \ref{subsec: conditions on char}. Fix a smooth, projective curve $C$ over $k$ and a line bundle $D$ on $C$ such that $\dim H^0(C,D)>0$. That is, when the genus of $C$ is at least 2, $D$ is degree at least $2g$ or $D$ is the canonical bundle of $C$. To such data, one can associate the moduli stack of $G$ Higgs bundles
	\[
	\cM_G = \mathrm{Maps}(C,[\fg_D/G])
	\]
	where $\fg_D = \fg\otimes D$ is the twisted bundle of Lie algebras and $[\fg_D/G]$ is the stack quotient. When $D$ is the canonical bundle of $C$ and $k=\C$, the associated moduli space of polystable Higgs bundles with vanishing chern class was shown to be homeomorphic (with respect to the classical topology) to the moduli space of reductive representations of the fundamental group $\pi_1(C)$ in $G(\C)$ by Corlette and Simpson \cite{corlette,simpson}, generalizing work of Hitchin and Donaldson \cite{hitchin1987self, donaldson1987twisted}.  Furthermore this homeomorphism restricts to a (real) diffeomorphism on the smooth locus. In \cite{hitchin}, Hitchin introduced a beautiful fibration 
	\[
	h_G\colon \cM_G\to \cA_G
	\]
	where $\cA_G$ is an affine space of half the dimension of $\cM_G$, which is a global analogue of the characteristic polynomial map, and whose generic fiber is essentially a union of abelian varieties.  Since its introduction, the Hitchin fibration has found applications across a wide range of mathematics. Among its remarkable properties:  Over $\C$, it is a Lagrangian fibration with a known mirror dual (in the sense of SYZ mirror symmetry); there are several statements exhibiting mirror symmetric dualities between the Hitchin fibrations for $G$ and its Langlands dual $G^{\vee}$ including \cite{hausel-thaddeus, kapustin-witten, donagi-pantev, GWZ}; and it provides a geometric framework for the theory of endoscopy leading to B.C. Ng\^o's proof of the Fundamental Lemma \cite{ngo2010lemme}.
	
	In this paper, we study a generalized Hitchin fibration associated to a symmetric pair. In particular, we let 
	\[
	\theta\colon G\to G
	\]
	be an algebraic involution of $G$, and let $K = (G^\theta)^\circ$ be the connected component of the fixed points of $\theta$ on $G$. Fix a smooth, closed subgroup $H$ of $G$ such that
	\[
	K\subset H\subset N_G(K).
	\]
	The involution $\theta$ induces a Cartan decomposition
	\[
	\fg = \fk\oplus \fp
	\]
	where $\fk$ and $\fp$ are the $(+1)$ and $(-1)$ eigenspaces, respectively, of $\theta$ on $\fg$. To the $G$-variety $X = G/H$ we associate a moduli space of relative Higgs bundles
	\[
	\cM = \cM_X = \mathrm{Maps}(C,[\fp_D/H])
	\]
	For $k=\C$ and $D$ the canonical bundle, there is an extension of non-abelian Hodge theory that replaces morphisms $\pi_{1}(C)\rightarrow G(\mathbb{C})$, with morphisms $\pi_{1}(C)\rightarrow G(\mathbb{R})$ and $G$-Higgs bundles with relative Higgs bundles for a symmetric pair associated to the real form $G(\mathbb{R})$ \cite{garcia2009hitchin, garcia2,garcia3}.  This result had already appeared for compact real forms in Narasimhan--Seshadri \cite{NS}, and there are implicitly some results for split real forms in \cite{hitchin-teichmuller}.  We refer to the survey \cite{garcia4} for more on applications of Higgs bundles for symmetric pairs to representations of the fundamental group. One still has a relative Hitchin fibration
	\[
	h\colon \cM\to \cA
	\]
	with $\cA =\mathrm{Maps}(X,(\fp\git H)_D)$ the affine space classifying maps to the twisted GIT quotient $\fp\git H := \Spec( k[\fp]^H )$. The geometry of fibers of these Hitchin systems were studied extensively using spectral covers in \cite{schaposnik2013spectral,schaposnik2014introduction,schaposnik2015spectral,hitchin2014nonabelianization,baraglia2019cayley,branco2018higgs} among others, and a theory of cameral covers was initiated in \cite{peon2013higgs,gppn}. 
	
	The generic fibers of the Hitchin fibration for symmetric spaces involve two novel geometric behaviors:  First, there may be unexpected components in fibers. For example, in Schaposnik's thesis \cite{schaposnik2013spectral} for the symmetric space $X = \GL_{2n}/(\GL_n\times \GL_n)$, fibers are generically identified with a disjoint union of $2^\ell$ copies of the Picard stack classifying line bundles on a spectral curve for an explicit $\ell$. Moreover, the connected components may still fail to be abelian varieties. For example, Hitchin and Schaposnik give a symmetric pair for which the fibers of $h$ over a generic $a$ are identified with the space of rank two vector bundles on a spectral curve \cite{hitchin2014nonabelianization}.
	
	In this work, we describe in greater detail the geometry of the Hitchin fibration for symmetric spaces, providing invariant theoretic explanations for the anomalous behavior of the fibers and giving descriptions of the fibration $h$ in families. Our results will restrict to the regular locus. That is, we let $\fp^\reg$ denote the set of $x\in \fp$ whose centralizer in $H$ is of minimal dimension. We will restrict to the sub-locus $\cM^\reg$ of $\cM$ consisting of maps from $C$ valued in the open substack $[\fp^\reg_D/H]$. We will make use of a new invariant theoretic package developed in unpublished work of Morrissey and Ng\^o \cite{morrissey2022reg}. Following \emph{loc cit}, the basic structure theorem can be expressed as follows.
	\begin{thm}{(See Proposition \ref{prop: Areg to A is etale} and Theorem \ref{thm: general structure of h myfatslash})}
		\label{thm: general structure}
		There is a factorization of $h$ into
		\[
		\cM^{\reg}\xrightarrow{h^{\reg}} \cA^{\reg}\xrightarrow{\phi} \cA
		\]
		such that the map $\phi$ is \'etale. When $(G,\theta,H)$ is $\theta$-quasisplit (see Section \ref{sec: regularity and qsplit}), there is a commutative group stack $\cP$, defined over $\cA^\reg$, acting on $\cM^\reg$ so that $\cM^\reg$ is a $\cP$-torsor over $\cA^\reg$.
	\end{thm}

    We note that the map $\phi$ above can be made completely explicit; see Corollary \ref{cor: Areg to A}.
	
	The above theorem is a reflection of some new constructions in invariant theory. Namely, consider the Chevalley style morphism 
	\[
	\chi\colon [\fp^\reg/H]\to \fp\git H
	\]
	from the stack quotient of $H$ acting on $\fp^\reg$ to the corresponding GIT quotient. Unlike the map $[\fg^\reg/G]\to \fg\git G$, the map $\chi$ fails to be a gerbe. In \cite{gppn}, they suggest studying an intermediate quotient obtained by rigidifying the stack $[\fp^\reg/H]$ by inertia. This intermediate quotient, which we will denote by $\fp^\reg\myfatslash H$ and call the \emph{regular quotient}, is a coarse moduli space classifying all regular $H$-orbits in $\fp$. The map $\chi$ factors as
	\[
	[\fp^\reg/H]\to \fp^\reg\myfatslash H \to \fp\git H
	\]
	where the first map is a gerbe and the second is a non-separated cover. 
	
	Our main result is a completely explicit description of the regular quotient for symmetric pairs, in particular showing that $\fp^{\reg}\myfatslash H$ is a scheme.
	\begin{thm}{(See Theorem \ref{theorem: gluing pattern})}
		\label{thm: main}
		Let $(G,\theta,H)$ be a symmetric pair with $G$ simple.
		\begin{enumerate}
			\item If $(G,\theta,H)$ is not isogeneous to the split symmetric pair on $G = \SO_{4n}$, then there is a Zariski closed subset $Z\subset \fc$ such that
			\[
			\fp^\reg\myfatslash H\simeq \fc\coprod_{\fc\setminus Z}\fc
			\]
			consists of two copies of $\fc$ glued along the complement of $Z$. Moreover, this closed subset $Z$ can be computed explicitly via an inductive procedure.
			\item The split symmetric pair on $G = \SO_{4n}$ can be explicitly described as well, but involves gluing patterns for four sheets. See Example \ref{ex son son so2n}.
			\item The case of general $G$ can be reduced to the simple case.
		\end{enumerate}
	\end{thm}
	
	The gerbe
	\[
	[\fp^\reg/H]\to \fp^\reg\myfatslash H
	\]
	is banded by (the descent of) the group scheme $I^\reg\to \fp^\reg$ whose fiber over $x\subset \fp^\reg$ is the centralizer of $x$ in $H$
	\[
	I^\reg_x= \{h\in H\colon \mathrm{Ad}(h)\cdot x = x\}.
	\]
	Any question of the behavior of the map $h^{\reg}$ must, therefore, be inextricably linked to the question of describing the regular centralizer group scheme $I^\reg$. Our results here are restricted to $\theta$-quasisplit symmetric pairs, that is to say, those for which $I^\reg$ is abelian. For $\theta$-quasisplit symmetric pairs, $I^\reg$ descends to a smooth, commutative group scheme $J$ over the GIT quotient, which we will denote by $\fc:=\fp\git H$. 
	
	Let $\canT$ be the canonical torus for $G$ with Lie algebra $\canft$. To a symmetric space $X$, one has attached a canonical torus $\canA$ with Lie algebra $\canfa$ and, when $X$ is $\theta$-quasisplit, a canonical projection $\canT\to \canA$. The abelian algebra $\canfa$ comes equipped with a natural root system (which agrees with the ``restricted root system'' of the symmetric space), and we denote the corresponding Weyl group by $W_X$. For $\theta$-quasisplit symmetric spaces, $W_X$ is a subgroup of the Weyl group $W$. 
	
	Knop proved a Chevalley-style isomorphism $\fc\simeq \canfa\git W_X$ (See \cite[Page 12]{knop_german} and \cite[Theorem 4.1]{knop1996automorphisms}), and hence it admits a finite flat cover $\fa\to \fc$. We let $\canC_1$ be the kernel of the canonical map $\canT\to \canA$, and for $\nu\in W/W_X$, we let $\canC_\nu = \nu\cdot \canC_1$. Following the perspective taken in \cite{donagi2002gerbe} and \cite{ngo2010lemme}, we give the following description of the regular centralizer group scheme $J$.

\begin{thm}
\label{thm: ses_intro}
    (See Theorem \ref{thm: galois description of JA}) We have a short exact sequence
        \[
        1\to J\to J_G|_\fc\to J_X\to 1
        \]
        where $J_G$ is the regular centralizer group scheme for the action of $G$ on its Lie algebra $\fg$, and $J_X$ is an open affine subgroup of the Weil restriction
        \[
        \mathrm{Res}^\fa_\fc(\canA\times \fa)^{W_X}.
        \]
\end{thm}

Theorem \ref{thm: ses_intro} should be seen as giving a Galois description of $J$ as a 2-term complex $J_G|_{\fc}\to J_X$. A precise description of the image of $J_X$ in  the Weil restriction above is given in Theorem \ref{thm: exact image in JA1}. We note that the Weil restriction $\mathrm{Res}^{\canfa}_\fc(\canA\times \canfa)^{W_X}$ plays an important role in foundational work of Knop \cite{knop1996automorphisms}, where he constructs an action of a related group scheme on the cotangent bundle $T^*X$ for $X$ a $G$-variety, and uses this action to study $G$-automorphisms of $X$. This plays an important role in relative duality, which is explored in \cite{me2}. As an application of Theorem \ref{thm: ses_intro}, we give a direct description of $J$ in Theorem \ref{thm: exact description of J}. 
    
	Finally, we make these results more concrete in the case $X = \GL_{2n}/(\GL_n\times \GL_n)$. This was one of the original forms studied by Schaposnik in her thesis \cite{schaposnik2013spectral}. This example is arithmetically interesting as it is the geometric variety underlying Leslie's Fundamental Lemma for unitary Friedberg-Jacquet periods \cite{spencer_annals}, which is the subject of ongoing work of the first author. The following is an invariant theoretic version of Schaposnik's results on spectral covers.
	
	\begin{thm}{(See Proposition \ref{proposition: U(n,n) spectral cover})}
		Consider the symmetric pair on $\GL_{2n}$ with involution
		\[
		\theta\begin{pmatrix}
			A & B \\ C & D
		\end{pmatrix} = \begin{pmatrix}
			A & -B \\ -C & D
		\end{pmatrix}
		\]
		and $H = \GL_n\times \GL_n$ embedded block diagonally. Then, there is a spectral cover 
		\[
		\fs\to \fc
		\]
		which is a degree $n$, irreducible cover such that
		\[
		J\simeq \mathrm{Res}^\fs_\fc(\bG_m)
		\]
		Let $\ol{C}\to C\times \cA_X^{\reg}$ be the pullpack of the spectral cover $\fs$. With notation as in \ref{thm: general structure}, $\cM_X^\reg$ is a torsor under the action of
		\[
		\cP = \Pic(\ol{C}/\cA^{\reg}).
		\]
	\end{thm}

	\subsection{Acknowledgments}
	
	The second author would like to thank Ng\^{o} Bao Ch\^au for collaboration on the closely related project \cite{morrissey2022reg}, as well as discussions about this project. The first author would like to thank Ana Pe\'on-Nieto and Spencer Leslie for discussions in the early stages of this project, and would like to also thank Ng\^o Bao Ch\^au for his patient mentorship and for pointing him towards this project. In addition, the first author would like to thank Yiannis Sakellaridis for conversations around this work. In particular, the idea of considering the group scheme $J_X$ in Section \ref{sec: relative dual centralizers} came from his suggestion. The authors thank Griffin Wang for suggesting improvements to the paper. The paper also benefited greatly from the careful reading and suggestions of the anonymous referee, whose suggestions greatly simplified the examples in Section \ref{subsections: Examples regular quotient} and helped fix mistakes in Section \ref{section: GS cover}. Both authors thank Paul Levy for answering their many emails and questions about his thesis, and Oscar Garc\'ia-Prada and Zhilin Luo for many helpful discussions.

	\section{Background on Symmetric Pairs}
	\label{section: Background on Symmetric Pairs}
	
	In this section, we review the main results on symmetric pairs that will be used in the study of Hitchin systems appearing in the rest of this paper. While we attribute most results in this section to Paul Levy \cite{levy}, many were first proved by Kostant and Rallis in characteristic zero \cite{KR}.
	
	\subsection{Conditions on Characteristic}
	\label{subsec: conditions on char}
	
	Let $G$ be a reductive group over an algebraically closed field $k$, and $\fg$ its Lie algebra. We assume throughout that $p = 0$ or $p = \mathrm{char}(k)>2$ is good for $G$. Namely, if we let $\Delta$ be a choice of simple roots for the root system $\Phi$ of $G$ and if we express the longest element $\alpha$ of $\Phi$ relative to $\Delta$ as $\alpha = \sum_{\beta\in \Delta} m_\beta \beta$, then $p$ is good for $G$ if and only if $p>m_\beta$ for all $\beta\in \Delta$. In particular, if $p$ is greater than the Coxeter number of $G$, then this hypothesis will be satisfied. We also require throughout that there exists a non-degenerate symmetric bilinear form
	\[
	\kappa\colon \fg\times \fg\to k.
	\]
	Finally we require that there is a finite cover $\tilde{G}\rightarrow G$ such that $\tilde{G}^{\mathrm{der}}$ is simply connected and $\mathrm{Lie}(\tilde{G})\rightarrow \fg$ is an isomorphism.  We note that this is automatic for $G$ almost simple. 
	
	We will refer to the above hypotheses as the \emph{standard hypotheses} throughout.  We refer to \S6.4 of Jantzen \cite{jantzen1998representations} for further discussion of these hypotheses.
	
	\subsection{Involutions, the Restricted Root System, and the Satake Diagram}
	\label{subsec: restricted root systems}
	
	Let $G$ satisfy the standard hypotheses of Section \ref{subsec: conditions on char}. Let $\theta\colon G\to G$ be an algebraic involutive group automorphism. Let $K = (G^\theta)^\circ$ denote the neutral component of the $\theta$ fixed points in $G$. The involution $\theta$ induces a Cartan decomposition 
	\[
	\fg = \fk\oplus \fp,
	\]
	where $\fk$ and $\fp$ denote the $(+1)$ and $(-1)$ eigenspaces of $\theta$, respectively. In particular, $\fk$ is the Lie algebra of $K$.
	
	A complicating fact in the theory of symmetric pairs is that, although all maximal tori of $G$ are conjugate, conjugation does not respect the action of the involution $\theta$ on a $\theta$-stable torus. It is essential, therefore, to specify the action of $\theta$ on a maximal torus of $G$ when studying root systems. We will mainly restrict our attention to $\theta$-stable tori of $G$ for which the $\theta$ fixed subtorus is minimal. We introduce this notion now.
	\begin{defn}
		A $\theta$-split torus $A$ of $G$ is a torus of $G$ such that $\theta(a) = a^{-1}$ for all $a\in A$.
	\end{defn}
	Let $A$ be a fixed $\theta$-split torus of $G$ which is maximal among such tori. The Lie algebra $\fa = {\mathrm Lie}(A)$ is a maximal abelian subalgebra of $\fp$. Such subalgebras play a role analogous to Cartans of $\fg$, and hence we make the definition:
	\begin{definition}
		\label{definition: cartan of p}
		A \emph{$\theta$-Cartan} of $\fp$ is a maximal, abelian subalgebra $\fa$ in $\fp$.
	\end{definition}
	We will frequently fix a maximal $\theta$-split torus $A$ with $\theta$-Cartan $\fa\subset \fp$ in the sequel. This choice is justified by the following standard proposition.
	\begin{prop}
		\label{prop: all cartans are conjugate}
		(\cite[Theorem 2.11]{levy}) All maximal $\theta$-split tori $A$ are $K$-conjugate. Likewise, all $\theta$-Cartans of $\fp$ are $K$-conjugate.
	\end{prop}
	\begin{defn}
		The rank $r_\theta$ of an involution $\theta$ is the dimension
		\[
		r_\theta = \dim(\fa)
		\]
	\end{defn}
	\begin{remark}
		\label{rmk: split definition}
		In general $r_\theta$ is less than than the rank of $G$. When $A$ is a maximal torus of $G$, we call the involution \emph{split}. Every group $G$ has a unique split involution.
	\end{remark}
	
	We will call a maximal torus $T$ of $G$ \emph{maximally $\theta$-split} if $T$ contains a maximal $\theta$-split torus $A$. 
	
	We now introduce a root system associated to the pair $(G,\theta)$. Fix a maximally $\theta$-split torus $T$ containing a maximal $\theta$-split torus $A\subset T$. Let $\Phi$ be the set of roots of $G$ with respect to $T$, viewed as functions on $\ft = \mathrm{Lie}(T)$.
	
	\begin{defn}
		\label{definition: restricted roots}
		The set of \emph{restricted roots} is
		\[
		\Phi_r := \{ \alpha|_{\fa}\in \fa^*\colon \alpha\in \Phi,\; \alpha|_\fa\neq 0 \}.
		\]
	\end{defn}
	It is well-known (see for example \cite[Page 517]{levy}) that $\Phi_r$ forms a (possibly nonreduced) root system. Let $W_\fa$, referred to as the \emph{little Weyl group}, be the Weyl group associated to this root system. We can alternatively describe $W_\fa$ as a quotient.
	\begin{prop}
		\label{proposition : little weyl group is N_G(A)/C_G(A)}
		(\cite[Page 517]{levy}) There is an isomorphism $W_\fa \simeq N_G(\fa)/C_G(\fa)\simeq N_K(\fa)/C_K(\fa)$. In particular, the latter acts as a reflection group on $\fa$.
	\end{prop}
	
	We note that the Weyl group $W_\fa$ associated to $\Phi_r$ is the same as the Weyl group of the reduced root system $\Phi_r^{\mathrm{red}}$.
	
	We conclude this section by introducing the Satake diagram of a group $G$ with involution $\theta$, which is a decorated form of the Dynkin diagram of $G$ classifying the pair $(G,\theta)$. We will make use of these diagrams again in Sections \ref{section: reduction to levis} and \ref{subsections: Examples regular quotient}.
	
	Fix a maximal $\theta$-split torus $A$ and a maximally $\theta$-split torus $T$ containing $A$. The involution $\theta$ acts on the roots $\Phi$ of $G$ with respect to $T$. We say a root $\alpha\in \Phi$ is compact if $\theta(\alpha) = \alpha$.
	
	Now, let $M = C_G(\fa)$ be the Levi subgroup of $G$ obtained from the centralizer of $\fa = \mathrm{Lie}(A)$. Choose a parabolic $P$ of $G$ such that $\theta(P)\cap P = M$; we call such a parabolic a ``minimal $\theta$-split parabolic'' of $G$. Such a choice can be made by \cite[Section 1]{vust}; see also \cite[Prop. 1.3]{leslie}. Choose a Borel $T\subset B\subset P$. It also follows that $B\cap M$ is a Borel of $M$. The Borel $B$ fixes a choice of simple roots $\Sigma\subset \Phi$. Moreover, we have an embedding of Weyl groups $W_M\subset W$, where $W_M$ is the Weyl group of $M$, and in $W_M$, there is a well-defined longest element $w_M$ determined by the Borel $B\cap M$. The involution $\iota = -w_M\theta$ stabilizes the set $\Sigma$ of simple roots with respect to $(T,B)$, and hence induces a diagram automorphism $\iota$ on the Dynkin diagram of $G$.
	
	\begin{definition}
		\label{def: Satake diagram}
		Fix the choices of $(T,B)$ above. The \emph{Satake diagram} of the reductive group $G$ with involution $\theta$ is obtained as follows:  First, we shade the nodes of the Dynkin diagram of $G$ with respect to $(T,B)$ black if the corresponding simple root is compact and white otherwise. Then, we draw edges, denoted in this paper by light gray, thickened lines, connecting two white vertices whenever $\iota$ swaps them.
	\end{definition}
	
	\begin{example}
		The Satake diagram of a split symmetric pair is the corresponding Dynkin diagram, with all vertices shaded white and the involution $\iota$ acting trivially.
	\end{example}
	
	A table of other relevant Satake diagrams can be found in Figure \ref{fig:table of satake diagrams}.

	\subsection{Symmetric Pairs and the GIT Quotient}
	
	We now introduce the notion of \emph{symmetric pair} by considering the additional data of a subgroup $H$.
	
	\begin{definition}
		A \emph{symmetric pair} is the data of a triple $(G,\theta,H)$ where 
		\[
		\theta\colon G\to G
		\]
		is an algebraic involutive group homomorphism on $G$ and $H$ is a smooth, closed subgroup $H\subset G$ such that 
		\[
		K\subset H\subset N_G(K).
		\]
	\end{definition}
	
	The adjoint action of $G$ on $\fg$ restricts to an action of $H$ on the $(-1)$ eigenspace $\fp\subset \fg$.
	
	The following result helps characterize such subgroups $H$. 
	
	\begin{prop}
		\label{prop: extension of K}
		Choose a maximal $\theta$-split torus $A$.
		\begin{enumerate}[(a)]
			\item The normalizer is given explicitly by 
			\[ N_G(K) = \{ g\in G\colon g\theta(g^{-1})\in Z(G) \}. \]
			\item We have $N_G(K) = F^*\cdot K$ where $F^* = \{a\in A\colon a^2\in Z(G)\}$. Note that $F^{*}$ depends on the choice of $A$. Furthermore, $(F^*)^\circ = Z_-$ is the connected component of the subgroup of $Z(G)$ on which $\theta$ acts by inversion, i.e. $\theta(z) = z^{-1}$.
			\item There is a short exact sequence
			\[
			1\to G^\theta\to N_G(K)\xrightarrow{\tau} (F^*)^2\to 1
			\]
			where $\tau(g) = g\theta(g^{-1})$ and $(F^*)^2 = \{a^2\colon a\in F^*\}$.
			\item The group $G^\theta/K = \pi_0(G^\theta)$ is a discrete group.
			\item For any symmetric pair $(G,\theta,H)$, the identity component $H^\circ$ is reductive.
		\end{enumerate}
	\end{prop}
	
	\begin{proof}
		Part (a) follows from the proof of Lemma 1.1 of \cite{sekiguchi}. 
		Part (b) is Lemma 8.1 in \cite{richardson}. Parts (c) and (d) follow immediately from (b). Part (e) is directly from Lemma 8.1 of \cite{richardson}.
	\end{proof}
	
	We have a Chevalley-style result on the GIT quotient $\fp\git H:=\Spec k[\fp]^H$.
	\begin{thm}
		\label{theorem : Chevalley}
		(\cite{levy}, Theorem 4.9 and Corollary 4.10) Fix a $\theta$-Cartan $\fa\subset\fp$. The natural inclusion map $\fa\to \fp$ induces a isomorphisms $\fa\git W_\fa\simeq \fp\git K\simeq \fp\git N_G(K)$. In particular, for any closed $K\subset H\subset N_G(K)$, we have $\fa\git W_\fa\simeq \fp\git H$. Note that, by construction, this map is $\bG_m$-equivariant under the natural actions on both $\fa$ and $\fp$.
	\end{thm}
	
	The invariant theory of this GIT quotient is well studied. We will make use of the following fact.
	\begin{lem}
		\label{lem: exponents}
		(\cite{levy}, Lemma 4.11) We can write $k[\fa]^{W_\fa} = k[f_1,\dots, f_r]$ for $r = \dim(\fa)$ algebraically independent homogeneous polynomials $f_1,\dots, f_r$ of degrees $e_1,\dots, e_r$, respectively, which we will call the \emph{exponents of the root system $\Phi_r$}. (Note that the generators $f_j$ are not canonical but the exponents $e_i$ are.) Moreover, the sum of these exponents can be computed as
		\[ \sum_i e_i = r+\frac{\#\Phi_r^{\mathrm{red}}}{2}. \]
	\end{lem}
	\begin{proof}
		The lemma follows from taking degree of both sides of the equality in \cite[Lemma 4.11]{levy}, noting that the length of the longest element in $W_\fa$ is given by the number of positive roots in the reduced root system.
	\end{proof}

	\subsection{Regularity and the $\theta$-Quasisplit Condition}
	\label{sec: regularity and qsplit}
	
	In \cite{donagi2002gerbe} and \cite{ngo2010lemme}, the notion of regularity was leveraged to study properties of the Hitchin fibration. In so doing, the regular centralizer group scheme was cemented as a central object in the geometry of Hitchin systems and their generalizations. In \cite{gppn}, this point of view was explored, and certain cameral covers were introduced with the goal of understanding regular centralizers. In this section, we review the basic properties fo regularity in the symmetric space setting, and consider a class of special symmetric spaces, called $\theta$-quasisplit, for which the notions of regularity for $G$ and for the symmetric space agree. We will return to this class in Section \ref{section: GS cover}, where we will describe the regular centralizer group schemes for $\theta$-quasisplit symmetric spaces.
	
	\begin{definition}
		\label{definition: regular in p}
		We denote by $I = I_H\subset \fp\times H$ the group scheme of centralizers over $\fp$, i.e.
		\[
		I = \{(x,h)\colon h\cdot x = x\}.
		\]
		An element $x\in \fp$ is called \emph{regular} if $\dim(I_x)$ is the minimal possible.
		Let $\fp^{\reg}\subset \fp$ denote the open subscheme of regular elements in $\fp$.
	\end{definition}
	
	\begin{remark}
		By Proposition \ref{prop: extension of K}, $H$ is an extension of $K$ by a group $F_H\cdot Z_H$ where $F_H$ is finite and $Z_H\subset Z$ is a subgroup of the center. Hence, the notion of regularity only depends on the involution $\theta$ and not the choice of subgroup $H$.
	\end{remark}
	
	\begin{prop}
		\label{prop:dimension criteria for regularity}
		(\cite{levy}, Lemma 4.3) Let $A$ be a maximal $\theta$-split torus of $G$. For any $x\in \fp$, the following are equivalent:
		\begin{enumerate}[(a)]
			\item $x$ is regular;
			\item $\dim C_G(x) = \dim \fa + \dim C_K(A)$;
			\item $\dim C_K(x) = \dim C_K(A)$.
		\end{enumerate}
	\end{prop}

	We call an element $x\in \fp$ \emph{semisimple} if $x$ is semisimple in $\fg$. That this definition is correct is motivated in part by the following two results on Jordan decompositions in $\fp$.
	
	\begin{lem}
		\label{lemma : jordan decomposition}
		(\cite{levy}, Lemma 2.1) Any $x\in \fp$ admits a decomposition $x = s+n$ for $s,n\in \fp$, $s$ semisimple, $n$ nilpotent, and $n\in \mathrm{Lie}\; C_G(s) = \fc_\fg(s)$.
	\end{lem}
	
	As in the Lie algebra $\fg$, the regular, semisimple locus is dense and easy to understand.
	
	\begin{lem}
		\label{lemma: reg,ss locus is dense}
		Let $\fp^{\mathrm{rs}}\subset \fp$ denote the subscheme of regular, semisimple elements in $\fp$.
		\begin{enumerate}[(a)]
			\item Let $x\in \fp$. Then, $x$ is semisimple if and only if $x$ is contained in a $\theta$-Cartan of $\fp$.
			\item The regular, semisimple locus $\fp^{\mathrm{rs}}$ is dense in $\fp$.
			\item If $x\in \fa$ is regular, semisimple, then $\mathfrak{c}_\fg(x) = \mathfrak{c}_\fg(\fa)$.
		\end{enumerate}
	\end{lem}
	\begin{proof}
		(a) and (b) follow immediately from \cite{levy}, Corollary 2.10 and Theorem 2.11. (c) follows from \cite{levy}, Lemma 4.3.
	\end{proof}
	
	\begin{lem}
		\label{lem: K-conjugacy in the reg,ss locus}
		For any regular, semisimple $x\in \fp^{\mathrm{rs}}$, the fiber of the map $\chi\colon \fp\to \fp\git K$ over $\chi(x)$ is equal to the $K$-orbit of $x$, which is closed.
	\end{lem}
	\begin{proof}
		Let $y\in \fp$, and let $y = s+n$ be the Jordan decomposition of Lemma \ref{lemma : jordan decomposition}. Then by Theorem \ref{theorem : Chevalley}, $\chi(y)\in \fp\git K$ is the point corresponding to the closed orbit $H\cdot s$. Hence, if $x$ is regular semisimple, $\chi(y) = \chi(x)$ holds if and only if $x = s$. Since $n$ is in the centralizer $\fc_\fg(s)$, it follows that $n=0$ and $y$ is regular, semisimple. 
		
		It remains to show that if $x,y$ are regular, semisimple with $\chi(x) = \chi(y)$, then $x$ is $K$ conjugate to $y$. By Lemma \ref{lemma: reg,ss locus is dense} and Proposition \ref{prop: all cartans are conjugate}, it follows that any regular, semisimple $x,y\in \fp^{\mathrm{rs}}$ satisfying $\chi(x) = \chi(y)$ are $W_\fa$-conjugate. Since $W_\fa$ has representatives in $K$ by Proposition \ref{proposition : little weyl group is N_G(A)/C_G(A)}, it follows that any such $x$ and $y$ are $K$ -conjugate.
	\end{proof}
	
	\begin{lem}
		\label{lem: unique Cartan}
		If $x\in \fp^{\mathrm{rs}}$, then there exists a unique $\theta$-Cartan $\fa$ such that $x\in \fa$.
	\end{lem}
	\begin{proof}
		By part (a) of Lemma \ref{lemma: reg,ss locus is dense}, there exists a $\theta$-Cartan containing $x$. For any $\theta$-Cartan $\fa$ containing $x$, we have by part (c) of Lemma \ref{lemma: reg,ss locus is dense} that $\fc_\fp(x) = \fc_\fp(\fa)$. But by \cite[Lemma 2.3]{levy}, $\fc_\fp(\fa) = \fa$ for any $\theta$-Cartan $\fa$. We conclude that $\fa = \fc_\fp(x)$ is uniquely determined by $x$.
	\end{proof}
	
	\begin{prop}
		\label{proposition : regularity detected by Jordan decomposition}
		Let $x\in \fp$ have decomposition $x = s+n$ as in Lemma \ref{lemma : jordan decomposition}, and put $L = C_G(s)^0$, $\fl = {\mathrm Lie}(L)$, and $\fp_L = \fl\cap \fp$. Then $L$ is $\theta$-stable and $x$ is regular in $\fp$ if and only if $n$ is regular as an element of $\fp_L$.
	\end{prop}
	
	\begin{proof}
		We follow an identical argument to the proof of Proposition 9.12 in \cite{richardson}. We assume without loss of generality that $s\in \fa$ so that $A\subset L$. By Proposition \ref{prop:dimension criteria for regularity}, $x$ is regular if and only if $\dim C_{K}(x) = \dim C_K(A)$ and $n$ is regular in $L$ if and only if $\dim C_{K\cap L}(n) = \dim C_{K\cap L}(A)$. We have $C_{K}(x)^0 = (C_K(s)\cap C_K(n))^0 = C_{K\cap L}(n)^0$, so that $\dim C_{K}(x) = \dim C_{K\cap L}(n)$. Since $s\in A$, we have $C_K(A)^\circ\subset C_K(s)^\circ = L$, and so $C_{K}(A)^\circ = C_{K\cap L}(A)^\circ$. The result now follows.
	\end{proof}
	
	Let $I^\reg = I|_{\fp^\reg}$ be the restriction of the centralizer group scheme to the regular locus of $\fp$. In contrast with regular centralizers for the adjoint action of $G$, the group scheme $I^\reg$ need not be commutative. A special role in the literature is played by those symmetric pairs for which commutativity holds. This class coincides with the those whose notion of regularity agrees with that of the group $G$ acting on $\fg$. We make the following definition.
	
	\begin{definition}
		\label{def : quasisplit}
		We say a symmetric pair $(G,\theta,H)$ is \emph{$\theta$-quasisplit} if $\fp^\reg\subset \fg^\reg$.
	\end{definition}
	
	\begin{remark}
		The $\theta$-quasisplit condition does not depend on the choice of subgroup $H$, only on the involution $\theta$.
	\end{remark}
	
	As the author found it difficult to locate a proof in the literature, we include here a proof of several equivalent characterizations for the $\theta$-quasisplit condition.
	
	\begin{prop}
		\label{proposition: equivalent characterizations of quasi-split}
		The following are equivalent:
		\begin{enumerate}[(a)]
			\item $(G,\theta,H)$ is $\theta$-quasisplit;
			\item $C_G(A) = T$ is a maximal (maximally $\theta$-split) torus;
			\item $I^\reg$ is a commutative group scheme;
		\end{enumerate}
	\end{prop}
	\begin{proof}
		Let $x\in \fp$. By part (b) of Proposition \ref{prop:dimension criteria for regularity}, $x$ is regular if and only if
		\[
		\dim C_G(x) = \dim \fa+\dim C_K(\fa).
		\]
		Therefore, we see that $x\in \fp^\reg$ is also regular in $G$ if and only if $\dim C_K(\fa)\cdot A = \dim T$. By Proposition \ref{prop:dimension criteria for regularity}, $C_G(\fa) = C_K(\fa)\cdot A$ and so this latter condition is equivalent to $C_G(\fa) = T$. We note that in this case $C_G(\fa) = C_G(A)$ by \cite[Lemma 2.4]{levy}, so we conclude that (a) and (b) are equivalent.
		
		Assume (a) holds. Then if $x\in \fp^{\reg}$, $I^\reg_x$ is contained in the regular centralizer group scheme of $x$ in $G$, which is abelian. Hence, (c) holds.
		
		Conversely, if (c) holds, then by Lemma \ref{lemma: reg,ss locus is dense}, there exists $x\in \fp^{\reg}$ such that $\fc_\fg(A) = \fc_\fg(x)$. Then $\dim \fc_\fg(A) = \dim \fc_\fg(x) = r$ and by Lemma 4.2 of \cite{levy}, we conclude that $C_G(A)$ is a maximal torus. 
	\end{proof}

	\begin{prop}
		(\cite{leslie}, Lemma 1.6) For $\theta$ quasisplit, the little Weyl group $W_\fa$ is naturally a subgroup $W_\fa\subset W$.
	\end{prop}
	
	Let $T$ be a maximally $\theta$-split torus, $\Phi$ the root system of $G$ with respect to $T$, and $\Phi_r$ the restricted root system. We denote by
	\[
	r\colon \Phi\to \Phi_r\cup \{0\}
	\]
	the restriction map sending $\alpha\mapsto \alpha|_\fa$. Note that roots of $G$ may, a priori, restrict to zero on $\fa\subset \ft$. For $\theta$-quasisplit symmetric pairs, this does not happen.
	
	\begin{lem}
		\label{lemma: no roots restrict to zero}
		For $(G,\theta,H)$ quasisplit, the set $r^{-1}(0)$ is empty; that is, no root in $\Phi$ restricts to zero on $\fa$.
	\end{lem}
	\begin{proof}
		Fix a maximally $\theta$-split Cartan $\ft = \ft_0\oplus \fa$ where $\ft_0\subset\fk$ is the $(+1)$ eigenspace of $\theta$ on $\ft$. Suppose that $\alpha\in \Phi$ restricts to zero on $\fa$. Let $(x,y)\in \fk\oplus \fp = \fg$ be an eigenvector with eigencharacter $\alpha$. Then, using the compatibility of the bracket on $\fg$ with the Cartan decomposition, we have, for all $t\in \ft_0$ and $a\in \fa$,
		\begin{equation}
			\label{eqn:resroot}
			\alpha(t)(x,y) =\alpha(t+a)(x,y)= \big({\mathrm ad}(t)(x)+{\mathrm ad}(a)(y), {\mathrm ad}(a)(x)+{\mathrm ad}(t)(y)  \big).
		\end{equation}
		In particular, ${\mathrm ad}(a)(y)$ is independent of $a$, so $y\in \fc_{\fp}(\fa)$. But since the symmetric pair is $\theta$-quasisplit, by \cite[Lemma 2.3]{levy} we have $\fc_\fp(\fa) = \fa\subset \fc_\fp(\ft_0)$. Hence, ${\mathrm ad}(t)(y) = 0$ and equation \ref{eqn:resroot} implies that
		\[
		{\mathrm ad}(a)(x) = \alpha(t)y
		\]
		for all $t\in \ft_0$ and $a\in \fa$. This can only be true if both sides of the expression are uniformly zero, so $\alpha = 0$ is not in $\Phi$, a contradiction.
	\end{proof}
	
	Fix a Cartan $\ft\supseteq \fa$ of $\fg$ extending a $\theta$-Cartan of $\fa$. In general (not assuming quasisplit), there is a map 
	\[
	\fa\git W_\fa\to \ft\git W
	\]
	produced by the map of GIT quotients $\fp\git K\to \fg\git G$, which is a finite, generically unramified map. This map is unramified precisely in the $\theta$-quasisplit case:
	\begin{lem}
		\label{lemma: qspt implies unramified map of hitchin bases}
		(See \cite{panyu}, Theorem 3.6 for the characteristic zero result.) If the symmetric pair is $\theta$-quasisplit, then the map $\fa\git W_\fa\to \ft\git W$ is unramified, and in fact, is a closed embedding.
	\end{lem}
	\begin{proof}
		Let $A\subset T$ be a maximal $\theta$-split torus and maximal torus of $G$, respectively, with Lie algebras $\fa\subset \ft$. By Zariski's Main Theorem, it suffices to prove that the morphism $\fa\git W_\fa\to \ft\git W$ is generically one to one. Let $\fa^\reg$ and $\ft^\reg$ denote the regular locus of $\fa$ and $\ft$, respectively. Since the symmetric pair is quasisplit, $\fa^\reg\subset \ft^\reg$. Suppose that $x,y\in \fa^\reg$ have the same image under the composition $\fa^\reg\to \fa\git W_\fa\to \ft\git W$. Then, there exists some $g_0\in N_G(\ft)$ such that $g_0\cdot x = y$, and the set of all $g\in G$ conjugating $x$ to $y$ is given explicitly by the set $g_0 T$.
		
		Applying $\theta$ to the equation $g_0\cdot x = y$, we see that $\theta(g_0)$ also takes $x$ to $y$. Hence, $\theta(g_0)^{-1}g_0$ is in the centralizer of $x$ in $G$, and so lies in $T = C_G(x)$. We have $\theta(g_0) = g_0 a_0$ for some $a_0\in T$. Since $\theta^2 = 1$, we conclude that $\theta(a_0) = a_0^{-1}$ and so $a_0\in A$. 
		
		We now claim that there exists $t\in A$ such that $g_0 t\in G^\theta$. Indeed, let $t\in A$ be such that $t^2 = a_0$. Then
		\[
		\theta(g_0 t) = g_0 a_0 t^{-1} = (g_0 t) (a_0t^{-2}) = g_0t.
		\]
		As $A$ centralizes $x$, we conclude that $(g_0t)\cdot x = y$ are conjugate under $G^\theta$, and since they are regular semisimple, are also conjugate under $K$ by Lemma \ref{lem: K-conjugacy in the reg,ss locus}. Let $h\in K$ be such that $h\cdot x = y$. Then also $h\cdot \fa = \fa$ since by Lemma \ref{lem: unique Cartan} any regular semisimple element of $\fp$ is contained in a unique $\theta$-Cartan. We conclude that $h\in N_K(\fa)$, and so the lemma follows.
	\end{proof}
	
	We record here an identity that will be important for dimension counts later.
	
	\begin{lem}
		(\cite{richardson}, Lemmas 3.1 and 3.2) We have the identity
		\[
		\dim \fk - \dim \fp = \dim C_K(\fa) - \dim \fa.
		\]
		In particular, if the symmetric pair is $\theta$-quasisplit, then
		\[\dim \fk - \dim \fp = r - 2r_\theta \]
		where $r$ is the rank of the group $G$ and $r_\theta = \dim \fa$ is the rank of the involution.
	\end{lem}
	
	\subsection{Examples}
	
	We state in this section some examples of symmetric pairs. We will illustrate the main results of this work using the examples from this list in Section \ref{subsections: Examples regular quotient}.
	
	\begin{ex}
		\label{example: complex case}
		\textbf{The Diagonal Case.} Let $G_1$ be a reductive group, and consider $G = G_1\times G_1$ with the swapping involution
		\[
		\theta(g,h) = (h,g)
		\]
		Put, $H=K = G_1$ the diagonal copy of $G_1$ in $G_1\times G_1$. The Cartan decomposition is given by
		\[
		\fg = \{(x,x)\colon x\in \fg_1\}\oplus \{(x,-x)\colon x\in \fg_1\}.
		\]
		Then, the action of $H = G_1$ on $\fp\simeq \fg_1$ is simply the adjoint representation of $G_1$, and the restricted root system of $(G_1\times G_1,G_1)$ is given by the root system for $G_1$.
	\end{ex}
	
	\begin{ex}
		\label{example:U(n,n) case}
		\textbf{The Case $\GL_n\times\GL_n\subset \GL_{2n}$.} Let $G = \GL_{2n}$ and consider the involution
		\[
		\theta(x) = I_{n,n}xI_{n,n}\quad \text{where }I_{n,n} = \begin{pmatrix}
			I_n & 0 \\ 0 & -I_n
		\end{pmatrix}
		\]
		Let $H=K = G^\theta = \GL_n\times \GL_n\subset \GL_{2n}$ embedded block diagonally. The Cartan decomposition is
		\[
		\fg=\fk\oplus \fp=\left\{\left(\begin{matrix}A & 0\\ 0 & B\end{matrix}\right)\right\}\oplus \left\{\left(\begin{matrix}0 & C\\ D & 0\end{matrix}\right)\right\}.
		\]
		A $\theta$-Cartan in $\fp$ is rank $n$, given by
		\[
		\fa = \left\{\begin{pmatrix}
			0 & \delta \\
			\delta & 0
		\end{pmatrix} \colon \delta\text{ is diagonal} \right\}
		\]
		The restricted root system is the type $C_n$ root system\footnote{Note that various different normalizations of this root system are used in applications. For example, the normalization used by Sakellaridis and Venkatesh scales the long root to make this a type $B_n$ root system \cite{sakellaridis2017periods}.} computed explicitly as
		\[\Phi_r=\{\pm (\delta_{j}^*\pm\delta_{k}^*)\colon j\neq k\text{ and } 1\leq j,k\leq n\}\cup \{\pm 2\delta_j^*\colon 1\leq j\leq n\}\]
		where $\delta_j^*$ denotes the dual basis element to the $j$-th coordinate of $\delta$ in $\fa$.
	\end{ex}
	
	\begin{ex}
		\label{example: so x so}
		\textbf{The Case $\SO_n\times \SO_n\subset \SO_{2n}$.}  We construct the split symmetric pair of type $D_{2n}$. Let $G = \SO_{2n}$ with the involution
		\[
		\theta(x) = I_{n,n}xI_{n,n}\quad \text{where }I_{n,n} = \begin{pmatrix}
			I_n & 0 \\ 0 & -I_n
		\end{pmatrix}
		\]
		(compare with Example \ref{example:U(n,n) case}). Let $H = K = \SO_n\times \SO_n\subset \SO_{2n}$ embedded block diagonally. Note that this is an index 2 subgroup in $G^\theta = S(O_n\times O_n)$. The Cartan decomposition is
		\[
		\fg = \fk\oplus \fp = \left\{ \begin{pmatrix}
			A & 0 \\ 0 & B
		\end{pmatrix}\colon A,B\in \mathfrak{so}_n \right\} \oplus \left\{ \begin{pmatrix}
			0 & C \\ -C^t & 0
		\end{pmatrix} \right\}
		\]
		We fix a Cartan in $\fp$
		\[
		\fa = \left\{ \begin{pmatrix}
			0 & \delta \\ -\delta & 0
		\end{pmatrix} \colon \delta\text{ is diagonal} \right\}
		\]
		The restricted root system is the type $D_n$ root system
		\[
		\Phi_r = \{ i(\pm \delta_j^*\pm \delta_k^*)\colon  1\leq j<k\leq n \}
		\]
		where $\delta_j^*$ denotes the $j$-th coordinate function of $\delta$ in $\fa$ and $i$ is the imaginary unit.
	\end{ex}

	\begin{ex}
		\label{ex: so x so odd}
		\textbf{The Case $\SO_{2m}\times \SO_{2n-2m+1}\subset \SO_{2n+1}$, $2m\leq n$.} Let $G = \SO_{2n+1}$ and consider the involution
		\[
		\theta(g) = I_{2m,2n-2m+1}gI_{2m,2n-2m+1},\quad \text{where }I_{2m,2n-2m+1} = \begin{pmatrix}
			I_{2m} & \\
			& -I_{2n-2m+1}
		\end{pmatrix}
		\]
		Let $H=K = \SO_{2m}\times \SO_{2n-2m+1}$. Fix the $\theta$-Cartan
		\[
		\fa = \left\{ \left( \begin{array}{c|ccc} & \mathbf{0}_{m\times (n-m)}  & \delta & \mathbf{0}_{m\times (n-m+1)} \\
			\hline
			\mathbf{0}_{(n-m)\times m} & & & \\
			-\delta & & & \\
			\mathbf{0}_{(n-m+1)\times m} & & & 
		\end{array} \right) \colon \delta \text{ is diagonal }2m\times 2m\right\}
		\]
		The restricted root system is the type $B_m$ root system
		\[
		\Phi_r = \{i(\pm\delta_j^*\pm\delta_k^*)\colon 1\leq j<k\leq m\}\cup \{\pm i\delta_j^*\colon 1\leq j\leq m\}.
		\]
		where again $\delta_j^*$ denotes the $j$-th coordinate function on the diagonal matrix $\delta$ appearing in $\fa$, and $i$ is the imaginary unit.
	\end{ex}

	\subsection{Nilpotent Orbits}
	
	In sharp contrast to the case of $G$ acting on $\fg^\reg$, a symmetric pair may have several distinct $H$ orbits of regular, nilpotent elements. In fact, this will to a large extent govern the geometry of the Hitchin fibration for symmetric pairs. In this section, we review results of \cite{KR}, \cite{sekiguchi}, and \cite{levy} on regular nilpotent $K$-orbits.
	
	Recall that $K$ acts on the nilpotent cone $\cN_\fp = \cN\cap \fp$ of $\fp$. Kostant and Rallis in \cite{KR} showed that in characteristic zero, although the nilpotent cone is not necessarily irreducible, it has finite many components, each of which contains a unique open orbit for the action of $K$. Levy extended this result to positive characteristic.
	\begin{thm}
		\label{theorem : structure of nilpotent cone}
		(\cite{levy}, Theorem 5.1) Each irreducible component of $\cN_\fp$ contains a unique regular $K$-orbit as an open, dense subset. In particular, $\cN_\fp$ is equidimensional, and the irreducible components of $\cN_\fp$ are in 1-1 correspondence with connected components of $\cN_\fp^{\reg}$.
	\end{thm}
	\begin{cor}
		\label{lemma: non regulars codimension geq 1 in cNp}
		The space $\cN_{\fp}\backslash \cN_{\fp}^{\reg}$ is of codimension $\geq 1$ in $\cN_{\fp}$.
	\end{cor}
	
	The number of $K$-conjugacy classes of regular nilpotents was studied and classified by Sekiguchi over $\C$ \cite{sekiguchi} and by Levy in positive characteristic \cite{levy}. To state the result, we make the following definition. 
	\begin{defn}
		An \emph{isogeny of symmetric pairs} 
		\[
		(G',\theta',H')\to (G,\theta,H)
		\]
		is an isogeny $G'\to G$ restricting to an isogeny $H'\to H$ such that the following diagram commutes
		\[
		\xymatrix{
			G'\ar[r]^-{\theta'}\ar[d] & G'\ar[d] \\
			G\ar[r]^-{\theta} & G
		}
		\]
		
		We say that two symmetric pairs $(G,\theta,H)$ and $(G',\theta',H')$ are isogeneous if there exists an isogeny of symmetric pairs between them.
	\end{defn}
	
	We will make frequent use of the following classification result in computations.
	
	\begin{prop}
		\label{prop:table of nilpotent orbits}
		(\cite{levy}, Proposition 6.21) Let $G$ be a simple group, and $\theta$ an involution on $G$. The number of regular nilpotent $K$ orbits (and hence the number of irreducible components of the nilpotent cone) is exactly two if and only if $(G,\theta,H=K)$ is isogenous to one on the following list (listed as pairs $(G,K)$ with involution implied):
		
		\begin{center}
			\begin{tabular}{| l |}
				\hline
				$(\SL_{2n}, \SO_{2n})$; \\
				$(\SL_{2n}, S(\GL_n\times \GL_n))$; \\
				$(\SO_{2n+1},\SO_{2m}\times \SO_{2(n-m)+1})$, $2m<2(n-m)+1$; \\
				$(\Sp_{2n},\GL_n)$; \\
				$(\SO_{2n},\SO_{2m}\times \SO_{2(n-m)})$, $m\neq n/2$; \\
				$(\SO_{4n},\GL_{2n})$; \\
				$(\SO_{4n+2},\SO_{2n+1}\times \SO_{2n+1})$; \\
				$(G,\SL_8)$, for $G$ simple of type E$_7$; \\
				$(G,G'\times \bG_m)$, for $G$ simple of type E$_7$ and $G'$ simple of type E$_6$;\\
				\hline
			\end{tabular}
		\end{center}
		
		In addition, the split symmetric pair $(\SO_{4n},\SO_{2n}\times \SO_{2n})$ has exactly 4 regular nilpotent orbits. All other symmetric pairs with $G$ simple and $H=K$ have irreducible nilpotent cone in $\fp$, and hence a single regular nilpotent orbit.
	\end{prop}
	
	\begin{remark}
		Among the above involutions, only the following are $\theta$-quasisplit:
		
		\begin{center}
			\begin{tabular}{| l |}
				\hline
				$(\SL_{2n}, \SO_{2n})$; \\
				$(\SL_{2n}, S(\GL_n\times \GL_n))$; \\
				$(\SO_{2n+1},\SO_{n}\times \SO_{n+1})$; \\
				$(G,\SL_8)$, for $G$ simple of type E$_7$; \\
				$(\SO_{4n},\SO_{2n}\times \SO_{2n})$ (which has 4, not 2, nilpotent orbits) \\
				\hline
			\end{tabular}
		\end{center}
	\end{remark}
	
	\begin{remark}
		\label{remark: H orbits of regular nilpotents}
		Note that the center acts trivially on $\cN_{\fp}^{\reg}$. Hence, for any symmetric pair $(G,\theta,H)$, Proposition \ref{prop: extension of K} implies that the $H$-orbits on $\cN_\fp^{\reg}$ are given by $(\cN_\fp^{\reg}/K)/\pi_0(H)$.
	\end{remark}
	
	If one sets $H=N_G(K)$, then the classification of regular, nilpotent orbits becomes trivial.
	\begin{thm}
		(\cite{KR}, Proposition 4; \cite{levy}, Theorem 5.16) \label{theorem: N_G(K) acts transitively on regular nilpotents}
		The normalizer group $N_G(K)$ acts transitively on the set of regular nilpotents $\cN_\fp^\reg$.
	\end{thm} 
	In particular, $\pi_0(N_G(K))$ acts transitively on $\cN_\fp^{\reg}/K$.
	
	The above classifies nilpotent orbits for involutions on simple $G$. For our purposes later, we will also need the classification of nilpotent orbits for the diagonal case, which is trivial.
	
	\begin{lem}
		\label{lemma: nilpotents in diagonal case}
		The diagonal case $G_1\subset G_1\times G_1$ of Example \ref{example: complex case} has a single nilpotent $K$ orbit on $\cN_\fp^{\reg}$.
	\end{lem}
	\begin{proof}
		This follows immediately from the isomorphism of stacks $[\fp/K]\simeq [\fg_1/G_1]$ given by projecting onto the first variable.
	\end{proof}

	\subsection{Generalities on Kostant-Rallis Sections}
	\label{subsec: generalities on Kostant Rallis sections}
	
	In this subsection, we review the theory of Kostant-Rallis sections, as introduced in \cite{KR} and generalized in \cite{levy}. We work in the generality of \cite{levy}; in particular, in this section, it is essential that the standard hypotheses of Section \ref{subsec: conditions on char} hold for $G$. 
	
	\begin{lem}
		\label{lem: existence of sections}
		(\cite{levy}, Corollary 6.29) Let $e\in \cN_\fp^{\reg}$ be a regular nilpotent. Then there exists a slice $e+\mathfrak{v}\subset \fp^{\reg}$ contained in the regular locus of $\fp$ such that the map
		\[ e+\mathfrak{v}\to \fp\git K \]
		is an isomorphism whose fiber over $0\in \fp\git K$ is $e$. Here, $\fv\subset \fp$ is a linear space of dimension $\dim\fv = \dim \fa$ which is orthogonal to the linear subspace $[e,\fh]\subset \fp$.
	\end{lem}
	
	\begin{remark}
		In characteristic zero, the result above as proved by Kostant and Rallis using certain $\fsl_2$ triples compatible with the involution $\theta$ \cite{KR}. In positive characteristic, Paul Levy used the theory of associated characters to replace $\fsl_2$ triples \cite{levy}. As this paper will only rely on the existence of sections, we will not recall either of these constructions here.
	\end{remark}
	
	Fix a Kostant-Rallis section $\cS = e+\fv$. For later applications to smoothness, we will need a bit more on the differential of the action map
	\begin{equation}
		\label{eqn: action on KR section}
		H\times \mathcal{S}\to \fp^\reg.
	\end{equation}
	We record here the following Lemma.
	\begin{lem}
		\label{lem: differential at zero is surjective}
		The differential of the action map \eqref{eqn: action on KR section} at $(1,e)$ is surjective.
	\end{lem}
	\begin{proof}
		The differential of the above map is identified with the map
		\[
		\fh\oplus \fv\to \fp,\quad (x,y)\mapsto [e,x]+y
		\]
		Since $e$ is regular, the codimension of $[e,\fh]$ in $\fp$ is equal to the dimension of a $\theta$-Cartan, which by Lemma \ref{lem: existence of sections} is exactly the dimension of $\fv$. Moreover, by the construction of $\fv$, it is orthogonal to $[e,\fh]$. Hence by a dimension count, we have $\fp = [e,\fh]+\fv$, and we conclude that the map above is surjective.
	\end{proof}

	We will study the map $\fp\to \fa\git W_\fa$ produced by Theorem \ref{theorem : Chevalley} in some detail; it will provide the underlying structure of the Hitchin fibration for symmetric pairs. In this spirit, we now prove this map's flatness.
	\begin{lem}
		\label{lemma: p-->a//Wa is flat}
		The map $\chi\colon \fp\to \fp\git K\simeq \fa\git W_{\fa}$ is flat, as is the map $\fp^{\reg}\rightarrow \fp\git K$.
	\end{lem}
	\begin{proof}
		For any $y\in \fp\git K$, let $\fp_y$ denote the fiber of $\chi$ at $y$. Consider the function 
		\[
		\psi(x)=\dim_x(\fp_{\chi(x)}),
		\]
		where $\dim_x$ denotes the local dimension at $x$. This function is upper semi-continuous on $\fp$, and by $\bG_m$ equivariance of $\chi$, is constant on $\bG_m$-orbits of $\fp$. 
		
		By Lemma \ref{theorem : structure of nilpotent cone}, for any $x\in \fp_0 = \cN_\fp$, the local dimension of the fiber at $x$ is given by the dimension of a regular $K$-orbit. Moreover, for any $x\in \fp^{\mathrm{rs}}$ which is regular semisimple, Lemma \ref{lem: K-conjugacy in the reg,ss locus} states that $\fp_{\chi(x)}$ is the unique closed $K$-orbit through $x$, which is also regular. Hence, $\psi(x)$ is also equal to the dimension of a regular $K$-orbit for the open subset of $x\in \fp^{\mathrm{rs}}$. We conclude by upper semicontinuity that $\psi$ is constant. 
		
		Since $\chi$ is a morphism between smooth schemes whose fibers are equidimensional, miracle flatness implies that $\psi$ is flat.
	\end{proof}

	\section{The Regular Quotient}
	\label{section: the regular quotient}
	
	\subsection{Motivation and Generalities}
	\label{subsec: general theory on regular quotient}
	
	Many facts about the (usual) Hitchin fibration can be abstracted to basic properties in invariant theory. Principally among these is that the morphism
	\[
	[\fg^\reg/G]\to \fg\git G
	\]
	is a gerbe banded by the regular centralizer group scheme. For example, this property is the invariant theoretic shadow of why generic fibers of the Hitchin fibration are Picard stacks. However, this is plainly not the case for symmetric pairs, as observed in \cite[Section 4.2]{gppn}. We illustrate with an example.
	
	\begin{ex}
		\label{ex: not a gerbe}
		Let $G = \SL_2$ with involution conjugating by the matrix $\mathrm{diag}(1,-1)$. This is the $n=1$ case of Example \ref{example:U(n,n) case}. Then, $H = S(\bG_m\times\bG_m)\simeq \bG_m$ acts on $\fp\simeq \bA^2$ by the hyperbolic action $x\cdot (a,b) = (xa,x^{-1}b)$. The regular locus is $\fp^{\reg} = \bA^2\setminus \{0\}$, and we see that there are two regular orbits lying over the closed orbit $0\in \fp\git H$, see Figure 1. 
		
		\begin{figure}
			\begin{center}
				\begin{tikzpicture}
					\begin{axis}[
						axis line style={draw=none},
						axis equal,     
						xmin=-5, xmax=5, 
						ymin=-5, ymax=5, 
						ticks = none,
						samples=101,
						domain=-5:5]
						\addplot[domain=.2:5]{1/x};
						\addplot[domain=-5:-.2]{1/x};
						\addplot[domain=.2:5]{-1/x};
						\addplot[domain=-5:-.2]{-1/x};
						\addplot[domain=1:5]{5/x};
						\addplot[domain=-5:-1]{5/x};
						\addplot[domain=1:5]{-5/x};
						\addplot[domain=-5:-1]{-5/x};
						\addplot[domain=.6:5]{3/x};
						\addplot[domain=-5:-.6]{3/x};
						\addplot[domain=.6:5]{-3/x};
						\addplot[domain=-5:-.6]{-3/x};
						\addplot[color=red]{0};
						\draw[color=blue] (0,-5) -- (0,5);
						\filldraw[white] (0,0) circle (2pt);
						\draw (0,0) circle (2pt);
					\end{axis}
				\end{tikzpicture}
				\begin{tikzpicture}
					\draw (-3,2) -- (3,2);
					\draw (-4,2) node {$\fp^\reg\myfatslash H = $};
					\filldraw[white] (0,2) circle (2pt);
					\filldraw[blue] (0,2.1) circle (1pt);
					\filldraw[red] (0,1.9) circle (1pt);
					\draw (-3,0) -- (3,0);
					\draw (-4,0) node {$\fp\git H = $};
					\draw[->] (-4.1,1.5) -- (-4.1,.5);
					\draw[->] (0,1.5) -- (0,.5);
				\end{tikzpicture}
				\caption{\footnotesize (Left) The orbits of $H = S(\bG_m\times \bG_m)$ acting on $\fp^\reg\simeq \bA^2\setminus \{0\}$ for the symmetric space $X = \SL_2/S(\bG_m\times \bG_m)$. Note the two orbits, drawn in blue and red, whose closure includes the (non-regular) closed orbit $\{0\}$. The regular quotient for this symmetric pair (pictured right) is the affine line with doubled origin.}
			\end{center}
		\end{figure}
	\end{ex}
	
	To solve this problem, Pe\'on-Nieto and Garc\'ia-Prada proposed a rigidification $\fp^\reg\myfatslash H$ of the stack $[\fp^\reg/H]$ such that there is a factorization
	\[
	[\fp^\reg/H]\to \fp^\reg\myfatslash H\to \fp\git H
	\]
	with the first map being a gerbe and the second a (non-separated) cover \cite[Section 4.2]{gppn}. To illustrate, in Example \ref{ex: not a gerbe}, $\fp\git H\simeq \bA^1$ is an affine line while $\fp^\reg\myfatslash H$ is an affine line with doubled origin.
	
	In unpublished work \cite{morrissey2022reg} of Ng\^{o} and Morrissey, such quotients are introduced in the greater generality of a reductive group $G$ acting on an affine normal scheme $M$. Examples of these generalized Hitchin systems include the multiplicative Hitchin system studied in \cite{griffin}, Hitchin systems for higher dimensional varieties studied in \cite{chen-ngo}, and Hitchin systems associated to spherical varieties studied in \cite{me2}. While it will turn out that $\fp^\reg\myfatslash H$ will be a scheme, the resulting quotients $M^\reg\myfatslash G$ are, in general, Deligne-Mumford stacks.
	
	We will not need this generality in this paper. Instead, we will develop the theory in the following more limited setting:  Let $V$ be a vector space and $H$ a reductive group acting on $V$. Let $I_V\subset V\times H$ be the group scheme over $V$ of stabilizers of the $H$ action. We have the usual definition of regularity in $V$; namely, $x\in V$ is regular when the dimension $\dim(I_{V,x})$ is minimal.
	\begin{assumption}
		\label{asp: H action on V}
		We assume the action of $H$ on $V$ satisfies:
		\begin{itemize}
			\item The regular centralizer group scheme $I_V^\reg := I_V|_{V^\reg}$ is smooth over $V^\reg$.
			\item The map $V^\reg\to V\git H$ is flat.
		\end{itemize} 
	\end{assumption}
	
	These hypotheses will be satisfied for the representation of $H$ acting on $\fp$ by Proposition \ref{prop: smoothness of Ireg} and Lemma \ref{lemma: p-->a//Wa is flat}. We will use the following result.
	
	\begin{thm}
		\label{thm:aov}
		(\cite[Theorem A.1]{AOV}) Let $M$ be a locally finitely presented algebraic stack over a scheme $S$ with inertia stack $\mathcal{I}M$. Let $I\subset \mathcal{I}M$ be a flat, finitely presented subgroup stack. Then, there exists an algebraic stack $M\myfatslash I$ and morphism $\rho\colon M\to M\myfatslash I$ such that:
		\begin{enumerate}
			\item $\rho$ makes $M$ into an fppf gerbe over $M\myfatslash I$;
			\item For each object $\xi\in M(T)$, the map
			\[
			\ul{\Aut}(\xi)\to \ul{\Aut}(\rho(\xi))
			\]
			is surjective with kernel $I_\xi$.
		\end{enumerate}
	\end{thm}
	
	As a corollary of this result, we obtain the following.
	
	\begin{cor}
		\label{cor:existence of regular quot V/H}
		Let $V$ be a vector space with a linear action of a reductive group $H$ satisfying Assumption \ref{asp: H action on V}. Then, there exists a unique scheme $V^\reg\myfatslash H := [V^\reg/H]\myfatslash I_V^\reg$ and morphism $V^\reg/H\to V^\reg\myfatslash H$ characterized by the following properties:
		\begin{enumerate}
			\item The stack $[V^\reg/H]$ is an fppf gerbe over $V^\reg\myfatslash H$.
			\item The natural map $[V^\reg/H]\to V\git H$ factors uniquely through $V^\reg\myfatslash H$, and the map $V^\reg\myfatslash H\to V\git H$ is quasi-finite.
		\end{enumerate}
	\end{cor}
	\begin{proof}
		Note that $I_V^\reg$ is an $H$-equivariant group scheme on $V^\reg$, and hence descends to a group scheme over the stack quotient $[V^\reg/H]$. In fact, this descent is exactly given by the inertia stack $\mathcal{I}[V^\reg/H]$. As $I_V^\reg$ is assumed to be flat, and is certainly finitely presented, it follows that the quotient $[V^\reg/H]\myfatslash I_V^\reg$ exists as an algebraic stack, which we abbreviate to $V^\reg\myfatslash H$.
		
		By property (2) of Theorem \ref{thm:aov}, the group $\ul{\Aut}(\xi)$ is trivial for objects $\xi\in (V^\reg\myfatslash H)(T)$. Hence, $V^\reg\myfatslash H$ is an algebraic space. There is an obvious map
		\begin{equation}
			\label{eqn: map regular quot to git}
			V^\reg\myfatslash H\to V\git H
		\end{equation}
		through which $[V^\reg/H]\to V\git H$ factors. On geometric points, this map takes a regular orbit in $V^\reg\myfatslash H$ to the unique closed $H$-orbit in its closure in $V$. For any closed $H$-orbit $\cO$ in $V$, there are finitely many regular $H$-orbits containing $\cO$ in their closures since regular orbits are open in the fibers of the map $V^\reg\to V\git H$. Hence, \eqref{eqn: map regular quot to git} is quasi-finite. Since $V\myfatslash H$ admits a quasi-finite map to a scheme, it too must be a scheme.
	\end{proof}
	
	\begin{remark}
		The above result was used in \cite[Lemma 4.10]{gppn} to introduce the space $\fp^\reg\myfatslash H$.
	\end{remark}
	
	\begin{remark}
		Note the existence result above does not imply surjectivity. Indeed, examples are known for which $V^\reg\myfatslash H\to V\git H$ is not a surjection. For example, consider the action of $H = \GL_2$ on pairs of matrices
		\[
		V = \mathfrak{gl}_2\times \mathfrak{gl}_2
		\]
		with the action by diagonal adjoint action. The map $V^\reg\myfatslash H\to V\git H$ is the open embedding of $\bA^2\setminus \{0\}$ into $\bA^2$.
		
		In the case of $H$ acting on $\fp$ from a symmetric pair, the existence of Kostant-Rallis sections will imply surjectivity.
	\end{remark}
	
	We collect some corollaries of Corollary \ref{cor:existence of regular quot V/H}.
	
	\begin{prop}[Properties of Regular Quotient]
		\label{propostion: properties regular quotients}
		The regular quotient has the following properties:
		\begin{itemize}
			\item If $I_V^\reg$ is abelian then it descends to a group scheme $J\rightarrow V^{\reg}\myfatslash H$.
			\item In general, $I_V^\reg$ descends to a band in the sense of Giraud \cite{giraud1971cohomologie} $J_{band}\rightarrow V^{\reg}\myfatslash H$.
			\item The map $[V^{\reg}/H]\rightarrow V^{\reg}\myfatslash H$ is a gerbe banded by $J_{band}$; when $I_M^\reg$ is abelian it is a $J$-gerbe.
			\item If there are compatible $\bG_m$ actions on $I^\reg$ and $V^\reg$, then there is a canonical $\bG_m$ action on $V^\reg\myfatslash H$ and the morphisms $[V^\reg/H]\to V^\reg\myfatslash H$ is $\bG_m$ equivariant.
		\end{itemize}    
	\end{prop}
	
	The first three properties follow directly from \cite[Remark A.2]{AOV}, while the fourth can be easily deduced.
	
	We will, of course, be interested in the special case of $H$ acting on $V = \fp$ coming from a symmetric pair $(G,\theta, H)$.

	\subsection{Quotients in the fppf topology and representability}
	
	In latter sections, we will frequently consider quotients of non-constant group schemes over the GIT base $\fc$. In this section, we will prove an assortment of general representability results to justify these quotients. We begin with a result which we will apply to certain quotients by $I_H^\reg$.
	
	\begin{prop}
		\label{prop: I/J is representable}
		Let $I$ and $J$ be flat group schemes of finite type over a base scheme $S$ equipped with a closed embedding $J\to I$ over $S$. Let $I/J$ be the fppf quotient sheaf of the map of fppf sheaves $J\to I$. Then, $I/J$ is representable by a separated, flat algebraic space over $S$.
	\end{prop}
	\begin{proof}
		The fppf quotient $I/J$ is a sheaf valued in groups. Moreover, it is representable by an algebraic space as the action of $J$ on $I$ is free. We have Cartesian diagrams in algebraic spaces
		\[
		\xymatrix{
			I\times_SJ\ar[r]^-{\mathrm{act}}\ar[d]_-{\pi_1} & I \ar[d] \\
			I \ar[r] & I/J
		}\quad \quad \xymatrix{
			I\times_S J\ar[d]\ar[r] & I\times_S I\ar[d] \\
			I/J\ar[r]^-{\Delta} & I/J\times_S I/J
		}
		\]
		As $J$ is finite type and quasi-compact, the first diagram implies that the covering map $I\to I/J$ is fpqc. By fpqc descent, we deduce that $I/J$ is flat over $S$. Since the top arrow of the second diagram is a closed embedding, we deduce from the diagram that $I/J$ is separated.
	\end{proof}
	
	We will also need a result for quotient of constant groups with discrete fibers.
	
	\begin{proposition}
		\label{prop: respresentability of finite quotients}
		Let $F$ be a finite group and let $\underline{F} = F\times S$ be the constant group scheme over a noetherian scheme $S$ with fiber $F$. Let $U\subset \underline{F}$ be an open subgroup of $\ul{F}$ over $S$; that is, we can write
		\[
		U = \coprod_{a\in F} U_a\hookrightarrow \ul{F} = \coprod_{a\in F} S_a
		\]
		where $S_a\simeq S$ for each $a\in F$, and the map $U\to \ul{F}$ is a disjoint union of open inclusions $U_a\to S_a$. Then, the fppf quotient $\ul{F}/U$ is representable by a (non-separated) scheme, which is described explicitly by gluing the sections $S_a$ and $S_b$ in $\ul{F}$ along the open subscheme $U_{ab^{-1}}\subset S$.
	\end{proposition}
	\begin{proof}
		As $S$ is noetherian, any open subset of $S$ is quasi-compact, and hence the inclusion $U\to \ul{F}$ is quasi-compact. An identical argument to Proposition \ref{prop: I/J is representable} shows that the fppf quotient $\ul{F}/U$ is representable by an algebraic space. There is an evident morphism from this algebraic space $\ul{F}/U$ to the naive quotient obtained by gluing, and it therefore follows from \cite[Lemma 3.1]{olsson-starr} that the quotient $\ul{F}/U$ is representable by a scheme.
	\end{proof}

	\subsection{The regular quotient and smoothness of stabilizer group schemes via Kostant--Rallis sections}
	\label{subsec: regular quotient via Kostant Rallis sections}
	
	In this section we describe the regular locus $\fp^{\reg}$ as the union of the $H$-orbits of potentially multiple Kostant--Rallis sections.  We use this to deduce smoothness of several of the group schemes considered in the previous section in a way completely analogous to the case of the adjoint action of $G$ on $\fg$ as considered in \cite{ginzburg2008variations, riche2017kostant}. We will deduce that the regular quotient can be obtained by gluing together multiple copies of the GIT quotient together. In particular, we show the regular quotient for the action of $H$ on $\fp^\reg$ is a (non-separated) scheme. An explicit description of the gluing will be described in the Section \ref{sec: explicit descriptin of regular quotient}.  This is a modification of an argument for the case of the Vinberg monoid found in Proposition 2.12 of \cite{bouthier2015dimension} and Equation 2.7 and Lemma 2.2.8 of \cite{chi2018geometry}. We will then give a direct argument that $I^{\reg}$ descends to the regular quotient.
	
	The key technical input is the following.
	
	\begin{lemma}[Analogue of Proposition 2.12 of \cite{bouthier2015dimension} and Lemma 2.2.8 of \cite{chi2018geometry}]
		\label{lemma Gm equivariant and includes nilpotent cone is whole thing}
		Let $U\subset \fp^{\reg}$ be stable under the $H\times \bbG_{m}$-action.  If $U\cap \cN_\fp=\cN_\fp^{\reg}$ then $U=\fp^{\reg}$.
	\end{lemma}
	
	The following proof is identical to that of \cite[Lem 2.2.8]{chi2018geometry}; we provide it here for completeness.
	
	\begin{proof}
		We let $F:=\fp\backslash U$. Let $d_{\reg}$ denote the dimension of a regular $K$ orbit in $\fp$, which is given by the difference of the dimension of $\fp$ and the dimension of the centralizer of a regular element of $\fp$ in $K$. Let $\chi_F = \chi|_F$ be the restriction of the projection $\chi\colon \fp\to \fp\git K$ to $F$, and let $V\subset F$ denote the following variety of $F$
		\[
		V = \{x\in F\colon \dim_x \chi_F^{-1}\big(\chi_F(x)\big) < d_\reg\}.
		\]
		It is clear that $V$ and $F$ are stable under the action of $\bG_m$. Moreover, by upper semicontinuity of the function $x\mapsto \dim_x \chi_F^{-1}\big(\chi_F(x)\big)$, the set $V$ is open in $F$.
		
		We have $\cN_\fp\cap F = \cN_\fp\setminus \cN_\fp^\reg$, so by Lemma \ref{lemma: non regulars codimension geq 1 in cNp}, $V$ contains $0\in \fp$. Since $V$ is $\bG_m$ stable and open in $F$, it follows that $V = F$

		By assumption this is a $H\times \bbG_{m}$ stable subscheme of $\fp$. We let $V\subset F$ be the inverse image under $\chi_F := \chi|_{F}$ of the subset 
		\[\{x\in \fa\git W_{\fa}\colon \dim(\chi_{F}^{-1}(x))< \dim(H)-\dim(\fa)\}.\]
		This is an open subscheme of $F$ by upper semicontinuity.  Furthermore, it includes $0\in \fp$ by Lemma \ref{lemma: non regulars codimension geq 1 in cNp}.  As $V$ is preserved by $\bbG_{m}$ and $0$ is in the closure of every $\bbG_{m}$-orbit of $F$ we have that $V=F$ and consequently, $V\cap \fp^\reg = \emptyset$. As $\fp\setminus \fp^\reg\subset \fp\setminus U = F$, it follows that $U = \fp^\reg$.
	\end{proof}

	We now recall that, by Theorem \ref{theorem : structure of nilpotent cone}, for each irreducible component $S\in Irr(\cN_{\fp})$ there is a unique regular $K$-orbit $\cO_S$ of $\fp$ in $S$. Furthermore, by Lemma \ref{lem: existence of sections}, there is a (non-unique) Kostant--Rallis section $\kappa_{S}:\fa\git W_{\fa}\rightarrow \fp^{\reg}$ such that $\kappa_{S}(0)\in \cO_S$.
	
	Let $\cI$ denote the set of $H$ orbits on $\cN_{\fp}^\reg$. For each representative pick a Kostant-Rallis section $\kappa_{i}$ whose image at 0 is in the regular nilpotent $H$ orbit $i\in \cI$.  Let $\cS_{i}$ be the image of the corresponding Kostant-Rallis section. Let $I_{\cS_i}^\reg = I^\reg|_{\cS_i}$. Then, by Proposition \ref{prop: I/J is representable}, the fppf quotient sheaf
	\begin{equation}
		\label{eqn: slice in preg}
		(H\times \cS_i)/I^\reg_{\cS_i}
	\end{equation}
	is represented by a separated, flat algebraic space over $S$. Moreover, the action map induces 
	\[
	(H\times \cS_i)/I^\reg_{\cS_i}\to \fp^\reg,
	\]
	which by Lemma \ref{lem: K-conjugacy in the reg,ss locus} is an isomorphism over the regular, semisimple locus $\fp^{\mathrm{rs}}\subset \fp^\reg$ and which is evidently quasifinite. Hence by \cite[Lemma 3.1]{olsson-starr}, \eqref{eqn: slice in preg} is representable by a scheme, and by Zariski's Main Theorem this map is an open embedding, whose image we denote by $\fp^{\kappa_{i}, H}\subset \fp^\reg$.
	
	\begin{proposition}
		\label{proposition: regular locus is a union}
		The scheme $\fp^\reg$ decomposes as a union
		\[\fp^{\reg}=\bigcup_{i\in \cI}\fp^{\kappa_{i}, H}.\]
	\end{proposition}
	
	\begin{proof}
		This is an immediate consequence of Lemma \ref{lemma Gm equivariant and includes nilpotent cone is whole thing} applied to $U=\cup_{i\in \cI}\fp^{\kappa_{i}, H}$.
	\end{proof}

	\subsubsection{Application to smoothness of $I^{\reg}$}
	
	We start by proving a Lemma.
	
	\begin{lemma}
		\label{lemma: H orbit onto the H orbit}
		We keep notation as in the previous section. For $i\in \cI$, then the morphism 
		\begin{equation}
			\label{eqn : action map on KR sections}
			H\times \cS_{i}\rightarrow \fp^{\kappa_{i}, H}
		\end{equation}
		is smooth and surjective.
	\end{lemma}
	\begin{proof}
		The morphism \eqref{eqn : action map on KR sections} is surjective by definition of $\fp^{\kappa_i,H}$ and is evidently $H$-equivariant and $\bbG_{m}$-equivariant. Hence, the locus in $H\times \cS_i$ at which \eqref{eqn : action map on KR sections} is smooth is open and stable under the actions of $H$ and $\bG_m$. Moreover, by Lemma \ref{lem: differential at zero is surjective}, the differential of \eqref{eqn : action map on KR sections} at $(1,\kappa_{i}(0))$ is surjective, and so the smooth locus contains $(1,\kappa_i(0))\in H\times \cS_i$. Hence, smooth locus must be the entirety of $H\times \cS_i$.
	\end{proof}
	
	\begin{proposition}
		\label{propostion: map to GIT quotient smooth and surjective}
		For any $i\in \cI$, the composition $\chi|_{\fp^{\kappa_{i}, H}}:\fp^{\kappa_{i}, H}\hookrightarrow \fp\rightarrow \fp\git H$ is smooth and surjective.
	\end{proposition}
	
	\begin{proof}
		This is identical to the proof of Proposition 3.3.3 of \cite{riche2017kostant}, namely the composition $H\times \cS_{i}\rightarrow \fp^{\kappa_{i, H}}\rightarrow \fp\git H\cong \cS_{i}$ is identified with the projection to $\cS_{i}$.  Hence \cite[Tag 02K5]{stacks2022stacks} implies  $\chi|_{\fp^{\kappa_{i}, H}}$ is smooth and surjective.
	\end{proof}
	
	We denote by $I_{\cS_{i}}:=I_{H}^{\reg}\times_{\fp^{\reg}}\cS_{i}$ the restriction of $I_{H}^{\reg}$ to $\cS_{i}$.
	
	\begin{proposition}
		\label{prop: Icsi is smooth over Si}
		The map $I_{\cS_{i}}\rightarrow \cS_{i}$ is smooth.
	\end{proposition}
	
	\begin{proof}
		This proof of Proposition 3.3.5 of \cite{riche2017kostant} carries over to this setting.  For completeness we summarize:  As schemes over $\cS_{i}$ we have isomorphisms
		\[I_{\cS_{i}}\cong \cS_{i}\times_{\fp\times \cS_{i}}(H\times \cS_{i})\cong  \cS_{i}\times_{\fp\times_{\fp\git H} \cS_{i}}(H\times \cS_{i})\cong  \cS_{i}\times_{\fp}(H\times \cS_{i}).\]
		Hence as $H\times \cS_{i}\rightarrow \fp^{\reg}\hookrightarrow \fp$ is smooth by Lemma \ref{lem: differential at zero is surjective}, we have that $I_{\cS_{i}}\rightarrow \cS_{i}$ is smooth.
	\end{proof}
	
	\begin{prop}
		\label{prop: smoothness of Ireg}
		Recall we assume that $G$ satisfies the standard hypotheses of Section \ref{subsec: conditions on char}. Then the regular centralizer group scheme $I^{\reg}\rightarrow \fp^{\reg}$ is smooth.
	\end{prop}
	
	Note that for $\theta$-quasisplit pairs $(G,\theta)$ in characteristic 0 this is proved for $I_{G^{\theta}}^\reg$ and $I_{N_G(K)}^\reg$ in \cite[Lemma 5 and Lemma 8]{gppn} adapting the ideas of \cite{donagi2002gerbe}. In Theorem 4.7 of \cite{leslie} this is generalized to the $\theta$-quasisplit case when the characteristic is $p>2$ and $p$ is such that the regular centralizer group scheme $I_G^{\reg}\rightarrow \fg^{\reg}$ for the adjoint action of $G$ on $\fg$ is smooth. This holds for example under condition (C3) of \cite{riche2017kostant}.
	
	\begin{proof}
		We recall the notation $\cI$ for an indexing set of the $\pi_0(H)$-orbits of irreducible components in the nilpotent cone $\cN_\fp$. For each $i\in \cI$, we fix a Kostant-Rallis section $\kappa_i$ whose image over $0$ is contained in the $\pi_0(H)$ orbits of $\cN_\fp^\reg$ indexed by $i$. By Proposition \ref{proposition: regular locus is a union} it is sufficient to show that for each $i\in \cI$, the morphism $I^{\reg}|_{\fp^{\kappa_{i}, H}}\rightarrow \fp^{\kappa_{i}, H}$ is smooth. The proof is now identical to Corollary 3.6 in \cite{riche2017kostant}:  Namely, the diagram 
		\[
		\begin{tikzcd}
			H\times I_{\cS_{i}} \arrow{r}\arrow{d} & I^{\reg}|_{\fp^{\kappa_{i}, H}}\arrow{d}\\
			H\times \cS_{i} \arrow{r} & \fp^{\kappa_{i}, H}
		\end{tikzcd}
		\]
		is Cartesian.  Hence Lemma \ref{lemma: H orbit onto the H orbit} implies that $H\times I_{\cS_{i}}\rightarrow I_{H}^{\reg}$ is smooth and surjective.  Furthermore the Lemma \ref{lemma: H orbit onto the H orbit} and Proposition \ref{prop: Icsi is smooth over Si} tell us that the composition $H\times I_{\cS_{i}}\rightarrow \fp^{\kappa_{i}, H}$ is smooth.  Hence $I^{\reg}|_{\fp^{\kappa_{i}, H}}\rightarrow \fp^{\kappa_{i}, H}$ is smooth by \cite[Tag 0K25]{stacks2022stacks}.
	\end{proof}
	
	\subsubsection{Construction of the regular quotient via Kostant-Rallis sections}
	\label{subsubsec: application regular quotient}
	
	We continue the notation of the previous section:  Namely, $\cI$ will denote an indexing set for the $\pi_0(H)$ orbits of irreducible components of the nilpotent cone $\cN_\fp$. Let $\wt{\fp\git H}$ be the union of $|\cI|$ copies of $\fp\git H$ where we glue the copies labeled by $i$ and $j$ on the subscheme $U_{ij} = \chi(\fp_i^\reg\cap \fp_j^\reg)\subset \fp\git H$ where the sections $\kappa_{i}$ and $\kappa_{j}$ are conjugate. Note that the locus $U_{ij}$ is open in $\fp\git H$ since $\chi$ is flat, see Lemma \ref{lemma: p-->a//Wa is flat}. Note that \emph{a priori} it is not clear that the gluing is done in a fashion compatible with the $\bbG_{m}$-action, and thus it is not clear that $\widetilde{\fp\git H}$ has a $\bbG_{m}$-action coming from the $\bbG_{m}$-actions on $\fa\git W_{\fa}$.

	\begin{thm}
		\label{theorem : Description of regular quotient gluing not explicit}
		We have a $\bbG_{m}$-equivariant isomorphism of schemes
		\[\widetilde{\fp\git H}\simeq \fp^{\reg} \myfatslash H,\]
		where the $\bbG_{m}$-action comes on $\widetilde{\fp\git H}$ comes from the $\bbG_{m}$-action on each copy of $\fp\git H$.
		Furthermore this isomorphism commutes with the natural morphisms to $\fp\git H$.
	\end{thm}
	
	\begin{remark}
		Note that, while we assume that $H$ is smooth throughout, the left hand side is definable regardless of whether or not $H$ is smooth. In contrast, the right hand side is only defined when $I^\reg$ is smooth.
	\end{remark}
	
	\begin{proof}
		We check that the quotient $\widetilde{\fp\git H}$ satisfies the conclusions of Theorem \ref{thm:aov}. The map $\fp^\reg\to \fp\git K$ factors through the scheme $\widetilde{\fp\git H}$, so we have a morphism $[\fp^\reg/H]\to \widetilde{\fp\git H}$. Since $\widetilde{\fp\git H}$ is a scheme and $I^\reg$ descends to the full inertia scheme of $[\fp^\reg/H]$, property (2) is immediate. It suffices, therefore, to show that $[\fp^\reg/H]$ is an fppf gerbe over $\widetilde{\fp\git H}$. To this end, consider the pullback of $[\fp^\reg/H]$ along the fppf cover $\fp^\reg\to \widetilde{\fp\git H}$. There is a natural $BI^\reg$ action on the pullback $[\fp^\reg/H]\times_{\widetilde{\fp\git H}}\fp^\reg$ over $\fp^\reg$. The diagonal map
		\[
		\fp^\reg\to \fp^\reg\times \fp^\reg\to [\fp^\reg/H]\times_{\widetilde{\fp\git H}}\fp^\reg
		\]
		gives a section of 
		\[
		[\fp^\reg/H]\times_{\widetilde{\fp\git H}}\fp^\reg\to \fp^\reg.
		\]
		The action of $BI^\reg$ on this section gives a map
		\[
		BI^\reg\to [\fp^\reg/H]\times_{\widetilde{\fp\git H}}\fp^\reg.
		\]
		which is easily checked to be an isomorphism. We conclude that $[\fp^\reg/H]$ is indeed a gerbe, banded by $I^\reg$, over $\fp\git H$.
		
		We have concluded by uniqueness that $\fp^\reg\myfatslash H\simeq\widetilde{\fp\git H}$ over $\fp^\reg$. There is a unique $\bG_m$ action on $\fp^\reg\myfatslash H$ lifting the $\bG_m$ action on $\fp\git H$, and hence there is a natural $\bG_m$ action on $\widetilde{\fp\git H}$ coming from the $\bG_m$ action on $\fp\git H$.
	\end{proof}

	\subsection{Explicit Description of the Regular Quotient}
	\label{sec: explicit descriptin of regular quotient}
	
	\subsubsection{Overview}
	
	We now turn to an explicit description of the geometry of the regular quotient $\fp^{\reg}\myfatslash H$.  We ultimately provide two different descriptions.
	
	The first appears in Theorem \ref{thm: Description of regular quotient as a quotient of component groups}, where we describe the regular quotient in terms of certain quotients of component groups. This perspective reduces to computing the differences between certain regular centralizer group schemes as the group varies. This description is also used to prove the inductive description below.
	
	The second is done by the following multistep procedure:
	\begin{itemize}
		\item Firstly, we reduce to the case of simply connected simple groups and the diagonal case $H=G_{1}\xhookrightarrow{\Delta} G_{1}\times G_{1}$ of Example \ref{example: complex case}, using Proposition \ref{prop: reduction of regular quotient to simples}.
		\item Secondly, we reduce understanding the orbits above a point in $\fa\git W_{\fa}$ to the case of nilpotent cones of certain Levis, that we call descendants of the symmetric pair, in Proposition \ref{prop: description of regular quotient: reduction to Levi's} and Proposition \ref{proposition: Above the hyperplanes}. This is a Lie algebra version of the descendants described in Section 5.1.1 of \cite{sl_endoscopy}.
		\item We use the immediately preceding point to describe the structure for simple, simply connected groups (and the diagonal case, Example \ref{example: complex case}).  Except for the case of $SO(n)\times SO(n)\hookrightarrow SO(2n)$ (considered in Example \ref{ex son son so2n}) this is not complicated due to the fact that there are at most $2$ regular $H$-orbits in the nilpotent cone. The resulting explicit description of the regular quotient is included as Theorem \ref{theorem: gluing pattern}. For such cases, we get a description $\fp^\reg\myfatslash H\cong \fa\git W_{\fa}\coprod_{U}\fa\git W_{\fa}$ for an explicitly described open $U\subset \fa\git W_{\fa}$.  
	\end{itemize}
	
	In Section \ref{subsections: Examples regular quotient} we explicitly compute the regular quotient in several cases of interest.

	\subsubsection{Description of the Regular Quotient via Comparison of Regular Centralizers}
	\label{subsec: centralizer description of regular quotient}
	
	In this section, we give a first description of the regular quotient using a comparison of regular centralizer group schemes for $H$ and for the full normalizer $N_G(K)$. Throughout this section, we will use subscripts to indicate in which group centralizers are taken; for example, $I_H^\reg$ will denote the centralizers in the group $H$. 
	
	\begin{lem}
		\label{lemma: the quotient INGK as a cover}
		Fix a maximal $\theta$-split torus $A$, and let 
		\[
		F^* = \{a\in A\colon a^2\in Z(G)\}\subset A
		\]
		be as in Proposition \ref{prop: extension of K}. Moreover, let $Z_-$ denote the subgroup of the center $Z(G)$ on which $\theta$ acts by inversion. 
		\begin{enumerate}
			\item We have an isomorphism over $\fp^{\mathrm rs}$
			\begin{equation}
				\label{eqn: INGK/IK and Z-}
				\left.\big(I^{\reg}_{N_G(K)}/Z_-\cdot I^{\reg}_H\big)\right|_{\fp^{\mathrm rs}} \simeq \coprod_{a\in F^*/Z_-\cdot (F^*\cap H)}\fp^{\mathrm rs}
			\end{equation}
			\item The isomorphism \eqref{eqn: INGK/IK and Z-} extends to an isomorphism 
			\begin{equation}
				\label{eqn: INGK mod IK over reg locus}
				I^{\reg}_{N_G(K)}/Z_-\cdot I^{\reg}_H \simeq \coprod_{a\in F^*/Z_-\cdot (F^*\cap H)} U_{a}
			\end{equation}
			over $\fp^\reg$, where $U_a\hookrightarrow \fp^{\reg}$ is the inclusion map for an open set $\fp^{\mathrm{rs}}\subset U_a\subset \fp^{\reg}$.
		\end{enumerate}
	\end{lem}
	\begin{remark}
		Note that, in the above lemma, $Z_-\cdot H$ also gives a symmetric pair $(G,\theta,Z_-\cdot H)$, and so the inclusion $Z_-\cdot I_H^\reg\to I_{N_G(K)}^\reg$ satisfies the conditions of Proposition \ref{prop: I/J is representable}. In particular, the quotient $I_{N_G(K)}^\reg/Z_-\cdot I_H^\reg$ is a separated, flat algebraic space over $\fp^\reg$ representing the corresponding fppf quotient sheaf. It is part of the claim of the lemma that this quotient is representable by a scheme over $\fp^\reg$.
	\end{remark}
	\begin{proof}
		Note that $F^*$ is a smooth group scheme since $\mathrm{char}(k)\neq 2$, so that the quotient $F^*/Z_-\cdot (F^*\cap H)$ is reduced. Hence, the above disjoint unions taken over the finite group $F^*/Z_\cdot (F^*\cap H)$ make sense.
		
		The inclusion $I_{N_G(K)}^\reg\subset N_G(K)\times \fp^{\reg}$ defines a map
		\begin{equation}
			\label{eqn: INGK/IK on pi0}
			I_{N_G(K)}^\reg/Z_-\cdot I_H^\reg\to \big(N_G(K)/Z_-\cdot H\big)\times \fp^{\reg}
		\end{equation}
		with target a constant group scheme with discrete fibers. This map is evidently quasi-finite, and so by \cite[Lemma 3.1]{olsson-starr}, we see that the quotient $I_{N_G(K)}^\reg/Z_-\cdot I_H^\reg$ is representable by a scheme.
		
		For any fixed $y\in \fp^{\mathrm{rs}}$, let $A_y$ be the unique maximal $\theta$-split torus whose Lie algebra contains $y$, and let $F_y^*\subset A_y$ be the subgroup of $a\in A_y$ such that $a^2\in Z(G)$. Then, it is clear from Proposition \ref{prop: extension of K} that $I_{N_G(K),y}^\reg = F_y^*\cdot Z_-\cdot I_{H,y}^\reg$. Therefore, the fiber of the map \eqref{eqn: INGK/IK on pi0} at $y$ is identified with the identity map
		\[
		F_y^*/Z_-\cdot (F_y^*\cap K)\to F_y^*/Z_-\cdot (F_y^*\cap K)
		\]
		In particular, \eqref{eqn: INGK/IK on pi0} is an isomorphism over the regular, semisimple locus, proving part (1).
		
		For (2), we first show that the map \eqref{eqn: INGK/IK on pi0} has open image. Indeed, the map $I_{N_G(K)}^\reg\to \fp^\reg$ is flat by \ref{prop: smoothness of Ireg}. Hence, for any connected open set $V\subset I_{N_G(K)}^\reg$, the image of $V$ in $\fp^\reg$ is also connected and open. By connectedness of $V$, the image of $V$ in $\big(N_G(K)/Z_-\cdot H\big)\times \fp^{\reg}$ lands in a single component. Since the map $I_{N_G(K)}^\reg\to \fp^\reg$ factors through \eqref{eqn: INGK/IK on pi0}, and since $\big(N_G(K)/Z_-\cdot H\big)\times \fp^{\reg}$ is a constant group scheme with finite and discrete fibers, it therefore follows that the image of $V$ under the map 
		\[
		I_{N_G(K)}^\reg\to \big(N_G(K)/Z_-\cdot H\big)\times \fp^{\reg}
		\]
		is open. 
		
		We finally claim that the map \eqref{eqn: INGK/IK on pi0} remains an injection over $\fp^{\reg}$. In particular, this amounts to the following claim:  Let $y\in \fp^\reg$. For any $g_1,g_2\in I_{N_G(K),y}^\reg$, if $g_1 = hg_2$ for $h\in H$, then in fact $h\in I_{H,y}^\reg$. Indeed, since $g_1$ and $g_2$ centralize $y$, we have
		\[
		y = \mathrm{ad}(hg_2)\cdot y = \mathrm{ad}(h)y,
		\]
		and hence, $h\in I_{G,y}\cap H = I_{H,y}^\reg$.
		
		It follows that the map \eqref{eqn: INGK/IK on pi0} describes the quotient $I_{N_G(K)}^\reg/Z_-\cdot I_H^\reg$ as the disjoint union of open subsets of $\fp^{\reg}$ extending the sheets $\big( N_G(K)/Z_-\cdot H \big)\times \fp^{\mathrm{rs}}$.
	\end{proof}
	
	\begin{lem}    
		\label{lem: N acts transitively on fibers}
		The group $N_G(K)$ acts transitively on the fibers of the map $\fp^\reg\to \fp\git K$.
	\end{lem}
	\begin{proof}
		Theorem \ref{theorem : Description of regular quotient gluing not explicit} reduces this to the $N_G(K)$ action on the zero fiber $\cN_\fp^\reg$. Theorem \ref{theorem: N_G(K) acts transitively on regular nilpotents} proves this case.
	\end{proof}

	\begin{thm}
		\label{thm: Description of regular quotient as a quotient of component groups}
		Consider the natural action of the constant group scheme $\ul{N_G(K)}:=N_G(K)\times \fp\git H$ on $\fp^\reg$ over $\fp\git H$.
		\begin{enumerate}
			\item A choice of Kostant-Rallis section $\kappa\colon \fp\git H\to \fp^\reg$ gives an identification
			\begin{equation}
				\label{eqn: preg as a quotient}
				\fp^\reg\simeq \ul{N_G(K)}/\kappa^*I^\reg_{N_G(K)},
			\end{equation}  
			as schemes over $\fp\git H$.
			\item The regular quotient $\fp^\reg\myfatslash H$ is identified with the quotient
			\begin{equation}
				\label{eqn: regular quot via reg centralizers}
				\fp^\reg\myfatslash H \simeq \frac{\ul{N_G(K)}/\ul{Z_-\cdot H}}{\kappa^*(I^\reg_{N_G(K)}/Z_-\cdot I^\reg_H)}
			\end{equation}
		\end{enumerate}
	\end{thm}
	
	\begin{remark}
		Before proving Theorem \ref{thm: Description of regular quotient as a quotient of component groups}, we first remark on the representability of the various quotients appearing in the statement of the theorem. Let $\underline{N_G(K)} = N_G(K)\times \fp\git H$ denote the constant group scheme on $\fp\git H$, and let $\kappa\colon \fp\git H\to \fp^\reg$ denote a Kostant-Rallis section. Note that $\kappa^*I_{N_G(K)}^\reg$ is a flat, affine group scheme over $\fp^\reg$. Hence, by Proposition \ref{prop: I/J is representable}, the fppf quotient $\underline{N_G(K)}/\kappa^*I_{N_G(K)}^\reg$ is a separated, flat algebraic space over $\fp\git H$. The map
		\[
		\underline{N_G(K)}/\kappa^*I_{N_G(K)}^\reg\to \fp^\reg
		\]
		obtained by action on a Kostant-Rallis section is quasi-finite, and so we conclude that the quotient $\underline{N_G(K)}/\kappa^*I_{N_G(K)}^\reg$ is representable.
		
		We have shown in Lemma \ref{lemma: the quotient INGK as a cover} that the quotient $I_{N_G(K)}^\reg/Z_-\cdot I_H^\reg$ is representable by a flat group scheme over $\fp^\reg$, and hence its pullback under $\kappa$ is as well. We have a natural inclusion
		\[
		\kappa^*(I_{N_G(K)}^\reg/Z_-\cdot I_H^\reg)\subset \underline{N_G(K)/Z_-\cdot H}
		\]
		of flat, affine group schemes over $\fp\git H$. The group scheme $\underline{N_G(K)/Z_-\cdot H}$ is a constant group scheme with finite fibers, while by Lemma \ref{lemma: the quotient INGK as a cover}, the inclusion of $\kappa^*(I_{N_G(K)}^\reg/Z_-\cdot I_H^\reg)$ into $\underline{N_G(K)/Z_-\cdot H}$ is an open embedding. Therefore, representability of the right hand side of \eqref{eqn: regular quot via reg centralizers} follows from Proposition \ref{prop: respresentability of finite quotients}. Hence, all quotients in Theorem \ref{thm: Description of regular quotient as a quotient of component groups} are representable by schemes.
	\end{remark}
	
	We now prove the theorem.
	
	\begin{proof}[Proof of Theorem \ref{thm: Description of regular quotient as a quotient of component groups}]
		By acting on the image of the Kostant--Rallis section $\kappa$ we gain a surjective morphism $\ul{N_{G}(K)}\rightarrow \fp^{\reg}$.  This clearly factors through an isomorphism $\ul{N_{G}(K)}/\kappa^*I^\reg_{N_G(K)}\rightarrow \fp^{\reg}$.
		
		Part (2) then follows by considering the transitive $\underline{N_G(K)}$ action on the right hand side of the description of $\fp^{\reg}\myfatslash H$ given in Theorem \ref{theorem : Description of regular quotient gluing not explicit}. 
	\end{proof}

	\subsubsection{Reduction to the Adjoint Case}
	\label{subsection: reduction to the adjoint case}
	
	We begin by reducing to the case of $G$ simple, adjoint. Namely, since $\theta$ is an algebraic involution, the center $Z(G)$ is stable under $\theta$. Hence, there is a well-defined involution $\theta_{\ad}$ on the quotient $G_{\ad}:=G/Z(G)$. Let $Z_H(G) = Z(G)\cap H$.
	\begin{definition}
		To a symmetric pair $(G,\theta,H)$, we associate the \emph{adjoint symmetric pair} $$(G_{\ad},\theta_{\ad},H/Z_H(G)).$$
		We denote by $\fp_\ad$ the $(-1)$ eigenspace of $\theta_{\ad}$ on $\fg_{\ad} = \mathrm{Lie}(G_\ad)$.
	\end{definition}
	
	It is immediate that we have an exact sequence
	\begin{equation}
		\label{eqn: ses p to pad}
		0\to \mathfrak{z}^-\to \fp\to \fp_{\ad}\to 0
	\end{equation}
	where $\mathfrak{z}^-$ denotes the $(-1)$ eigenspace of $\theta$ acting on $\mathfrak{z}(\fg)$. The projection $\fp\to \fp_\ad$ is compatible with action of $H$, where $H$ acts on $\fp_\ad$ through the projection map $H\to H/Z_H(G)$. 
	
	Now, we may decompose $G_\ad$ as a product 
	\[
	G_\ad = \prod_j G_j
	\]
	where each $G_j$ is simple. Now, for each $j$, by simplicity of the factors $G_j$, we see that either $\theta(G_j) = G_j$ is preserved by the involution or else there exists a unique index $i$ such that $\theta$ induces isomorphisms 
	\[
	\theta|_{G_j}\colon G_j\to G_i\quad \text{and}\quad \theta|_{G_i}\colon G_i\to G_j
	\]
	Reordering the indices as necessary, we deduce the following decomposition.
	
	\begin{cor}
		\label{cor: Gad decomposition}
		There is a decomposition 
		\begin{equation}
			\label{eqn: Gad decomposition compatible with theta}
			G_\ad \simeq\prod_{j}(G_j\times G_j)\times \prod_{i}G_i
		\end{equation}
		where for each index $i$, $G_i$ is simple and $\theta_i = \theta|_{G_i}$ preserves $G_i$ and for each index $j$, $G_j$ is simple and $\theta_j = \theta|_{G_j\times G_j}$ is the swaps the factors in $G_j\times G_j$.
	\end{cor}
	
	Now, consider the decomposition \eqref{eqn: Gad decomposition compatible with theta}. For each index $i$, let $\fp_i$ denote the $(-1)$ eigenspace of $\theta_i$ on $\fg_i$ and let $H_i$ be the image of $H/Z_H(G)$ along the projection $G_{\ad}\to G_i$. We are ready to state the main result of this subsection.
	
	\begin{proposition}
		\label{prop: reduction of regular quotient to simples}
		Let $(G,\theta,H)$ be a symmetric pair with associated adjoint pair $(G_\ad,\theta_\ad,H/Z_H(G))$. Then, with notation as above, we have
		\[
		\fp^\reg\myfatslash H\simeq \mathfrak{z}^- \times \prod_j \fg_j\git G_j\times \left(\prod_i \fp_i^\reg\myfatslash K_i\right)/\pi_0(H)
		\]
		where in the second terms, $\fg_j\git G_j$ denotes the GIT quotient of $\fg_j$ by the diagonal action of $G_j$, and $\pi_0(H)$ acts through the diagonal map
		\[
		\pi_0(H)\twoheadrightarrow \pi_0(H/Z_H(G))\to \prod_i\pi_0(H_i).
		\]
	\end{proposition}
	\begin{proof}
		From the short exact sequence \eqref{eqn: ses p to pad}, we have
		\[
		\fp^\reg\myfatslash H\simeq \mathfrak{z}^-\times \fp^\reg_\ad\myfatslash (H/Z_H(G))
		\]
		Moreover, the decomposition of Corollary \ref{cor: Gad decomposition} gives a decomposition of Lie algebras
		\[
		\fp_\ad = \left(\prod_j \fg_j\right)\oplus \left( \prod_i \fp_i \right)
		\]
		which is compatible with $H$ actions, where $H$ acts through $H\to H/Z_H(G)$ on the left hand side and through the diagonal map
		\[
		H\to H/Z_H(G)\to \prod_j N_{G_j\times G_j}(G_j)\times \prod_i H_i
		\]
		on the right hand side. The proposition therefore follows.
	\end{proof}

	\subsubsection{Reduction to Levi Subgroups}
	\label{section: reduction to levis}
	
	From now on, we assume that $G$ is simple and of adjoint type. We describe in this section the process of Levi induction needed to compute the gluing loci for the regular quotient. The Levi induction we use is closely related to the degeneration used by S. Leslie in Section 5.1.1 of \cite{sl_endoscopy} termed the ``descendant'' of a semisimple element. We will first state a Lie theoretic definition of descendants and give a root theoretic description of the Lie algebra of this Levi subgroup. Then, we will prove a reduction result, see Proposition \ref{proposition: Above the hyperplanes}, which will prove invaluable for computations.
	
	\begin{definition}
		Fix a $\theta$-Cartan $\fa\subset \fp$. For an element $x\in \fa$, the descendant of $x$ is the tuple $(G_x,\theta|_{G_x},H_x)$ where $G_x$, respectively $H_x$, is the stabilizer of $x$ in $G$, resp. $H$.
	\end{definition}

	\begin{proposition}
		\label{prop: descendants are symmetric pairs}
		For $x\in \fa$, the descendant at $x$ is a symmetric pair.
	\end{proposition}
	\begin{proof}
		First, note that $G_x$ is a Levi of $G$ and hence is reductive. Since the adjoint action preserves the Cartan decomposition, $G_x$ is stable by the action of $\theta$. It is trivial that $H_x$ contains the connected component of the fixed point scheme $K_{G_x}:=(G_x^\theta)^\circ$. Moreover, for any $h\in H_x$, $\mathrm{Ad}(h)$ by definition lies in $N_G(K) = N_G(G^\theta)$ and so preserves $G_x^\theta$. Hence, $H_x$ is contained in the normalizer $N_{G_x}(G_x^\theta) = N_{G_x}(K_{G_x})$. 
		
		Finally, we note that $G_x$ is smooth as it is a Levi of $G$ and the characteristic is assumed to be good for $G$. For $H_x$, take a character $\lambda$ whose Lie algebra is the span of $x$. The subgroup $H_x$ is the fixed points of the image of $\lambda$, which is a subgroup of multiplicative type. Hence, by Proposition R.1.1 of \cite{conrad}, $H_x$ is smooth.
	\end{proof}
	
	Fix a maximally $\theta$-split torus $T\supset A$. As in Subsection \ref{subsec: restricted root systems}, we fix the Levi $M = C_G(\fa)$ as well as a so-called minimal $\theta$-split parabolic $P$, i.e. a parabolic $P$ of $G$ such that $P\cap \theta(P) = M$. (See \ref{subsec: restricted root systems}, above Definition \ref{def: Satake diagram}.) Fix a choice of Borel $B$ of $G$ such that $T\subset B\subset P$. Let $\Sigma\subset \Phi$ denote the corresponding set of simple roots of $G$. 
	
	The Levi subgroup $G_x$ is determined by a subset of the simple roots, which we denote by $\Sigma_x\subset \Sigma$. We determine the Satake diagram of $(G_x,\theta|_{G_x})$ in the following.
	
	\begin{proposition}
		\label{prop: descendants given by satake}
		The Satake diagram of the descendant of the symmetric pair $(G,\theta,H)$ at a semisimple element $x\in \fa$ is given by deleting from the Satake diagram of $(G,\theta)$ all vertices not in $\Sigma_x\subset \Sigma$.
	\end{proposition}
	\begin{proof}
		As the involution on the descendant is given by restricting the involution on $G$ to $G_x$, the involution $\iota$ on $\Sigma_x$ as well as the compact nodes must match those of $(G,\theta)$.
	\end{proof}

	\begin{rem}
		We note that the above does not determine the subgroup $H_x$ of $G_x$. Indeed, even if one takes $H = K$, it is \emph{not} true in general that $H_x = K_{G_x}$ for the descendant. For example, consider the symmetric pair corresponding to the diagonally embedded $S(\GL_2\times \GL_2)\subset \SL_{4}$ (see example \ref{example:U(n,n) case}). Then, the Levi $G_x$ with Lie algebra
		\[
		\fg_x = \left\{ \begin{pmatrix}
			A & B \\ B & A
		\end{pmatrix}\colon {\mathrm Tr}(A) = 0\right\}
		\]
		is a descendant at $x = \begin{pmatrix}
			& I_2 \\ I_2 & 
		\end{pmatrix}$. One computes
		\[
		K_x = \left\{ \begin{pmatrix}
			g & \\ & g
		\end{pmatrix}\colon \det(g) = \pm 1\right\}.
		\]
		In particular, $K_x$ is disconnected. By definition, $K_{G_x} = (G_x^\theta)^\circ = (K_x)^\circ$ is the diagonally embedded copy of $\SL_2$.
	\end{rem}
	
	The converse of Proposition \ref{prop: descendants given by satake} is given in the following result.
	
	\begin{proposition}
		\label{prop: satake gives descendants}
		The Satake diagrams arising from descendants of $(G,\theta,H)$ are exactly those for which we delete an $\iota$-stable collection of white nodes from the Satake diagram of $(G,\theta, H)$.
	\end{proposition}
	\begin{proof}
		The centralizer $G_x$ always contains the Levi $M = C_G(\fa)$, and therefore, $\Sigma_x\subset \Sigma$ must contain the simple roots $\Sigma_M$ of the Levi $M$. These are exactly the compact (black) nodes in the Satake diagram. Of the remaining nodes, we note first that $\Sigma_x$ must be $\iota$-stable. Conversely, for any $\iota$-stable subset of simple roots $\Sigma_M\subset \Sigma'\subset \Sigma$, we can take the associated Levi subgroup $M\subset L'\subset G$. Since $\mathfrak{l}' = \mathrm{Lie}(L')$ contains $\fa$, its center $\mathfrak{z}(\mathfrak{l}')$ is contained in $\fa$. Moreover, $L'$ is the centralizer of a generic element in $\mathfrak{z}(\mathfrak{l}')$, and so the result follows.
	\end{proof}
	
	To simplify notation, for the rest of this section we fix $x\in \fa$, and put $L = G_x$.
	
	\begin{lem}
		\label{lem: restricted roots of L are sub of G}
		There is an inclusion $\Phi_{r,L}\subset \Phi_r$. Moreover, the little Weyl group $W_{\fa,L}$ of the Levi $L$ is a subgroup of $W_\fa$, and the resulting map
		\[
		\varphi_L\colon \fa\git W_{\fa,L}\to \fa\git W_\fa
		\]
		is a flat, ramified cover.
	\end{lem}
	\begin{proof}
		By construction $\fa\subset \fp_L\subset \fp$ and $\Phi_L\subset \Phi_G$. Hence, we have a commutative diagram
		\[
		\xymatrix{
			\Phi_L\ar@{^{(}->}[r]\ar[dr]_-{\mathrm{res}} & \Phi_G\ar[d]^-{\mathrm{res}} \\
			& \fa^*
		}
		\]
		where $\mathrm{res}$ denotes the restriction map to $\fa$ and $\fa^* = \Hom(\fa,k)$. It follows that $\Phi_{r,L}\subset \Phi_{r}$. 
		
		It is therefore immediate that $W_{\fa,L}\subset W_\fa$. Flatness follows from flatness of the projection maps $\fa\to \fa\git W_\fa$ and $\fa\to \fa\git W_{\fa,L}$.
	\end{proof}
	
	Recall $\Phi_r^{\mathrm{red}}$, resp. $\Phi_{r,L}^{\mathrm{red}}$, denote the \emph{reduced} restricted root system of the symmetric pairs $(G,\theta, H)$, resp. $(L,\theta, H_L)$. The function 
	\[
	\prod_{\alpha\in \Phi_r^{\mathrm{red}}} d\alpha\in k[\fa]
	\]
	is $W_\fa$ invariant and so defines a divisor $\fD\subset \fa\git W_\fa$. Similarly, we define a divisor $\fD_{L}\subset \fa\git W_{\fa,L}$. While it is obvious since $L$ is a Levi of $G$ that the root system for $L$ is a subroot system of $G$, i.e. $\Phi_L\subset \Phi$, we also have this result for the restricted root systems of the symmetric pairs.
	
	The Levi $L$ will be used to describe the stack $[\fp/H]$ on the unramified locus of $\varphi_L$. This is described in the following proposition.
	
	\begin{proposition}
		\label{proposition: map from levi on hitchin base is etale}
		Choose a subset $\Psi\subset \Phi_r^{\mathrm{red}}-\Phi_{r,L}^{\mathrm{red}}$ such that for every pair of roots $\pm\alpha\in \Phi_r^{\mathrm{red}} - \Phi_{r,L}^{\mathrm{red}}$, the cardinality of $\{\pm\alpha\}\cap \Psi$ is one. Consider the function
		\begin{equation}
			\label{eqn: function of RL}
			\prod_{\alpha\in \Psi}d\alpha \in k[\fa]
		\end{equation}
		
		\begin{enumerate}
			\item \eqref{eqn: function of RL} is invariant under the $W_{\fa,L}$ action and so defines a divisor $\fR_L$.
			\item There is an equality of divisors
			\[ \varphi_L^*\fD = \fD_L+2\fR_L. \]
			\item The ramification of the map $\varphi_L$ is supported on the divisor $\fR_L$; in particular, $\varphi_L$ is \'etale on the complement of $\fR_L$.
		\end{enumerate}
	\end{proposition}
	\begin{proof}
		Part (1) follows verbatim from the proof of \cite[Lemma 1.10.2]{ngo2010lemme}, while part (2) follows from a direct computation. 
		
		The projection $\pi$ factors through $\varphi_L$, giving covers $\pi$ and $\pi_L$ as below.
		\[
		\xymatrix{
			\fa\ar[d]_-{\pi_L}\ar[dr]^-{\pi} & \\
			\fa\git W_{\fa,L}\ar[r]_-{\varphi_{L}} & \fa\git W_\fa
		}
		\]
		The ramification of $\varphi_L$ is given by the closed locus of points $y\in \fa\git W_{\fa,L}$ such that $\varphi_L(y)\in \fD$ while $y\not\in \fD_L$. By part (2), this locus is agrees with the divisor $\fR_L$.
	\end{proof}
	
	\begin{definition}
		\label{definition: UL and sundry}
		Let $\varphi_{L}$, $\pi$, $\pi_L$, and $\fR_L$ be as above and let $\fc = \fp\git H$ and $\fc_L = \fp_L\git H_L$. Let $V_L\subset \fa\git W_{\fa,L} = \fc_L$ be the image of $\fa\setminus \fR_L$ under the projection $\pi_L$, and let $U_L\subset \fp_L^\reg$ be the preimage of $V_L$ under the map $\chi_L^\reg\colon \fp_L^\reg\to \fc_L$.
		
		We will denote by $i_{L}\colon \fp_{L}\rightarrow \fp$ the inclusion map and
		\[  \chi\colon \fp\to \fc \quad\text{ and }\quad \chi_L\colon \fp_L\to \fc_L  \]
		the projection maps.
	\end{definition}
	
	\begin{prop}
		\label{prop: description of regular quotient: reduction to Levi's}
		The map $i_L$ restricts to an $H_L$-equivariant map
		\[
		i_L\colon U_L\to \fp^\reg,
		\]
		and there is an isomorphism of regular centralizer group schemes $i_L^*I_{H_L}^\reg|_{U_L}\simeq I_{H}^\reg|_{U_L}$.
		
		Furthermore, there is an induced isomorphism of regular quotients
		\[p_{L}\colon (\fp^{\reg}_{L}\myfatslash H_{L})|_{V_L}\rightarrow \varphi_L^*(\fp^{\reg}\myfatslash H)|_{V_L}.\]
	\end{prop}
	
	\begin{proof}
		Recall that $L = G_x$ for a fixed $x\in \fa$. Let $y\in U_L$. Then, let $y = s+n$ be the Jordan decomposition of $y$ from Lemma \ref{lemma : jordan decomposition}. Since $\chi(y) = \chi(s)\in V_L$, the centralizer $C_G(s)$ is contained in $L$, and since $C_G(y) = C_G(s)\cap C_G(n)$, it follows that $I_{G,y} = I_{L,y}$ and hence $I_{H_L,y} = I_{H,y}\cap I_{L,y}$. In particular, $i_L$ sends $U_L\cap \fp_L^\reg$ to the regular locus $\fp^\reg$. Moreover, since $I_{G,y} = I_{L,y}$ for $y\in U_L$, it is immediate that $i_L^*I_{H_L}^\reg|_{U_L}\simeq I_{H}^\reg|_{U_L}$.
		
		That there exists an isomorphism $p_L$ of regular quotients over $V_L$ follows from the construction of the regular quotient combined with the isomorphism of regular centralizers over $U_L$.
	\end{proof}

	To conclude, we will reduce to computations of regular nilpotent orbits. To do this we will use Proposition \ref{proposition: Above the hyperplanes}.
	
	\begin{lemma}
		\label{lemma: semisimple hyperplanes intersect at a point}
		Let $(G, \theta)$ be a semisimple group with involution $\theta$.  Then the intersection of all root hyperplanes of the restricted root system is 
		\[\bigcap_{\alpha\in \Phi_{r}}H_{\alpha}=0\in \fa.\]
	\end{lemma}
	\begin{proof}
		For every root hyperplane $H_\alpha$ in $\fa$, let $S_\alpha$ denote the set of all hyperplanes of $\ft$ which restrict to $H_\alpha$. Let $S$ denote the set of all hyperplanes in $\ft$ which contain $\fa$. Note that for every root hyperplane $H\subset \ft$, $H\cap \fa$ is either a root hyperplane in $\fa$ or is all of $\fa$. Hence, $S$ and $S_\alpha$ as $\alpha$ varies gives a partition of all root hyperplanes of $\ft$. We conclude that
		\[
		\bigcap_{\alpha\in \Phi_r}H_\alpha = \fa\cap \bigcap_{\alpha\in \Phi_r}H_\alpha\subset \left(\bigcap_{H\in S} H\right)\cap\bigcap_{\alpha}\left(\bigcap_{H\in S_\alpha} H\right) = \bigcap_{H\subset \ft\colon \text{ root hyperplane}} H= \{0\}\qedhere
		\]
	\end{proof}

	\begin{proposition}
		\label{proposition: Above the hyperplanes}
		Let $(G,\theta, H)$ be a symmetric pair, and let $Y = \cap_\alpha H_\alpha\subset \fa$ be the intersection of all root hyperplanes in $\fa$, so that $Y = \fa^{W_\fa}$ when $\mathrm{char}(k)>|W_\fa|$. There is then an isomorphism of stacks over $\pi(Y)$
		\[[\fp/H]\times_{\fc}\pi(Y)\cong \cN_{\fp}/H\times \pi(Y)\]
		Restricting to regular elements gives:
		\[[\fp^{\reg}/H]\times_{\fc}\pi(Y)\cong \cN_{\fp}^{\reg}/H\times \pi(Y)\]
	\end{proposition}
	
	\begin{proof}
		Let $(G_\ad,\theta_{\ad},H/Z_H(G))$ be the associated adjoint symmetric pair. Then, from the short exact sequence \eqref{eqn: ses p to pad}, we have
		\[
		[\fp/H]\simeq \mathfrak{z}_-\times [\fp_L/H_L]
		\]
		By Lemma \ref{proposition: Above the hyperplanes} and this decomposition, we have that $\pi(Y) \simeq \mathfrak{z}_-$, and the result follows.
	\end{proof}
	
	We conclude this section by giving a description of the regular quotient $\fp^\reg\myfatslash H$ for $G$ simple. Recall from Theorem \ref{theorem : Description of regular quotient gluing not explicit} that it suffices to describe the explicit gluing on Kostant-Rallis sections. By Proposition \ref{prop: description of regular quotient: reduction to Levi's} and Proposition \ref{proposition: Above the hyperplanes} the number of sheets of $\fp^\reg\myfatslash H$ over the image of $x\in \fa$ is determined by the regular $H_L$-orbits of the nilpotent cone for $L = G_x$. This is determined by the list in Proposition \ref{prop:table of nilpotent orbits}. For all simple groups except $\SO_n\times \SO_n\subset \SO_{2n}$, there are at most 2 regular nilpotent orbits, so it suffices to describe only the number of sheets in fibers of the map $\fp^{\reg}\myfatslash H\to \fp^{\reg}\git H$, as is made precise in Theorem \ref{theorem: gluing pattern}. For the $\SO_n\times \SO_n\subset \SO_{2n}$ case, one also needs to compute the gluing pattern of the 4 sheets at the origin as they degenerate.  This case is discussed in Example \ref{ex son son so2n}.
	
	We summarize this as follows:
	
	\begin{thm}
		\label{theorem: gluing pattern}
		Let $(G,\theta, H)$ be a symmetric pair corresponding to a simple group $G$,such that $(G,H)\neq (SO_{4n},SO_{2n}\times SO_{2n})$ is not the split symmetric pair of type $D_{2n}$.  Let $U=\cup_{L}U_{L}$, where $U_{L}$ as in Definition \ref{definition: UL and sundry} and $L$ ranges over the descendants of $(G,\theta,H)$, such that $\cN_{\fp_{L}}^\reg$ has a single regular $H_{L}$-orbit\footnote{We note that this can be worked out using Proposition \ref{prop:table of nilpotent orbits} together with, if necessary, computing $\pi_{0}(H_{L})$ and its action on irreducible components of $\cN_{\fp_{L}}$.}.  We then have that $\fp^\reg\myfatslash H\cong \fa\git W_{\fa}\coprod_{U}\fa\git W_{\fa}$, and this identification is $\bbG_{m}$-equivariant.
	\end{thm}
	
	\begin{remark}
		\label{remark: single orbit then just get the GIT quotient}
		We note that if $\cN_{\fp}$ has one irreducible component, then by Proposition \ref{prop:table of nilpotent orbits} it follows $\fp^\reg\myfatslash H\cong \fa\git W_{\fa}$.
	\end{remark}
	
	\begin{proof}[Proof of Theorem \ref{theorem: gluing pattern}]
		This follows immediately from Theorem \ref{theorem : Description of regular quotient gluing not explicit}.
	\end{proof}
	
	We note that in the diagonal case of $(G,H)=(G_{1}\times G_{1}, G_1)$ there is no non-separated structure as shown in Example \ref{ex: non separated structure diagonal case}.  The following proposition shows that the regular semisimple locus is always in the open set $U$ of Theorem \ref{theorem: gluing pattern}, i.e. the map $\fp^\reg\myfatslash H\to \fp\git H$ is an isomorphism over the regular, semisimple locus.
	
	\begin{prop}
		\label{prop: one orbit on regular semisimple locus}
		The map $\fp^{\reg}\myfatslash H\to \fp^{\reg}\git H\simeq \fa\git W_\fa$ is an isomorphism over $\fa^\reg\git W_\fa$.
	\end{prop}
	\begin{proof}
		Follows from Lemma \ref{lem: K-conjugacy in the reg,ss locus}.
	\end{proof}
	
	\begin{thm}
		\label{thm: description of U as complement of subdivisor!}
		Let $(G,\theta, H)$ be a symmetric pair for a simple group $G$ which is not the split symmetric pair of type $D_{2n}$. Then $\fp^\reg\myfatslash H\cong \fa\git W_{\fa}\coprod_{U}\fa\git W_{\fa}$ where $U$ is the complement of a closed subvariety which is the union of intersections of root hyperplanes.
	\end{thm}
	
	\begin{proof}
		This follows immediately from Theorem \ref{theorem: gluing pattern} and Proposition \ref{prop: one orbit on regular semisimple locus}.
	\end{proof}
	
	Finally, we introduce an alternative method of encoding the data of the regular quotient using a minimal descendant. We will use this to describe concisely the regular quotient in several examples of interest.
	
	\begin{definition}
		\label{def: minimal descendants}
		Let $(L,\theta,H_L) = (G_x,\theta,H_x)$ be a descendant of $(G,\theta,H)$. We say that $(L,\theta,H_L)$ is a \emph{minimal descendant} of $(G,\theta,H)$ if $(L,\theta,H_L)$ is minimal (with respect to inclusion) with the property that the isomorphism 
		\[p_{L}\colon (\fp^{\reg}_{L}\myfatslash H_{L})|_{V_L}\xrightarrow{\sim} \varphi_L^*(\fp^{\reg}\myfatslash H)|_{V_L}\]
		of Proposition \ref{prop: description of regular quotient: reduction to Levi's} extends to an isomorphism
		\[
		(\fp^{\reg}_{L}\myfatslash H_{L})\xrightarrow{\sim} \varphi_L^*(\fp^{\reg}\myfatslash H)
		\]
		over all of $\fc_L$.
		
		In this case, we call the Satake diagram of $(L,\theta,H_L)$ a \emph{minimal Satake diagram} of $(G,\theta, H)$.
	\end{definition}

	\subsection{Examples}
	\label{subsections: Examples regular quotient}
	
	\begin{ex}
		\label{ex: non separated structure diagonal case}
		Consider the diagonal case $G_{1}\overset{\Delta}{\subset}G_{1}\times G_{1}$ from Example \ref{example: complex case}. In this case, we have an isomorphism of stacks
		\[
		\fp/G_{1}\to \fg_{1}/G_{1}
		\]
		by projecting onto the first factor. The latter quotient is well studied over the regular locus. In particular, it is shown in \cite{donagi2002gerbe, ngo2010lemme} that the map
		\[
		\fg_{1}^{\reg}/G_{1}\to \fg_{1}\git G_{1}
		\]
		is a gerbe banded by the descent of $I_{G_{1}}^{\reg}$ to $\fg_{1}\git G_{1}$. The regular quotient $\fg^{\reg}_{1}\myfatslash G_{1}$ is therefore just the GIT quotient $\fg_{1}\git G_{1}$. Note that, as the Satake diagram of $G_1\overset{\Delta}{\subset}G_{1}\times G_{1}$ is given by two copies of the Dynkin diagram for $G_1$ with all nodes shaded white and connected in pairs (see figure \ref{fig:satake_diagonal}), the descendants of $G_{1}\overset{\Delta}{\subset}G_{1}\times G_{1}$ are again given by the diagonal involution on a product $L_1\times L_1$ for $L_1$ a Levi of $G_1$. 
		\begin{figure}
			\centering
			\begin{tikzpicture}
				\draw[gray!70, line width = .5mm] (-2.52,0.26) -- (-2.52,1.38);
				\draw[gray!70, line width = .5mm] (-1.13,0.26) -- (-1.13,1.38);
				\draw[gray!70, line width = .5mm] (2.52,0.26) -- (2.52,1.38);
				\draw[gray!70, line width = .5mm] (1.13,0.26) -- (1.13,1.38);
				\node at (0,1.5) {\dynkin[scale=4,line width=5mm]{A}{oo...oo}};
				\node at (0,0) {\dynkin[scale=4,line width=5mm]{A}{oo...oo}};
			\end{tikzpicture}
			\caption{Satake diagram for diagonal case $G_1\overset{\Delta}{\subset}G_{1}\times G_{1}$, when $G_1$ has type $A$.}
			\label{fig:satake_diagonal}
		\end{figure}
	\end{ex}

	\begin{figure}
		\centering
		\begin{tabular}{|m{7cm}|m{5cm}|}
			\hline
			Symmetric Pair &  Satake Diagram \\
			\hline
			\renewcommand{\arraystretch}{3}
			$(\SL_{2n}, \SO_{2n})$ &  \begin{tikzpicture}
				\node at (0,0) {\dynkin[scale=2,line width=5mm]{A}{oo..oo}};
			\end{tikzpicture} \\
			$(\SL_{2n}, S(\GL_n\times \GL_n))$ & \begin{tikzpicture}
				\dynkin[scale=2] A{IIIb};
			\end{tikzpicture} \\
			$(\SO_{2n+1},\SO_{2m}\times \SO_{2(n-m)+1})$, $2m\leq n$ & \begin{tikzpicture}
				\node at (0,0) {\dynkin[scale=2] BI};
				\node at (-.35,-.3) {$m$};
			\end{tikzpicture} \\
			$(\Sp_{2n},\GL_n)$ & \begin{tikzpicture}
				\dynkin[scale=2] CI;
			\end{tikzpicture} \\
			$(\SO_{2n},\SO_{n}\times \SO_{n})$ & \begin{tikzpicture}
				\dynkin[scale=2] D{Ic};
			\end{tikzpicture} \\
			$(\SO_{2n},\SO_{2m}\times \SO_{2(n-m)})$, $2m<n-1$ & \begin{tikzpicture}
				\node at (0,0) {\dynkin[scale=2] D{Ia}};
				\node at (-.17,-.3) {$m$};
			\end{tikzpicture}\\
			$(\SO_{2n},\SO_{n-1}\times \SO_{n+1})$, $n$ odd & \begin{tikzpicture}
				\dynkin[scale=2] D{Ib};
			\end{tikzpicture}\\
			$(\SO_{4n},\GL_{2n})$ & \begin{tikzpicture}
				\dynkin[scale=2] D{IIIa};
			\end{tikzpicture}\\
			$(E_7,A_7)$ & \begin{tikzpicture}
				\dynkin[scale=2] EV;
			\end{tikzpicture}\\
			$(E_7,E_6\times k^\times)$ & \begin{tikzpicture}
				\dynkin[scale=2] E{VII};
			\end{tikzpicture}\\
			\hline
		\end{tabular}
		\caption{Table of Satake diagrams for all simple symmetric pairs with nontrivial regular quotient.}
		\label{fig:table of satake diagrams}
	\end{figure}

	\begin{example}
		\label{example: structure of regular quotient for U(n,n) case}
		We revisit Example \ref{example:U(n,n) case}, given by the symmetric pair $H =K = \GL_n\times\GL_n\subset \GL_{2n} = G$. The Satake diagram can be read off of the table in Figure \ref{fig:table of satake diagrams}. In particular, the descendants of $(G,\theta,H)$ have Satake diagrams which are products of the diagonal case (Example \ref{ex: non separated structure diagonal case}) and smaller copies $\GL_m\times\GL_m\subset \GL_{2m}$ for $m< n$. It is easy to see that $H\cap G_x$ is connected in the latter case, so we get 2 sheets whenever we have a factor of $G_x$ isomorphic to $\GL_m\times\GL_m\subset \GL_{2m}$. In other words, we summarize:
	\end{example}
	
	\begin{prop}
		\label{prop: minimal satake for U(n,n)}
		The minimal Satake diagram (in the sense of Definition \ref{def: minimal descendants}) for the symmetric pair $\GL_n\times\GL_n\subset \GL_{2n}$ is given by the Satake diagram corresponding to the Levy $\GL_1\times \GL_1\subset \GL_2$, i.e.
		
		\begin{center}
			\begin{tikzpicture}
				\dynkin[scale=3] A{IIIb};
				\draw (.35,-.3) node[cross,red] {};
				\draw (.35,.3) node[cross,red] {};
				\draw (0,-.3) node[cross,red] {};
				\draw (0,.3) node[cross,red] {};
				\draw (.91,-.3) node[cross,red] {};
				\draw (.91,.3) node[cross,red] {};
			\end{tikzpicture}
		\end{center}
		where crossed out vertices are removed.
	\end{prop}
	
	We can describe this locus with explicit equations on $\fa\git W_\fa$ as well. Having a Levi factor in $G_x$ of type $\GL_m\times\GL_m\subset \GL_{2m}$ is equivalent to the image of the set of simple roots $\Sigma_x$ under the restriction map containing a long root in the restricted root system. Hence, the gluing locus is described by the vanishing of the long root in $\fa\git W_\fa$. More explicitly, if one takes
	\[
	\fa = \left\{\begin{pmatrix}
		0 & \delta \\
		\delta & 0
	\end{pmatrix}\colon \delta = \mathrm{diag}(\delta_1,\dots, \delta_n)\text{ diagonal}\right\}
	\]
	Then, we have: 
	
	\begin{proposition}
		\label{proposition: gluing pattern U(n,n)}
		For $\fa$ as above, let $V\subset \fa$ be the subscheme that is the complement of the union of the hyperplanes $\delta_i = 0$. Then, we have $\fp^{\reg}\myfatslash K\cong \fa\git W_{\fa}\coprod_{U}\fa\git W_{\fa}$, where  $U:=V\git W_{\fa}\subset \fa\git W_{\fa}$.
	\end{proposition}
	
	\begin{proof}
		Follows immediately from the above computations and Theorem \ref{thm: description of U as complement of subdivisor!}.
	\end{proof}

	\begin{ex}
		\label{example: Regular quotient for SO in SL}
		Consider the split symmetric pair $\mathrm{PSO}_n\subset \PGL_n$. If $n$ is odd, there is only one regular nilpotent $H$-orbit in $\cN_\fp$ and the regular quotient and GIT quotient agree. We will assume therefore that $n$ is even. We have $\fa = \ft$ is the diagonal Cartan inside $\fp = \mathfrak{sym}_n$ (symmetric $n\times n$ matrices). The Satake diagram is the same as the Dynkin diagram of type $A_{n-1}$, with all vertices shaded white and the involution $\iota$ being trivial. All descendants are therefore products of split symmetric pairs of type $A$.
		
		It remains to check how $\pi_0(H_x)$ acts on Levi factors for each piece of the Satake diagram. Suppose that the Satake diagram of $(G_x,\theta|_{G_x})$ has $s$ disjoint components of ranks $n_1,\dots,n_s$, with indices $n_1,\dots, n_\ell$ odd and $n_{\ell+1},\dots, n_s$ even. Then, it is straightforward to check that:
		\begin{itemize}
			\item The nilpotent cone $\cN_\fp$ of $(G_x,\theta)$ has $2^\ell$ components;
			\item The subgroup $H_x$ is of the form
			\[
			H_x = P\left( \prod_{j=1}^r O_{n_j+1} \times \{\pm 1\}^{n_0}\right)
			\]
			where $n_0 = n -s- \sum_j n_j$. Let $Q$ be the quotient of the component group $\pi_0(H_x)$ through which the action on the irreducible components of the nilpotent cone $\cN_\fp$ is faithful. Then it is straightforward to check
			\[
			|Q| = \begin{cases}
				2^{\ell-1} & \text{if }n_0=0,\; \ell = s\\
				2^\ell & \text{else} 
			\end{cases}
			\]
		\end{itemize}
	\end{ex}
	
	Hence, the 2 sheets are unglued over the image of $x\in \fa$ with Satake diagram a disjoint union of $s$ connected diagrams, each of odd rank, such that the sum of the ranks of these components is $n-s$. In fact, we may conclude:
	\begin{prop}
		\label{prop: minimal satake for split type A}
		The minimal Satake diagram (in the sense of Definition \ref{def: minimal descendants}) for the split symmetric pair of type $A_{n-1}$ is the Satake diagram obtained by removing every other node from the split type $A_{n-1}$ Satake diagram, i.e.
		
		\begin{center}
			\begin{tikzpicture}
				\dynkin[scale=3]{A}{ooo...ooo};
				\draw (1.61,0) node[cross,red] {};
				\draw (.35,0) node[cross,red] {};
			\end{tikzpicture}
		\end{center}
	\end{prop}
	
	We conclude by describing the gluing locus using equations on the diagonal matrices $\fa = \{\mathrm{diag}(\delta_1,\dots,\delta_n)\}\subset \mathfrak{pgl}_n$. The closure of the set of $x\in \fa$ with Satake diagram as in Proposition \ref{prop: minimal satake for split type A} is described in $\fa$ by the $W$ orbit of closed subvariety described by the equations 
	\[
	\delta_1 = \delta_2, \; \delta_3 = \delta_4, \; \dots\; , \delta_{n-1} = \delta_n
	\]
	We conclude:
	\begin{prop}
		\label{prop: non-separated structure for split form of type A}
		Let $n$ be even. With notation for $\fa$ as above, let $Z_1\subset \fa$ denote the vanishing of the equations
		\[
		\delta_1 = \delta_2, \; \delta_3 = \delta_4, \; \dots\; , \delta_{n-1} = \delta_n
		\]
		and let $Z = \cup_{w\in W}w(Z_1)\subset \fa$. Put $V = \fa\setminus Z$ and $U = V\git W_\fa\subset \fa\git W_\fa$. Then for the split symmetric pair of type $A_{n-1}$, we have 
		\[
		\fp^\reg\myfatslash K\simeq \fa\git W_\fa\coprod_U\fa\git W_\fa.
		\]
	\end{prop}
	\begin{proof}
		This follows by Theorem \ref{thm: description of U as complement of subdivisor!} and the computations above.
	\end{proof}
	
	In the above proposition, we have
	\[
	\fa\git W_\fa = \ft\git W = \Spec k[a_2,\dots, a_n],
	\]
	where $a_i$ is the $i$-th elementary symmetric polynomial in the components $\delta_i$ in $\mathrm{diag}(\delta_1,\dots, \delta_n)\in \fa$.
	The open subset $V$ in Proposition \ref{prop: non-separated structure for split form of type A} is the complement of a closed subset of $\fa\git W_\fa$, which is described in low dimensions as follows:
	\begin{itemize}
		\item When $n=2$: $a_2=0$;
		\item When $n=4$: $a_3=0,\; a_4=a_2^2$;
		\item When $n=6$: $16a_4=a_2^2$, $2a_5=a_2a_3$, $4a_6=a_3^2$.
	\end{itemize}

	\begin{ex}
		\label{ex son son so2n}
		Consider the split symmetric pair $P(\SO_{2n}\times \SO_{2n})\subset \mathrm{PSO}_{4n}$ of Example \ref{example: so x so}. Recall that this is the unique split symmetric pair of type $D_{2n}$, which is simple and simply-laced when $n\geq 2$. Recall from Proposition \ref{prop:table of nilpotent orbits} that for $n\geq2$ this is the unique family of simple symmetric pairs up to isogeny for which there are 4 regular nilpotent orbits.
		
		We assume that $n\geq 2$, and consider the split symmetric pair of type $D_{2n}$. The Satake diagram is given in Figure \ref{fig:table of satake diagrams}, namely the usual Dynkin diagram of type $D_{2n}$ with all white nodes. Subdiagrams of this are Satake diagrams of the split involutions of type $D_{n_0}$ for $n_0<2n$ or of type $A_{m}$ for $m\leq 2n-1$. Fix some $x\in \fa$ and consider the descendant at $x$. Let $Q$ denote the quotient through which $\pi_0(H_x)$ acts faithfully on the components of $\cN_\fp^\reg$. We can verify the table in Figure \ref{fig:descendants of split Dn}.
	\end{ex}
	
	\begin{figure}
		\begin{adjustbox}{rotate=90,center}
			\begin{tabular}{|m{4.5cm}|m{5.5cm}|m{2.3cm}|m{4.8cm}|m{4.3cm}|}
				\hline
				Type of $G_x$ & Satake diagram for $(G_x,\theta|_{G_x})$ & $\#(\cN_{\fp_x}^\reg/K_{G_x})$ & $\#Q$ & $\#(\cN_{\fp_x}^\reg/H_x)$ \\
				\hline
				$\displaystyle D_{n_0}\times \prod_{i=1}^a A_{n_i}\times \prod_{i=1}^b A_{m_i}$, for $n_0$,$n_i$ odd, $m_i$ even & \begin{tikzpicture}
					\dynkin[scale=2] D{ooo...oooooo};
					\draw (.7,0) node[cross,red] {};
					\draw (1.26,0) node[cross,red] {};
				\end{tikzpicture} & $2^{a+1}$ & $\begin{cases}
					2^{a} & \text{if $b=0$ and }\\
					& \;\;n=n_0+a+\sum_i n_i \\
					2^{a+1} & \text{else}
				\end{cases}$ & $\begin{cases}
					2 & \text{if $b=0$ and }\\
					&\;\;n=n_0+a+\sum_i n_i \\
					1 & \text{else}
				\end{cases}$ \\
				$\displaystyle D_{n_0}\times \prod_{i=1}^a A_{n_i}\times \prod_{i=1}^b A_{m_i}$, for $n_i$ odd, $n_0$,$m_i$ even &\begin{tikzpicture}
					\dynkin[scale=2] D{ooo...oooooo};
					\draw (.7,0) node[cross,red] {};
					\draw (1.61,0) node[cross,red] {};
				\end{tikzpicture} & $2^{a+2}$ & $\begin{cases}
					2^{a} & \text{if $b=0$ and }\\
					& \;\;n=n_0+a+\sum_i n_i \\
					2^{a+1} & \text{else}
				\end{cases}$ & $\begin{cases}
					4 & \text{if $b=0$ and }\\
					& \;\;n=n_0+a+\sum_i n_i \\
					2 & \text{else}
				\end{cases}$ \\
				$\displaystyle \prod_{i=1}^a A_{n_i}\times \prod_{i=1}^b A_{m_i}$, for $n_i$ odd, $m_i$ even &\begin{tikzpicture}
					\dynkin[scale=2] D{ooo...oooooo};
					\draw (.7,0) node[cross,red] {};
					\draw (1.61,0) node[cross,red] {};
					\draw (2.49,-.3) node[cross,red] {};
				\end{tikzpicture} &  $2^a$ & $\begin{cases}
					2^{a-1} & \text{if $b=0$ and }\\
					& \;\;n=a+\sum_i n_i \\
					2^{a} & \text{else}
				\end{cases}$  & $\begin{cases}
					2 & \text{if $b=0$ and }\\
					&\;\;n=n_0+a+\sum_i n_i \\
					1 & \text{else}
				\end{cases}$\\
				$A_1\times A_1\times \displaystyle \prod_{i=1}^a A_{n_i}\times \prod_{i=1}^b A_{m_i}$, for $n_i$ odd, $m_i$ even & \begin{tikzpicture}
					\dynkin[scale=2] D{ooo...oooooo};
					\draw (.7,0) node[cross,red] {};
					\draw (1.26,0) node[cross,red] {};
					\draw (2.31,0) node[cross,red] {};
				\end{tikzpicture} &  $2^{a+2}$ & $\begin{cases}
					2^{a} & \text{if $b=0$ and}\\
					& \;\; n=a+\sum_i n_i \\
					2^{a+1} & \text{else}
				\end{cases}$ & $\begin{cases}
					4 & \text{if $b=0$ and }\\
					& \;\;n=n_0+a+\sum_i n_i \\
					2 & \text{else}
				\end{cases}$ \\
				\hline
			\end{tabular}   
		\end{adjustbox}
		\caption{Table of all possible descendants of split type $D_{2n}$.}
		\label{fig:descendants of split Dn}
	\end{figure}     
	
	We summarize in the following:
	
	\begin{prop}
		\label{prop: minimal satake for split type D}
		The minimal Satake diagram (in the sense of Definition \ref{def: minimal descendants}) for the split symmetric pair of type $D_{2n}$ is obtained by deleting every second node in the Satake diagram of the symmetric pair, i.e.
		\begin{center}
			\begin{tikzpicture}
				\dynkin[scale=2] D{ooo...oooooo};
				\draw (.35,0) node[cross,red] {};
				\draw (1.61,0) node[cross,red] {};
				\draw (2.31,0) node[cross,red] {};
			\end{tikzpicture}
		\end{center}
	\end{prop}
	
	To avoid messy computations, we do not summarize the gluing pattern in explicit equations for this case. We note that all sheets are glued away from a codimension $2$ locus and there are 4 sheets over a codimension $n$ locus.

	\begin{ex}
		\label{ex: regular quotient SO 2m times SO 2n+1-2m}
		Consider the case of $\SO_{2m}\times \SO_{2n+1-2m}\subset \SO_{2n+1}$, $2m\leq n$; see example \ref{ex: so x so odd}. Recall that the restricted root system in this case is simple of type $B_m$.
		
		The Satake diagram for this symmetric pair can be read from the table in figure \ref{fig:table of satake diagrams}; in particular, this pair is not $\theta$-quasisplit and has compact (black shaded) nodes in the Satake diagram. For any $x\in \fa$, the descendant at $x$ must have Satake diagram given by a product of split Satake diagrams of type $A$ and exactly one of the same type as the original symmetric pair, i.e. one of the form $\SO_{2(m-n+n_0)}\times \SO_{2(n-m)+1}\subset \SO_{2n_0+1}$ for $n_0\leq n$. Suppose that we have $a$ copies of odd type $A$, with ranks $n_1,\dots, n_a$ and $b$ copies of even type $A$. Then, the nilpotent cone, $\cN_{\fp_x}$, of the descendant at $x$ has $2^{a+1}$ irreducible components, and the action of $\pi_0(H_x)$ is through a faithful action of a group of size $2^a$ if $b=0$ and $n = n_0+a+\sum_i n_i$ and a group of size $2^{a+1}$ else. Hence, we deduce the following description:
	\end{ex}
	
	\begin{prop}
		\label{prop:minimal satake for non-qs case}
		The minimal Satake diagram for the symmetric pair $\SO_{2m}\times \SO_{2n+1-2m}\subset \SO_{2n+1}$, $2m<n$, is given by deleting every second white vertex. For example, when $m$ is even, the diagram is depicted by
		\begin{center}
			\begin{tikzpicture}
				\dynkin[scale=3] B{oooo...ooo...ooo};
				\fill[black] (3.22,0) circle (.55mm);
				\fill[black] (3.57,0) circle (.55mm);
				\fill[black] (2.87,0) circle (.55mm);
				\fill[black] (2.31,0) circle (.55mm);
				\draw (1.96,0) node[cross,red] {};
				\draw (1.05,0) node[cross,red] {};
				\draw (.35,0) node[cross,red] {};
				\draw (1.96,-.13) node {$m$};
			\end{tikzpicture}
		\end{center}
		while for $m$ odd, it is
		\begin{center}
			\begin{tikzpicture}
				\dynkin[scale=3] B{oooo...oooo...ooo};
				\fill[black] (3.92,0) circle (.55mm);
				\fill[black] (3.57,0) circle (.55mm);
				\fill[black] (3.22,0) circle (.55mm);
				\fill[black] (2.66,0) circle (.55mm);
				\draw (1.96,0) node[cross,red] {};
				\draw (1.05,0) node[cross,red] {};
				\draw (.35,0) node[cross,red] {};
				\draw (2.31,-.13) node {$m$};
			\end{tikzpicture}
		\end{center}
	\end{prop}
	
	We finish with a description of the explicit gluing locus in $\fa$. Let $\fa$ have coordinates $\delta_i$, with the simple roots on $\fa$ being given by $i(\delta_j-\delta_{j+1})$, $j=1,\dots, n-1$, and $i\delta_n$. (Here, $i$ denotes the imaginary unit.) Let $Z_1\subset \fa$ be the closed subvariety of $\fa$ defined by the equations
	\[
	\delta_1-\delta_2 = \delta_3-\delta_4 = \cdots =\delta_{m-1}-\delta_m = 0
	\]
	if $m$ is even and 
	\[
	\delta_1-\delta_2 = \delta_3-\delta_4 = \cdots =\delta_{m-2}-\delta_{m-1} = \delta_m=0
	\]
	if $m$ is odd. Let $Z = \cup_{w\in W}w(Z_1)$ be the union of the $W$ orbits of $Z_1$, and let $U = \fa\setminus Z$. Then we have the following:
	\begin{prop}
		\label{prop:gluing for so x so odd}
		With notation as above, for the symmetric pair $\SO_{2n+1}/\SO_{2m}\times\SO_{2n+1-2m}$, let $V =U\git W_\fa\subset \fa\git W_\fa$ be the image of $U$ in $\fa\git W_\fa$. Then, the regular quotient is given by 
		\[
		\fp^\reg\myfatslash H = \fa\git W_\fa\coprod_{U}\fa\git W_\fa
		\]
	\end{prop}
	
	In particular, the non-separated locus (the set $Z$) is of codimension $\lfloor \frac{m}{2}\rfloor$ in $\fa\git W_\fa$.

	\section{Galois Description of \texorpdfstring{$J$}{J} for \texorpdfstring{$\theta$}{theta}-Quasisplit Symmetric Pairs}
	\label{section: GS cover}
	
	For the results of this section, we restrict to the case of $\theta$-quasisplit symmetric pairs. Moreover, for the results of this section, we assume that the characteristic of the $k$ does not divide the order of the Weyl group $W$. We will denote $\fc = \fp\git H\simeq \fa\git W_\fa$. In particular, by Proposition \ref{proposition: equivalent characterizations of quasi-split}, we assume that the regular centralizer group scheme $I^\reg = I^\reg_H\to \fp^\reg$ is abelian. In this case, we have the following descent statement.
	
	\begin{prop}
		\label{prop: descent of I to J}
		If $(G,\theta,H)$ is a $\theta$-quasisplit symmetric pair, then the regular centralizer group scheme $I_H^\reg$ descends to a smooth group scheme $J$ over $\fc$ such that there is a canonical isomorphism $\chi^*J|_{\fp^\reg}\simeq I_H^\reg$. Moreover, there is a canonical extension of this isomorphism to a map
		\[
		\chi^*J\to I_H
		\]
	\end{prop}
	\begin{proof}
		First note that whether a symmetric pair $(G,\theta,H)$ is $\theta$-quasisplit is independent of $H$, and hence, $(G,\theta,N_G(K))$ is also a symmetric pair. Therefore, by Proposition \ref{proposition: equivalent characterizations of quasi-split}, part (c), the regular centralizer group $I_{N_G(K)}^\reg$ is also abelian. Moreover, $N_G(K)$ normalizes $H$, and so we can equip $I^\reg = I_H^\reg$ with an $N_G(K)$ equivariant structure. From here, the proof of \cite[Lemma 2.1.1]{ngo2010lemme} follows verbatim, using the transitivity of the $N_G(K)$ action in Lemma \ref{lem: N acts transitively on fibers} and \cite[Lemma 6.32]{levy}.
	\end{proof}
	
	Note that this is the descent of the group scheme that was also denoted by $J$ over the regular quotient, and hope this does not cause confusion.  We also note that this is clearly $\bbG_{m}$-equivariant.
	
	Our goal in this section is a Galois description of the regular centralizer group scheme $J$ for $\theta$-quasisplit symmetric pairs. 
	
	Following the skeleton of Section 2.4 in \cite{ngo2010lemme}, one might hope to find a finite, flat cover 
	\[
	\pi\colon \wt{\fc}\to \fc,
	\]
	with group $\wt{W}$ acting on the centralizer group scheme $C=C_H(\fa)$, such that $J$ embeds as an open subgroup scheme of the Weyl restriction $\mathrm{Res}_\fc^{\tilde{\fc}}(\wt{\fc}\times C_H)^{\wt{W}}$. Such a description was the objective of Section 5.1 of \cite{gppn} and Section 4 of \cite{leslie} using cameral covers modeled on the flat cover $\fa\to \fc$, which is finite flat with group $W_\fa$. However, this particular cover is inadequate for general results on regular centralizers, as illustrated by the following example.
	
	\begin{ex}
		\label{ex: sl3 case}
		Consider the symmetric pair on $G = \SL_3$ given by the involution $\theta$ conjugating by $\begin{pmatrix} 1 & \\ & -I_2 \end{pmatrix}$ with $H = S(\bG_m\times \GL_2)$ embedded block diagonally. We choose 
		\[
		\fa = \left\{ \begin{pmatrix} & \delta & 0 \\ \delta & & \\ 0 & &  \end{pmatrix} \colon \delta\in k \right\} \simeq \bA^1,
		\]
		so that its centralizer is
		\[
		C_H = C_H(\fa) = \left\{ \begin{pmatrix} x &  &  \\ & x & \\  & & y \end{pmatrix} \colon x^2y = 1 \right\}
		\]
		One can verify that $W_\fa\simeq \{\pm 1\}$ acts trivially on $C$. Hence, $\mathrm{Res}^\fa_\fc(C\times \fa)^{W_\fa} = C\times\fc$ is the constant group scheme. It is trivial to compute the fiber of $\pi_0(J_H)$ at 0 is $\mu_3$ while the fiber of $\mathrm{Res}^\fa_\fc(C\times \fa)^{W_\fa} = C\times \fc$ is connected everywhere. Hence, $J_H$ cannot be an open subscheme of $\mathrm{Res}^\fa_\fc(C\times \fa)^{W_\fa}$.
	\end{ex}
	
	\begin{remark}
		\label{rmk: specific errors in literature}
		Theorem 4.7 of \cite{leslie} and Theorem 21 of \cite{gppn} fail for Example \ref{ex: sl3 case}. 
	\end{remark}
	
	We will remedy this by working with a cover with several irreducible components. To do so, we will begin by introducing the canonical torus of a symmetric space $X$ and its relationship to the canonical torus of $G$. We will then correct the subtle error in the existing literature in section \ref{sec: revising Leslie's cover}. We show how one can still get a Galois description of the regular centralizer group scheme $J$ for quasisplit symmetric spaces by allowing the flat cover to have mutliple irreducible components. In section \ref{sec: relative dual centralizers}, we will then use this description to give a more precise description of a closely related group scheme. This description plays an essential role in our work on a relative Dolbeault geometric Langlands in \cite{me2}.

	\subsection{Canonical Tori}
	\label{sec: canonical torus}
	
	This section reviews the relevant results of \cite{leslie}. Let $\canT$ denote the canonical torus of the group $G$. Let $X$ be a symmetric space (or more generally a spherical variety, i.e. a $G$ variety with open Borel orbit). choose a Borel $B\subset G$ such that $B$ has an open orbit $\mathring{X}\subset X$. Let $P(X)$ be the stabilizer of $\mathring{X}$, and let $U$ be the unipotent radical of $B$. The parabolic subgroup $P(X)$ acts on $U\backslash\!\!\backslash\mathring{X}$ through a maximal quotient $P(X)\to \canA$ for a torus $\canA$, which we call the canonical torus of $X$. The canonical torus $\canA$ inherits an action of a Weyl group $W_X$, which is in general a subquotient of the Weyl group $W$. In the case of $X = G/H$ a quasi-affine homogeneous spherical variety, $P(X) = L(X)\cdot B$ where $L(X)$ is the centralizer of the action of $G$ on a generic regular semisimple element in $\fg/\fh\subset \fg^*$, see \cite[Theorem 2.3]{knop-invariant} and \cite[Lemma 3.1]{knop-invariant}. When $X$ is further restricted to the case of a symmetric space, the canonical torus is closely related to maximal $\theta$-split tori of $G$ and $W_X$ is identified with the little Weyl group $W_\fa$, see Lemma \ref{lem: isogeny A -> AX}.
	
	Now, let $X = G/H$ be a $\theta$-quasisplit symmetric space. The parabolic $P(X)$ is then equal to the Borel $B$, and we have a projection map $\canT\to \canA$. Our choice of notation is suggestive of a link between the canonical torus $\canA$ and a maximal $\theta$-split torus $A$. Indeed, the following is a very mild generalization of \cite[Lemma 1.11]{leslie}.
	\begin{lem}
		\label{lem: isogeny A -> AX}
		Fix a maximally $\theta$-split torus $A$ of $G$ and maximally split $T\supset A$. Recall from Proposition \ref{prop: extension of K} that 
		\[
		F^* = \{a\in A\colon a^2\in Z(G)\}.
		\]
		Choose a $\theta$-stable Borel $B\supset T$ of $G$ so that we get a isomorphism $T\to \canT$, and use $B$ to define the canonical torus $\canA$. The composition
		\begin{equation}
			\label{eqn: canonical comparison}
			A\hookrightarrow T\to \canT\to \canA
		\end{equation}
		is surjective with kernel $H\cap A = H\cap F^*$. In particular, \eqref{eqn: canonical comparison} is an isogeny whenever either $H\subset G^\theta$ or the subgroup $Z_-\subset Z(G)$ on which $\theta$ acts by inversion is finite. Moreover, the little Weyl groups $W_\fa\simeq W_X$ are identified, and the isogeny above is equivariant with respect to their actions on $A$ and $\canA$, respectively.
	\end{lem}
	\begin{proof}
		The action of $P(X) = B$ on $U\backslash\!\!\backslash \mathring{X}$ factors through the canonical torus $\canT = B/[B,B] = B/U$. Identify $T\simeq \canT$ via the Borel $B$, we see that an element $t\in T$ acts trivially on $U\backslash\!\!\backslash \mathring{X}$ if and only if $t\in H$. Hence the map \eqref{eqn: canonical comparison} is identified with the map
		\[
		A\hookrightarrow T\to T/T\cap H = A/A\cap H.
		\]
		It remains to show $H\cap A = H\cap F^*$. If $h\in H\cap A$, then by Proposition \ref{prop: extension of K}, we can write $h = ak$ for $a\in F^*$ and $k\in K$. Then, since $h\in A$, also $k\in A$ so that $k\in K\cap A$. But
		\[
		K\cap A = \{a\in A\mid a^2 = 1\}\subset F^*
		\]
		so we conclude that $h\in H\cap F^*$. 
		
		Finally, the identification of the Weyl group actions can be found in \cite[Lemma 1.11]{leslie}.
	\end{proof}
	
	\begin{remark}
		\label{rmk: kernel of can map on lie algebras}
		We note that $\mathrm{Lie}(H\cap F^*) = \fz_-\cap \fh$ where $\fz_- = \mathrm{Lie}(Z_-)$ is the Lie algebra of the subgroup of $Z(G)$ on which $\theta$ acts by inversion, and hence the kernel of $\mathrm{Lie}(A)\to \canfa$ is identified with $\fz_-\cap \fh\subset \mathrm{Lie}(A)$.
	\end{remark}

	Let $\canft=\mathrm{Lie}(\canT)$, respectively $\canfa = \mathrm{Lie}(\canA)$, be the Lie algebra of the canonical torus $\canT$, resp. $\canA$ (note that as $A$ and $\canA$ are isogeneous by lemma \ref{lem: isogeny A -> AX} this should not create any ambiguity). We end this section with a result of Leslie, which implies a canonical splitting of the morphism $\canft\to \canfa$.
	
	\begin{lem}(\cite[Prop 1.12]{leslie})
		\label{lem: canonical involution}
		There is a canonical involution $\theta_{can}$ on the canonical torus $\canT$ as well as on its Lie algebra $\canft$. When $H\subset G^\theta$ or when the subgroup $Z_-\subset Z(G)$ (defined in Lemma \ref{lem: isogeny A -> AX}) is finite, the map $\canft\to \canfa$ restricts to an isomorphism on the $(-1)$-eigenspace of $\theta_{can}$.
	\end{lem}
	
	The construction of $\theta_{can}$ involves first choosing a $\theta$-split Borel $B$ (i.e. $B$ such that $B\cap \theta(B)$ is a torus) and considering the $\theta$ action on $\canT = B/[B,B]$.

	\begin{definition}
		\label{def: A1}
		We will denote by $\canfa_1$ the $(-1)$-eigenspace of the canonical involution $\theta_{can}$ acting on $\canft$ and by $\canA_1$ the subtorus of $\canT$ on which $\theta_{can}$ acts by inversion.
	\end{definition}

	\subsection{Description of $J$}
	\label{sec: revising Leslie's cover}
	
	In this section we provide a description of $J$ via Weil restriction. This is an analogue of the work of Donagi--Gaitsgory for the usual Hitchin fibration \cite{donagi2002gerbe}, and corrects the earlier works \cite{leslie,gppn}. We will use this modified result to prove a description of a closely related group scheme in Section \ref{sec: relative dual centralizers}.
	
	We let $\fc_G = \fg\git G$ denote the GIT base for $G$ and $\fc = \fp\git H$ the GIT base for $X = G/H$.  Consider the base change
	\[
	\xymatrix{
		\canft\times_{\fc_G}\fc\ar[r]\ar[d] & \canft\ar[d] \\
		\fc\ar[r] & \fc_G
	}
	\]
	We note that the map 
	\[
	\canft\times_{\fc_G}\fc\to \fc
	\]
	is finite and flat of degree $|W|$ as it is the base change of the map $\canft\to \fc_G$. The variety $\canft\times_{\fc_G}\fc$ is not irreducible. In fact, following \cite{leslie}, we can describe its components as follows. Recall that $\canfa_1\subset \canft$ denotes the image of the canonical splitting of Lemma \ref{lem: canonical involution}. For every $\nu\in W/W_X$, choose a lift $\wt{\nu}\in W$. We may consider the subset $\canfa_\nu\subset\canft$ given by $\wt{\nu}(\canfa_1)$. As $W_X$ fixes $\canfa_1$, this is independent of the choice of lift $\wt{\nu}$. Then, we have
	\begin{lem}
		\label{lem: t is a union of a's}
		We can express the fiber product as a finite union
		\[
		\canft\times_{\fc_G}\fc = \bigcup_{\nu\in W/W_X}\canfa_\nu
		\]
		and the subvarieties $\canfa_\nu$ are the distinct (but not disjoint) components of this fiber product.
	\end{lem}
	\begin{proof}
		Follows from an identical argument to \cite[Lemma 3.3]{leslie}.
	\end{proof}
	
	\begin{cor}
		The morphism 
		\begin{equation}
			\label{eqn: Lie alg of can map}
			\canft\times_{\fc_G}\fc\to \canfa   
		\end{equation}
		induced by the canonical map $\canT\to \canA$ is defined over $\fc$.
	\end{cor}
	\begin{proof}
		For every $\nu$, the composition
		\[
		\fa_\nu\subset \ft\times_{\fc_G}\fc\to \fa
		\]
		is a $W_X$ equivariant isomorphism while the map 
		\[
		\fa_\nu\subset \ft\times_{\fc_G}\fc\to \fc
		\]
		is the quotient map by the $W_X$ action. As $\ft\times_{\fc_G}\fc$ is a union of these $\fa_\nu$, the result follows.
	\end{proof}
	
	Recall from Remark \ref{rmk: kernel of can map on lie algebras} that for each $\nu\in W/W_X$, the morphism \eqref{eqn: Lie alg of can map} restricts to maps $\canfa_\nu\to \canfa$ whose kernel is the $W$-stable subspace $\fz_-\cap \fh\subset \fa_\nu$. Here, $\fz_- = \mathrm{Lie}(Z_-)$ is the Lie algebra of the subgroup of $Z(G)$ on which $\theta$ acts by inversion. 
    
    Though it will not be used later in this article, it is interesting that the quotient $W/W_X$ admits canonical representatives. We take a brief detour to describe them.
	
	\begin{lem}
		\label{lem: all auts of a come from WX}
		Any map $\xi\colon \canfa\to \canfa$ which commutes with the morphisms to $\fc$ is equal to multiplication by an element of $W_X$.   
	\end{lem}
	\begin{proof}
		Let $\canfa_w = \{x\in \canfa\colon \xi(x) = w\cdot x\}$. Then, $\canfa_w$ is closed in $\canfa$ and $\canfa = \cup_{w\in W_X}\canfa_w$. We conclude that $\canfa = \canfa_w$ for some $w\in W_X$.
	\end{proof}
	
	\begin{proposition}
		\label{prop: description of W0}
		For each coset $\nu\in W/W_X$, there exists a unique lift $w_\nu\in W$ of $\nu$ such that the following diagram commutes
		\begin{equation}
			\label{eqn: W0 diagram}
			\xymatrix{
				\canfa_1\ar[r]^-{w_\nu}\ar@{^{(}->}[d] & \canft\times_{\fc_G}\fc\ar[d]^-{\mathrm{can}} \\
				\canft\times_{\fc_G}\fc\ar[r]^-{\mathrm{can}} & \canfa
			}
		\end{equation}
		Above, $\canfa_1\subset \canft\times_{\fc_G}\fc$ is the splitting induced by the canonical involution as in Definition \ref{def: A1}.
	\end{proposition}
	\begin{proof}
		Note first that all morphisms in the diagram \eqref{eqn: W0 diagram} are defined over $\fc$. Moreover, since $\fz_-\cap \fh$ is $W$ stable, for any $w\in W$, the map 
		\begin{equation}
			\label{eqn: lift of nu}
			\fa_1\xrightarrow{w}\canft\times_{\fc_G}\fc\xrightarrow{\mathrm{can}} \canfa
		\end{equation}
		has kernel $\fz_{-}\cap \fh$. In particular, identifying $\canfa_1/(\fz_{-}\cap \fh)\simeq \canfa$ using the canonical map, we may apply Lemma \ref{lem: all auts of a come from WX} to identify the map \eqref{eqn: lift of nu} with multiplication by some element $w_0$ of $W_X$ on $\canfa$. In particular, there is a unique $W_X$ translate $w_\nu = w_0^{-1}w$ of $w$ for which \eqref{eqn: lift of nu} agrees with the canonical map $\canfa_1\to  \canfa$.
	\end{proof}

	We now recenter our discussion on regular centralizers. The finite, flat map $\canft\to \fc_G$ was used in the work of Donagi-Gaitsgory to describe the regular centralizer group scheme $J_G$ \cite{donagi2002gerbe}. The following is a reformulation of this result due to Ng\^o.
	
	\begin{lem}(\cite[Lemma 2.4.6 and 2.4.7]{ngo2010lemme})
		\label{lem: Galois description of JG}
		Assume that the characteristic of $k$ does not divide the order of $W$. Let $J_G$ be the regular centralizer group scheme of $G$ acting on $\fg$. There is an open embedding 
		\begin{equation}
			\label{eqn: Galois description of JG}
			J_G\to \mathrm{Res}^{\canft}_{\fc_G}(\canT\times \canft)^W
		\end{equation}
		where $W$ acts diagonally on the product $\canT\times \canft$ and $\mathrm{Res}^{\canft}_{\fc_G}(-)$ denotes the Weil restriction along the map $\canft\to \fc_G$. Moreover, the image of this map admits a simple description:  For any root $\alpha$, let $\fh_{G,\alpha}\subset \canft$ be the $\alpha$-root hyperplane. Then \eqref{eqn: Galois description of JG} identifies $J_G$ with the subgroup of $\mathrm{Res}^{\canft}_{\fc_G}(\canT\times \canft)^W$ consisting of $S$-points
		\[
		S\times_{\fc_G}\canft\to \canT
		\]
		for which, for all roots $\alpha\in \Phi$, the composition
		\[
		S\times_{\fc_G}\fh_{G,\alpha}\to S\times_{\fc_G}\canft\to \canT\xrightarrow{\alpha}\bG_m,
		\]
		is trivial.
	\end{lem}
	
	Let $J_G|_\fc = J_G\times_{\fc_G}\fc$. The closed embedding $I^\reg\hookrightarrow I_{G}^\reg|_{\fp^\reg}$ is $H$ equivariant, and hence descends to a closed embedding $J\hookrightarrow J_{G}|_{\fc}$. Composing \eqref{eqn: Galois description of JG} with the map $J\hookrightarrow J_G|_\fc$ gives a locally closed embedding 
	\[
	J\hookrightarrow \mathrm{Res}^{\canft}_{\fc_G}(\canT\times \canft)^W|_\fc = \mathrm{Res}^{\canft\times_{\fc_G}\fc}_{\fc_G}\big(\canT\times (\canft\times_{\fc_G}\fc)\big)^W
	\]
	We describe the image of this embedding over the regular, semisimple locus in Proposition \ref{prop: Spencer's description fixed}. First, we introduce some notation. 
	
	\begin{definition}
		\label{def: canC definition}
		We denote by $\canC\subset \canT$ the kernel of the canonical projection $\canT\to \canA$. For every $\nu\in W/W_X$, we further denote $\canC_\nu = w_\nu\cdot \canC$.
	\end{definition}
	
	Note that $\canC$ is $W_X$ stable, and so the definition of $\canC_\nu$ above is independent of the choice of lift of $\nu$.

	We now state the revised description of $J$, correcting the errors in the literature referenced in Remark \ref{rmk: specific errors in literature}.
	
	\begin{proposition}
		\label{prop: Spencer's description fixed}
		Assume that the characteristic of $k$ is coprime to the order of $W$. The map \eqref{eqn: Galois description of JG} restricted to $\fc^\reg$ sends $J|_{\fc^\reg}$ to the Weil restriction 
		\begin{equation}
			\label{eqn: J over rss}
			\mathrm{Res}^{\canft\times_{\fc_G}\fc^\reg}_{\fc^\reg}\left(\bigoplus_{\nu\in W/W_X} \canC_\nu\times \canfa_\nu^\reg\right)^W.
		\end{equation}
		In particular, $J$ is open in the closure of \eqref{eqn: J over rss} in $\mathrm{Res}^{\canft\times_{\fc_G}\fc}_\fc(\canT\times(\canft\times_{\fc_G}\fc))^W$.
	\end{proposition}
	\begin{proof}
		First, we recall the construction of the morphism $J_G\to \mathrm{Res}^\canft_{\fc_G}\big(\canT\times_{\fc_G}\canft\big)^W$ from Lemma \ref{lem: Galois description of JG}, which was described in the proof of \cite[Proposition 2.4.2]{ngo2010lemme}. To describe this morphism, it is equivalent to describe the corresponding morphism $\pi_\fg^*I_{G}^\reg\to \canT\times \wt{\fg}^\reg$, where 
		\[
		\wt{\fg} = \{(x,B)\colon x\in \fg,\; B\text{ is a Borel of $G$},\; \text{ and }x\in \mathrm{Lie}(B)\}
		\]
		and the projection $\pi_\fg\colon \wt{\fg}\to \fg$ is the Grothendieck-Springer resolution. For this, \cite[Lemma 2.4.3]{ngo2010lemme} shows that $I_{G,x}\subset B$ for any pair $(x,B)$ in $\wt{\fg}^\reg$. The corresponding morphism $\pi_\fg^*I_G\to \canT\times \wt{\fg}^\reg$ at a point $(x,B)$ comes from the map 
		\[
		I_{G,x}\hookrightarrow B\to B/[B,B] = \canT
		\]
		Now, consider the composition
		\begin{equation}
			\label{eqn: J to JG1}
			J\to J_G|_\fc\to \mathrm{Res}^{\fc\times_{\fc_G}\canft}_{\fc}\big(\canT\times (\fc\times_{\fc_G}\canft)\big)^{W},
		\end{equation}
		where the second map is the Galois description of $J_G$ in Lemma \ref{lem: Galois description of JG}. Let $\wt{\fp}^\reg$ denote the base change of the diagram
		\[
		\xymatrix{
			\wt{\fp}^\reg\ar[d]^-{\pi_\fp} \ar[r] & \canft\times_{\fc_G}\fc\ar[d] \\
			\fp^\reg\ar[r] & \fc
		}
		\]
		By \ref{lem: t is a union of a's}, we can express $\wt{\fp}^\reg$ as a union of components
		\[
		\wt{\fp}^\reg = \bigcup_{\nu\in W/W_X}\wt{\fp}^\reg_\nu.
		\]
		By \cite[Proposition 3.5 and page 87]{leslie} the component $\wt{\fp}^\reg_1$ is the collection of pairs $(x,B)$ with $x\in \fp^\reg$ and $B$ a Borel of $G$ such that $B\cap \theta(B) = T$ is a maximal torus of $G$. (These are the so-called ``$\theta$-split Borels'' of $G$.) For any such Borel, the canonical involution on $\canT$ is obtained by the canonical involution on $T$ via the identification $T\simeq \canT$ induced by $B$. For any pair $(x,B)\in \wt{\fp}^\reg$, we claim that the map
		\[
		I_x\hookrightarrow B \to B/[B,B] = \canT
		\]
		has image in $\canC$. Over $\fc^\reg$, let $T = B\cap \theta(B)$, and identify $T\simeq \canT$ via $B$. Then, $I_x\subset I_{G,x} = T$ since $x\in \fp^{\mathrm{rss}}\subset \fg^{\mathrm{rss}}$ is regular, semisimple, and by the definition of the canonical involution (see Lemma \cite[Proposition 1.12]{leslie}), $I_x$ coincides with the group scheme $\canC$ under this map. Hence, we conclude that the restriction of the map $\pi_\fp^*I^\reg|_{\fp^{\mathrm{rss}}}\to \canT$ to $\wt{\fp}^{\mathrm{rss}}_1$ has image equal to $\canC$.
		
		Now, for any other component $\wt{\fp}_\nu^\reg$, the $W$-equivariance of the map $\pi_\fp^*I^\reg|_{\fp^{\mathrm{rss}}}\to \canT$ together with the image over $\wt{\fp}^{\mathrm{rss}}_1$ imply that the image is $\canC_\nu$. We conclude that the image of $J|_{\fc^\reg}$ is given by \eqref{eqn: J over rss}.
		
		As $J$ is closed in $J_G$ and the map \eqref{eqn: Galois description of JG} is open and an isomorphism over the regular, semisimple locus, $J$ is locally closed and is open in the closure of its image in the regular, semisimple locus.
	\end{proof}

	\begin{remark}
		We note that the above formulation is close to that of \cite[Theorem 4.7]{leslie}, but we use all components of $\canft\times_{\fc_G}\fc$ rather than restricting to the canonical component $\canfa_1\subset \canft\times_{\fc_G}\fc$.
	\end{remark}
	
	The description of Proposition \ref{prop: Spencer's description fixed} is quite inexplicit. We will prove a better description of the cokernel in the next section using this result.

	\subsection{Application to the Cokernel Group Scheme $J_X$}
	\label{sec: relative dual centralizers}
	
	We continue the notation of the previous section, with $\canfa = \mathrm{Lie}(\canA)$ the Lie algebra of the canonical torus. Recall that by Lemma \ref{lemma: qspt implies unramified map of hitchin bases}, the natural map $\fc\to \fc_G$ is an embedding when $X$ is $\theta$-quasisplit, and we have an inclusion $J\to J_G|_\fc$. Let us denote the cokernel of this inclusion by $J_X$. The goal of this section, is to give a Galois description for $J_X$ using the finite flat cover $\canfa\to \fc$. 
	
	Indeed, this idea is not new to this paper:  While Knop did not use the construction of regular centralizers in \cite{knop1996automorphisms}, he did consider the Weil restriction $J_A^1:=\mathrm{Res}^\canfa_\fc(\canA\times \canfa)^{W_\fa}$. The key result of \emph{loc cit} was providing an action of the fiberwise neutral component of $J_A^1$, which we denote by $J_A^0$, on the cotangent bundle $T^*X$ over $\fc$ for $X$ a $G$-variety. He used this action to characterize the group of $G$-equivariant automorphisms of $X$.
	
	In fact, we will show that Knop's groups $J_A^0$ and $J_A^1$ are closely related to the cokernel $J_X$ in the case where $X$ is a symmetric space. We prove the following comparison.
	\begin{thm}
		\label{thm: galois description of JA}
		Assume that the characteristic of $k$ does not divide the order of $W$. There exist open embeddings
		\[
		J_A^0\hookrightarrow J_X\hookrightarrow J_A^1.
		\]
	\end{thm}
	\begin{proof}
		Let $J_G^1 = \mathrm{Res}^\canft_{\fc_G}(\canT\times \canft)^W$ be the smooth group scheme on $\fc_G$ appearing in Lemma \ref{lem: Galois description of JG}. Let $S$ be a $\fc$ scheme, and consider an $S$ point $x\in \left(J_G^1|_{\fc}\right)(S)$, which is the data of a $W$-equivariant map
		\[
		\xi_x\colon S\times_{\fc_G} \canft\to \canT\times (\canft\times_{\fc_G}\fc).
		\]
		Write 
		\[
		S\times_{\fc_G}\canft = \bigcup_{\nu\in W/W_X}S\times_{\fc}\canfa_\nu\quad 
		\]
		and let $\xi_{x,1} = \xi_x|_{S\times_{\fc}\canfa_1}$ be the restriction of $\xi_x$ to the component $S\times_{\fc}\canfa_1$. The composition of $\xi_{x,1}$ with the canonical map $\canT\times (\ft\times_{\fc_G}\fc)\to \canA\times \fa$ gives a $W_X$ equivariant map
		\[
		S\times_{\fc}\canfa_1\to \canA\times \canfa
		\]
	    Identifying $\fa_1\simeq \fa$ via the canonical map, this determines a unique element of $J_A^1(S)$. This induces the desired map $\vartheta\colon J_G^1|_\fc\to J_A^1$. It is a direct computation that the kernel of this map agrees with \eqref{eqn: J over rss} over the regular semisimple locus $\fc^\reg$. 

        Since $\ker(\vartheta)$ is closed in $J_G^1$ and since $J\to J_G|_{\fc}\to J_G^1|_\fc$ restricts to an isomorphism $J|_{\fc^\reg}\to \ker(\vartheta)|_{\fc^\reg}$, the image of the embedding $J\to J_G^1|_\fc$ lies in $\ker(\vartheta)$. As $J$ is closed in $J_G|_\fc$, $J$ coincides with the preimage of $\ker(\vartheta)$ in $J_G|_\fc$. Hence, the induced map $J_X\to J_A^1$ is an injection. It is immediate that this map is an isomorphism over the regular, semisimple locus, so it is an open affine embedding. It follows from \cite[IV$_B$, 4.4]{SGA} that $J_X$ also contains $J_A^0$ as an open subscheme. 
	\end{proof}

	One can ask exactly which subgroup $J_X$ is inside of $J_A^1$.	We state and prove such a description below. 
	
	\begin{thm}
		\label{thm: exact image in JA1}
		Assume that the characteristic of $k$ does not divide the order of $W$. For any restricted root $\alpha\in \Phi_r$, define the corresponding hyperplane $\fh_\alpha$ to be the fixed points of the reflection $s_\alpha$ on $\canfa$. Define the open subgroup $J_X'\subset J_A^1$ to be the subscheme of $J_A^1$ consisting of $S$-points $x$ for which the composition 
		\[
		S\times_\fc \fh_\alpha \hookrightarrow S\times_{\fc}\canfa\xrightarrow{x} \canA\xrightarrow{\alpha} \bG_m
		\]
		avoids $-1$. Then, the embedding $J_X\hookrightarrow J_A^1$ of Theorem \ref{thm: galois description of JA} identifies $J_X$ with $J_X'$.
	\end{thm}
	\begin{proof}
        For any $\fc$ scheme $S$ and $x\in J_A^1(S)$, we say $\tilde{x}\in J_G^1|_\fc$ is a lift of $x$ if for all $\nu\in W/W_X$ the following diagram commutes
		\[
		\xymatrix{
			S\times_\fc \canfa\ar[r]^-{x}\ar@{^{(}->}[d] & \canA \\
			S\times_\fc\canft\ar[r]^-{\tilde{x}} & \canT\ar[u]
		}
		\]
        where the left vertical map is induced by the canonical section $\fa\simeq \fa_1\to \ft$ and the right vertical map is the canonical map $\canT\to \canA$. By Lemma \ref{lem: Galois description of JG}, the image of $J_G|_\fc$ in $J_A^1$ consists of $S$-points $x$ which have a lift $\tilde{x}$ satisfying the property that
		\[
		S\times_{\fc_G}\fh_{G,\beta}\to S\times_{\fc_G} \canft\xrightarrow{\tilde{x}} \canT\xrightarrow{\alpha}\bG_m
		\]
		avoids $-1$ for all roots $\beta\in \Phi$, where $\fh_{G,\beta}\subset \canft$ is the root hyperplane for $\beta$. The result follows as 
		\[
		\fh_{G,\beta}\times_{\fc_G}\fc = \bigcup_{\nu\in W/W_\fa} w_\nu\cdot \fh_{\mathrm{res}(\beta)}
		\]
		for all $\beta\in \Phi$ and for $\mathrm{res}\colon \Phi\to \Phi_r$ the restriction map from $\canft$ to $\canfa_1$.
	\end{proof}
	
	\begin{remark}
		In \cite{me2}, we show that the group scheme appearing in Theorem \ref{thm: exact image in JA1} corresponds by \cite[Theorem 7.7]{knop1996automorphisms} to the spherical root system of $X$ (with the Sakellaridis-Venkatesh renormalization).
	\end{remark}
	
	\begin{remark}
		One may note that the dual group $G_X^\vee$ is constructed with maximal torus the dual of the canonical torus of $X$ and with root system given by the dual of the spherical root system. Hence, Lemma \ref{lem: Galois description of JG} applied to the group $G_X^\vee$ also describes the regular centralizers for $G_X^\vee$ in terms of the finite flat map $\canfa\to \fc$. In fact, one will find that the group scheme of regular centralizers for $G_X^\vee$ is dual to $J_X$ in an analogous fashion to the duality between $J_G$ and $J_{G^\vee}$. This duality can be made precise; see our companion work on duality for relative Hitchin systems \cite{me2}.
	\end{remark}
	
	\subsection{Direct description of $J$}
	
	In this section, we give a more direct description of $J$, see Theorem \ref{thm: exact description of J}. The argument uses the description of $J$ as a kernel given in Theorem \ref{thm: galois description of JA}. As an application, we give conditions under which the description given in \cite{leslie} holds.
	
	\begin{lemma}
		\label{lem: normalization of t x c}
		The normalization of $\ft\times_{\fc_G}\fc$ is the disjoint union $\coprod_{\nu \in W/W_{X}} \fa_{\nu}$.
	\end{lemma}
	\begin{proof}
		Follows from \cite[Tag 0CDV]{stacks2022stacks} and Lemma \ref{lem: t is a union of a's}.
	\end{proof}
	
	We denote $\Norm = \coprod_{\nu \in W/W_{X}} \fa_{\nu}$ the normalization of Lemma \ref{lem: normalization of t x c}. We further denote
	\[
	\Pi = \mathrm{Res}^{\fc\times_{\fc_{G}}\ft}_{\fc}\big(T\times (\fc\times_{\fc_{G}}\ft)\big) \quad \text{ and }\quad \ol{\Pi} = \mathrm{Res}^{\Norm}_{\fc}\big(T\times \Norm\big)
	\]
	The group schemes $\Pi$ and $\ol{\Pi}$ are smooth, commutative group schemes over $\fc$. The natural map
	\[
	\Norm\to \ft\times_{\fc_G}\fc
	\]
	induces a morphism $\Pi\to \ol{\Pi}$.
	We define the group scheme
	\begin{equation}
		\label{eqn: J1 over c}
		J^{1}:=\left(\mathrm{Res}^{\Norm}_{\fc}\left(\bigoplus_{\nu\in W/W_X} \canC_\nu\times \canfa_\nu\right)\times_{\ol{\Pi}}\Pi\right)^W.
	\end{equation}
	The projection onto the second factor induces a map
	\begin{equation}
		\label{eqn: J1 to J1G}
		J^1\to J^1_G.
	\end{equation}
	
	\begin{lemma}
		\label{lem: J1 to J1G is closed}
		Assume that the characteristic of $k$ does not divide the order of $W$. Then the map \eqref{eqn: J1 to J1G} is a closed embedding.
	\end{lemma}
	\begin{proof}
		As closed embeddings are stable under base change, the morphism
		\[
		\mathrm{Res}^{\Norm}_{\fc}\left(\bigoplus_{\nu\in W/W_X} \canC_\nu\times \canfa_\nu\right)\times_{\ol{\Pi}}\Pi \to \Pi
		\]
		is a closed embedding. By the assumption on characteristic, taking $W$ invariants preserves closed embeddings, see for example \cite[Remark 6(iii)]{mumford}.
	\end{proof}

	Let $J' = J_G\times_{J_G^1}J^1$. That is, by Lemma \ref{lem: Galois description of JG}, $J'$ is the subgroup scheme of $J^{1}$ defined by the condition that for any root $\alpha$ of $G$ and for any $S$ point $\xi\in J^1(S)$, the composition 
	\[S\times_{\fc}(\fh_{G,\alpha}\cap (\fc\times_{\fc_{G}}\ft))\rightarrow S\times_{\fc}(\fc\times_{\fc_{G}}\ft)\xrightarrow{\xi} T\xrightarrow{\alpha}\bbG_{m},\]
	is trivial. It is immediate by base change that the group scheme $J'$ is an open subgroup of $J^1$ and that $J' \simeq J_G|_{\fc}\cap J^1$.
	
	\begin{thm}
		\label{thm: exact description of J}
		Assume that the characteristic of $k$ does not divide the order of the Weyl group $W$. The natural map $J\to J^1_G|_\fc$ identifies $J\simeq J'$.
	\end{thm}
	
	\begin{proof}
		By construction, the map $J^1_G|_\fc\to J_A^1$ in the proof of Theorem \ref{thm: galois description of JA} fits into a commutative diagram
		\[
		\xymatrix{
			J_G^1|_\fc\ar[rr]\ar@{^{(}->}[d] & & J_A^1\ar@{^{(}->}[d] \\
			\Pi\ar[r] & \ol{\Pi}\ar[r]^-{\Xi} & \mathrm{Res}^\fa_\fc(\canA\times \fa)
		}
		\]
		where the kernel of the map $\Xi$ above is the Weil restriction
		\[
		\mathrm{Res}^{\Norm}_\fc\left(\bigoplus_{\nu\in W/W_X}\canC_\nu\times \fa_\nu\right).
		\]
		In particular, we immediately deduce that the kernel of the map $\Pi\to \mathrm{Res}^\fa_\fc(\canA\times \fa)$ is identified with the fiber product
		\[
		\mathrm{Res}^{\Norm}_\fc\left(\bigoplus_{\nu\in W/W_X}\canC_\nu\times \fa_\nu\right)\times_{\ol{\Pi}}\Pi
		\]
		and so $J^1$ is the kernel of the morphism $J_G^1|_\fc\to J_A^1$. By \ref{thm: galois description of JA}, the group scheme $J$ is the kernel of the morphism $J_G|_\fc\to J_G^1|_\fc\to J_A^1$, and so we conclude that $J = J_G|_\fc \cap J^1 = J'$.
	\end{proof}
	
	Note that Theorem \ref{thm: exact description of J} immediately implies that $J'$ is a smooth commutative group scheme over $\fc$.
	
	We now use Theorem \ref{thm: exact description of J} to study when existing results in the literature do hold. We will show that Example \ref{ex: sl3 case} is essentially the only case where the results of \cite{leslie} fail. Recall that $\canC = \canC_1$ was defined in Definition \ref{def: canC definition}.
	
	\begin{corollary}
		\label{cor: correcting SL and GPPN}
		Suppose that a symmetric pair has no descendant isogeneous to a product of a torus and Example \ref{ex: sl3 case}. Then, we have
		\[
		J^1\simeq \mathrm{Res}^\fa_\fc(\canC\times \fa)^{W_X}
		\]
	\end{corollary}
	\begin{proof}
		For any descendant $(L,\theta,H_L)$, we have an agreement of canonical tori $\canA_L = \canA$. Recall that we have a map of GIT quotients
		\[
		\varphi_L\colon \fc_L\to \fc.
		\]
		Put 
		\[
		J_{A,L}^1 = \mathrm{Res}^\fa_{\fc_L}(\canA\times \fa)^{W_{X,L}}
		\]
		Then we have an isomorphism $\varphi_L^*J_A^1 \simeq J_{A,L}^1$ on open set $V_L\subset \fc_L$ over which $\varphi_L$ is unramified (see Definition \ref{definition: UL and sundry}). Similarly, we have $\varphi_L^*J_G^1 \simeq J_{L}^1$ over $V_L$, fitting into a Cartesian diagram
		\[
		\xymatrix{
			\varphi_L^*J_A^1\ar[r]\ar[d]^-{\cong} & J_G^1 \ar[d]^-{\cong} \\
			J_{A,L}^1\ar[r] & J_L^1
		}
		\]
		Hence, we have an isomorphism
		\[
		\varphi_L^*J^1 \simeq \ker(J_L^1|_{\fc}\to J_{A,L}^1)
		\]
		over $V_L$. 
		
		Now, the normalization map $\Pi\to \ol{\Pi}$ induces a map 
		\begin{equation}
			\label{eqn: correction map}
			J^1\to \mathrm{Res}^\fa_\fc(\canC\times \fa)^{W_X}
		\end{equation} 
		As this is a map of smooth affine group schemes, it suffices to show that \eqref{eqn: correction map} is an isomorphism away from codimension 2. Therefore, by the above discussion, we are reduced to considering semisimple rank 1 descendants.
		Moreover, for any symmetric pair $(G,\theta,H)$ with adjoint form $(G_{\mathrm{ad}},\theta_{\mathrm{ad}},H/Z(G)\cap H)$, there is a short exact sequence
		\[
		1\to (Z(G)\cap H)\times \fc\to \mathrm{Res}^\fa_\fc(\canC\times \fa)^{W_X} \to \mathrm{Res}^\fa_\fc(\canC_{\mathrm{ad}}\times \fa)^{W_X}\to 1
		\]
		where $\canC_{\mathrm{ad}}= \canC/Z(G)\cap H$ is the kernel of the canonical map $\canT_{\mathrm{ad}}\to \canA_{\mathrm{ad}}$ for the adjoint symmetric pair. Hence, we may further reduce to the case of rank 1 symmetric pairs $(G,\theta,H)$ with $G$ adjoint, simple.
		
		There are four adjoint simple rank 1 symmetric pairs:
		\begin{enumerate}
			\item $PGL_{2}$, with involution coming from conjugation by $\left(\begin{matrix}1 & 0\\ 0 & -1\end{matrix}\right)$.
			\item $PGL_{3}$, with involution coming from conjugation by $\left(\begin{matrix}1 & 0\\ 0 & -Id_{2}\end{matrix}\right)$.
			\item $PGL_{2}$ with involution $g\mapsto g^{-1}$.
			\item $PGL_{2}\times PGL_{2}$, with involution swapping the two $PGL_{2}$ factors.
		\end{enumerate}
		We now check by direct computation:
		
		In cases (1) and (3), we may identify
		\[
		\canC = Z(\PGL_2) = \{1\}
		\]
		so the claim follows trivially.
		
		In case (4), we have $\canC \simeq T_{\PGL_2}$ embedded diagonally in $T_{\PGL_2}\times T_{\PGL_2}$. The little Weyl group $W_X\subset W$ is the diagonally embedded $S_n\subset S_n\times S_n$. The map \eqref{eqn: correction map} over $0\in \fc$ is given by
		\[
		\canC^W\to \canC^{W_X}
		\]
		which are both the trivial group.
		
		As case (3) is excluded, the result now follows.
	\end{proof}

	\begin{remark}
		Note that for case (3) in the proof above, we have
		\[
		\canC = \left\{ \begin{pmatrix}
			a & & \\ & a & \\ & & b
		\end{pmatrix}\in P(\GL_1\times \GL_2)\colon a,b\in k^\times \right\}.
		\]
		(See Example \ref{ex: sl3 case}.) Therefore, the map \eqref{eqn: correction map} is, over $0\in \fc$,
		\[
		(\canC^W = \{1\})\to (\canC^{W_{X}} = \canC)
		\]
		which is not an isomorphism.
	\end{remark}
	
	\begin{remark}
		Under the assumptions of Corollary \ref{cor: correcting SL and GPPN}, the descriptions of $J$ due to S. Leslie \cite{leslie} is valid.
	\end{remark}

	\section{General Structure of the Hitchin Fibration}
	\label{sec: general structure}
	
	In this section, we introduce the Hitchin morphism and prove the basic structure theorems (see Corollary \ref{cor: Areg to A}, Proposition \ref{prop: Areg to A is etale}, and Theorem \ref{thm: general structure of h myfatslash}). Further structure, analogous to \cite{ngo2010lemme}, is the subject of ongoing work.
	
	Fix a smooth projective curve $C$ of genus at least 2 and a line bundle $D$ on $C$ which is either the canonical bundle or has degree $\deg(D)\geq 2g$. Also fix a symmetric pair $(G,\theta,H)$, and let $\cM = \mathrm{Maps}(C,[\fp_D/H])$ denote the stack of maps from the curve $C$ to the twisted stack quotient $[\fp_D/H]$ where $\fp_D = \fp\otimes D$. On $k$-points, $\cM$ classifies pairs
	\[
	\cM(k) = \{(\cT_H,\sigma)\colon \cT_H\text{ is a $H$ torsor and }\sigma\in \Gamma(C,\cT_H\wedge^{\mathrm{Ad}} \fp_D)\}
	\]
	We have a Hitchin base $\cA = \mathrm{Maps}(C,\fc_D)$ classifying maps from $C$ to the twisted GIT quotient $\fc_D = (\fp\git H)_D$. In particular, by Theorem \ref{theorem : Chevalley}, there is a $\bG_m$ equivariant isomorphism $\fc\simeq \bA^r$ where the $\bG_m$ action on $\bA^r$ is given by exponents $(e_1,\dots, e_r)$ of Lemma \ref{lem: exponents}. This induces an identification of affine spaces
	\[
	\cA \simeq \bigoplus_{i=1}^r H^0(C,D^{\otimes e_i})
	\]
	There is a natural Hitchin morphism
	\[
	h\colon \cM\to \cA
	\]
	induced by the Chevalley map $[\fp/H]\to \fc$. We restrict our attention to the regular locus in $\cM$; namely, we let $\cM^\reg = \mathrm{Maps}(C,[\fp^\reg_D/H])$ be the substack of $\cM$ classifying maps $C\to [\fp_D/H]$ which factor through the open substack $[\fp^\reg_D/H]\subset [\fp_D/H]$. We abuse notation to denote $h\colon \cM^\reg\to \cA$.
	
	To study the geometry of $h$ over the regular locus, we introduce the space $$\cA^{\reg} = \mathrm{Maps}(C,(\fp^\reg\myfatslash H)_D).$$ The factorization $[\fp^\reg/H]\to \fp^\reg\myfatslash H\to \fc$ induces a factorization
	\[
	\cM^\reg\xrightarrow{h^{\reg}} \cA^{\reg}\xrightarrow{\phi} \cA.
	\]
	
	We consider the following subsets in the Hitchin base $\cA$.
	\begin{definition}
		For any closed, $\bG_m$-stable subset $Z\subset \fc$, denote by $\cA_{Z}^\heartsuit\subset \cA$ the subvariety consisting of $S$ points
		\[
		a\colon C\times S\to \fc_D
		\]
		whose image is not completely contained in $Z_D$.
		
		Similarly, define $\cA_Z^\diamondsuit\subset \cA$ to be the open subvariety consisting of $S$ points $a$ as above whose image intersects $Z_D$ transversely.
	\end{definition} 
	
	We let $\cA^{\reg,\heartsuit}$, resp. $\cA^{\reg,\diamondsuit}$, denote the base change $\cA^{\reg}\times_{\cA}\cA^\heartsuit$, resp. $\cA^\reg\times_\cA\cA^\diamondsuit$. We describe the map $\cA^{\reg,\heartsuit}\to \cA^\heartsuit$ below. Note that if $Z\subset \fc$ is of codimension at least 2, then $\cA_Z^\diamondsuit$ consists of $S$ points $a$ whose image is disjoint from $Z_D$.
	
	Now, let $Z\subset \fc$ be the minimal closed subset away from which $\fp^\reg\myfatslash H\to \fc$ is an isomorphism. We assume without loss of generality that $Z$ is irreducible. Note that this is true whenever $G$ is almost simple. Furthermore, we assume that
	\[
	\fp^\reg\myfatslash H = \fc\coprod_{\fc\setminus Z}\fc,
	\]
	as is the case if $G$ is almost simple, and the symmetric pair $(G,\theta,H)$ is not the split form of type $D_{2n}$.
	
	We now restrict further to the case when the nonseparated locus $Z$ is a divisor in $\fc$. Note that it is immediate that $Z$ is preserved by the $\bG_m$ action since the $\bG_m$ on $\fc$ lifts to $\fp^\reg\myfatslash H$.
	Let $d(Z)$ denote the exponent of the $\bG_m$ action on $Z$, and put
	\[
	d :=d(Z)\deg(D)
	\]
	There is an evaluation map
	\[
	\cA_Z^\heartsuit\to \Sym^{d}(C)
	\]
	taking an $S$ point $a\colon C\times S\to \fc_D$ to the preimage $a^{-1}(Z_D)$.
	
	We may define a similar evaluation map for $\cA_Z^{\reg,\heartsuit}$. Namely, we denote the two components in $\fp^\reg\myfatslash H|_{Z}$ by $Z_1$ and $Z_2$. Then, we have an evaluation map
	\begin{equation}
		\label{eqn: evaluation map on Areg}
		\cA_Z^{\reg,\heartsuit}\to \coprod_{i=0}^d \Sym^i(C)\times \Sym^{d-i}(C)
	\end{equation}
	taking an $S$ point $a\colon C\times S\to (\fp^\reg\myfatslash H)_D$ to the preimage $(a^{-1}(Z_1),a^{-1}(Z_2))$. The resulting diagram
	\[
	\xymatrix{
		\cA_Z^{\reg,\heartsuit}\ar[r] \ar[d] & \coprod_{i=0}^d \Sym^i(C)\times \Sym^{d-i}(C)\ar[d] \\
		\cA_Z^{\heartsuit}\ar[r] & \Sym^d(C)
	}
	\]
	is commutative.
	
	\begin{prop}
		\label{prop: Areg to A over Aheart}
		Suppose that $Z\subset \fc$ is such that
		\[
		\fp^\reg\myfatslash H = \fc\coprod_{\fc\setminus Z}\fc
		\]
		and $Z$ is an irreducible divisor in $\fc$. Let 
		\[
		\Sym^{i,j}(C)\subset\Sym^i(C)\times \Sym^{j}(C)
		\]
		denote the open subscheme
		\[
		\{(D',D'')\in \Sym^i(C)\times \Sym^{j}(C)\mid \mathrm{Supp}(D')\text{ and Supp}(D'')\text{ are disjoint}\}.
		\]
		Then, the image of the evaluation map \eqref{eqn: evaluation map on Areg} is equal to the disjoint union $\coprod_{i=0}^d\Sym^{i,d-i}(C)$ and the resulting diagram
		\begin{equation}
			\label{eqn: comm diagram for Areg to A}
			\xymatrix{
				\cA_Z^{\reg,\heartsuit}\ar[d]\ar[r] & \coprod_{i=0}^d \Sym^{i,d-i}(C)\ar[d] \\
				\cA_Z^{\heartsuit}\ar[r] & \Sym^d(C)
			}
		\end{equation}
		is Cartesian.
	\end{prop}
	
	\begin{remark}
		If $\fp^\reg\myfatslash H$ has more complicated structure over a closed subset $Z'\subset Z$, then the above proposition remains valid over the locus of points $a\in \cA(S)$ for which the image of $a$ is disjoint from $Z'_D$.
	\end{remark}
	
	\begin{proof}
		We first characterize the image of the evaluation map \eqref{eqn: evaluation map on Areg}. We immediately get a decomposition into components $\cA^{\reg,\heartsuit}_Z = \cup_i \cA^{\reg,\heartsuit}_{Z,i}$, where $\cA^{\reg,\heartsuit}_{Z,i}$ denotes the preimage of $\Sym^i(C)\times \Sym^{d-i}(C)$ under \eqref{eqn: evaluation map on Areg}. It is clear that the image of $\cA^{\reg,\heartsuit}_{Z,i}$ consists of 
		\[
		(D',D'')\in \Sym^i(C)\times \Sym^{d-i}(C)
		\]
		such that the supports of $D'$ and $D''$ are distinct, which is exactly the open subset $\Sym^{i,d-i}(C)$. Consider the diagram
		\[
		\xymatrix{
			\cA_{Z,i}^{\reg,\heartsuit}\ar[d]\ar[r] & \Sym^{i,d-i}(C)\ar[d] \\
			\cA_Z^{\heartsuit}\ar[r] & \Sym^d(C)
		}
		\]
		The preimage of any $S$ point $a\in \cA_Z^{\heartsuit}(S)$ consists of the possible lifts of $a$ to $\tilde{a}\in \cA_{Z,i}^{\reg,\heartsuit}$. It is easy to see this is canonically identified with the fiber of the evaluation of $a$ under the
		\[
		\Sym^{i,d-i}(C)\to \Sym^{d}(C).
		\]
		Hence, the diagram \eqref{eqn: comm diagram for Areg to A} is Cartesian.
	\end{proof}
	
	\begin{cor}
		\label{cor: Areg to A}
		Suppose that $Z$ is an irreducible divisor in $\fc$ and $\fp^\reg\myfatslash H = \fc\coprod_{\fc\setminus Z}\fc$. Then we have a decomposition $\cA^\reg = \cup_{i=0}^d \cA^\reg_i$ where 
		\[
		\cA^\reg_i = \begin{cases}
			\cA_{Z,i}^{\reg,\heartsuit} & \text{if }i\neq 0,d \\
			\cA & \text{if }i=0,d
		\end{cases}
		\]
		where the restriction $\phi|_{\cA_{Z,i}^{\reg,\heartsuit}}$ is described by Proposition \ref{prop: Areg to A over Aheart} and the map $\cA^\reg_i\to \cA$ is the identity for the components $i = 0,d$.
	\end{cor}
	\begin{proof}
		Points $a\in \cA(S)\setminus\cA^{\heartsuit}_Z(S)$ consist of points
		\[
		a\colon C\times S\to \fc_D
		\]
		whose image over some $s\in S$ is completely contained in $Z_D$. It is clear therefore that if $S$ is local and if $a\in \cA(S)\setminus\cA^{\heartsuit}_Z(S)$, there are exactly two lifts $\tilde{a}\in \cA^\reg(S)$ and these lie in the closure of $\cA^{\reg,\heartsuit}_{Z,0}$ and $\cA^{\reg,\heartsuit}_{Z,d}$, respectively. The result now follows.
	\end{proof}

The following is immediate from the description in Corollary \ref{cor: Areg to A}.
    
	\begin{cor}
		\label{cor: Areg to A is etale for divisor case}
		Suppose that $Z$ is an irreducible divisor in $\fc$ and $\fp^\reg\myfatslash H = \fc\coprod_{\fc\setminus Z}\fc$. Then the map $\phi\colon \cA^{\reg}\to \cA$ is \'etale.
	\end{cor}
	
	In fact, it is always true that the map $\phi$ is \'etale. 
	
	\begin{prop}
		\label{prop: Areg to A is etale}
		The map $\phi\colon \cA^\reg\to \cA$ is \'etale. That is, Corollary \ref{cor: Areg to A is etale for divisor case} holds without any assumptions.
        \end{prop}
	\begin{proof}
		We note first that $\cA^\reg$ is finite type by \cite[Pages 267-68]{grothendieck}. Hence, the morphism $\phi$ is locally of finite type, and it suffices to show it is formally \'etale. Suppose that we have a commutative diagram
		\[
		\begin{tikzcd}
			T \arrow{d}\arrow{r} & \cA^\reg\arrow{d}\\
			T' \arrow{r}\arrow[dashed]{ru} & \cA
		\end{tikzcd}
		\]
		where $T\rightarrow T'$ is a first order thickening of affine schemes over $\cA$. We will show that there exists a unique lifting $T'\to \cA^\reg$ making the above diagram commute.
		
		By definition, this corresponds to there being a dotted morphism making the below diagram commute.
		\[
		\begin{tikzcd}
			C\times T \arrow{r} \arrow{d} & (\fp^{reg}\myfatslash H)_D\arrow{d}\\
			C\times T' \arrow{r}\arrow[dashed]{ru} & \fc_D
		\end{tikzcd}
		\]
		Now $\fp^{reg}\myfatslash H\rightarrow \fc$ is \'etale, and hence formally etale. By \cite[Tag 04FD]{stacks2022stacks}, there exists a unique lift $C\times T'\rightarrow (\fp^{reg}\myfatslash H)_D$, concluding the proof.
	\end{proof}

	In addition, we state a result on the structure of the map $h^{\reg}$ when $(G,\theta,H)$ is a $\theta$-quasisplit pair. Firstly note that by the $\bG_{m}$-equivariance of $J$, this group scheme descends to a group scheme on $[\fc/\bbG_{m}]$, which we pull back to a group scheme on $\fc_{D}$, which we also denote by $J$.  Now, for $a\colon S\times C\to (\fp^\reg\myfatslash H)_D$, we view $a$ as a map
	\[
	S\times C\to (\fp^\reg\myfatslash H)_D\to \fc_D
	\]
	and put $J_a :=a^*J$. We may define a commutative group stack $\cP$ over $\cA$ whose $S$ points are given by the space of $J_a$ torsors
	\[
	\cP(S) = (\text{$J_a$-torsors on $S\times C$}).
	\]
	for an $S$ point $a\colon S\times C\to(\fp^\reg\myfatslash H)_D$.
	
	\begin{thm}
		\label{thm: general structure of h myfatslash}
		Let $(G,\theta, H)$ be a $\theta$-quasisplit symmetric pair. The action of $\phi^{*}\cP$ makes $\cM^\reg\rightarrow \cA^{reg}$ into a $\phi^{*}\cP$-torsor. (We recall that $\phi$ denotes the morphism $\phi\colon \cA^{reg}\to\cA$.) 
	\end{thm}
	\begin{proof}
		The morphism $[\fp^\reg/H]\to \fp^\reg\myfatslash H$ is a gerbe banded by regular centralizers. The result is immediate as maps into a $J$-gerbe are torsors under the action of the space of $J$ torsors $\cP$.
	\end{proof}
	
	\begin{remark}
		We note that it is not necessarily the case that $\cM^\reg$ is dense in $\cM$ even over an open subset of $\cA$. For example, in the case of $X = \GL_{p+q}/\GL_p\times \GL_{q}$ ($p\neq q$), it is shown in \cite{bradlow2003surface} that $\cM_X$ contains a non-regular component whose image in $\cA$ is full.
	\end{remark}

	\section{The Symmetric Pair \texorpdfstring{$(\GL_{2n},\GL_n\times\GL_n)$}{(GL2n,GLn x GLn)}}
	
	In this section, we provide a spectral description of $J$ for the example of the symmetric pair $(\GL_{2n},\GL_n\times\GL_n)$.  This is very closely related to the description of Hitchin fiber in \cite{schaposnik2015spectral}.  The result of Schaposnik has the advantage that it is directly describing the Hitchin fiber rather than describing $\cP$ and $\cA^{\reg}$, while our result has the advantage that it is applicable over the entire Hitchin base.
	
	\subsection{The Spectral Cover}
	\label{sec: U(n,n) spectral cover}
	
	The map $\fa\git W_\fa\to \ft\git W$ (here $\ft$ and $W$ are for the group $GL_{2n}$) is an embedding by Lemma \ref{lemma: qspt implies unramified map of hitchin bases}. Over $\ft\git W$, there is a natural spectral cover $\ol{\fc}_{\GL_{2n}}\to \ft\git W$ given by $\ol{\fc}_{\GL_{2n}} = \ft\git S_{2n-1}$ for $S_{2n-1}\subset S_{2n}=W$ the index $2n$ subgroup which is the stabilizer of a fixed element in $\{1,2,...,2n\}$. Explicitly, we have $k[\ft\git W] = k[a_1,\dots, a_{2n}]$ where the $a_{i}$ are the degree $i$ elementary symmetric polynomials in $k[\ft]$, and
	\[
	k[\ol{\fc}_{\GL_{2n}}] = k[\ft\git W][x]/(x^{2n}+a_1x^{2n-1}+\cdots+a_{2n}).
	\]
	The map $\fa\git W_{\fa}\rightarrow \ft\git W$ corresponds to the map $k[a_{1},a_2,...,a_{2n}]\rightarrow k[a_{2},a_4,...,a_{2n}]$ which sends $a_{2i+1}\mapsto 0$ for each $0\leq i <n$.
	
	Let $\ol{\fc} = \ol{\fc}_{\GL_{2n}}\times_{\ft\git W}\fa\git W_\fa$ be the restriction of this spectral cover to $\fa\git W_\fa$. Explicitly, we have
	\[
	k[\ol{\fc}] = k[\fa\git W_\fa][x]/(x^{2n}+a_{2}x^{2n-2}+\cdots +a_{2n-2}x^2+a_{2n}) 
	\]
	The map $ k[\fa\git W_\fa][x]\rightarrow k[\fa\git W_\fa][x]/(x^{2n}+a_{2}x^{2n-2}+\cdots +a_{2n-2}x^2+a_{2n})$ gives a $\bbG_{m}$-equivariant embedding $\overline{\fc}\hookrightarrow (\fa\git W_{\fa})\times \bA^{1}$, where $\bbG_{m}$ acts with weight one on $\bA^{1}$.
	
	The involution $\theta$ on $GL_{2n}$ acts on $\ft$ and also on the quotient $\ft\git S_{n-1}$. Explicitly, this action takes $x\mapsto -x$ and $a_i\mapsto (-1)^ia_i$. The spectral cover $\ol{\fc}\subset \ft\git S_{2n-1}$ is preserved by this action, and so we have an involution $i\colon \ol{\fc}\to \ol{\fc}$ defined over $\fa\git W_\fa$ taking $x\mapsto -x$. We denote by
	\[ p\colon \ol{\fc}/i\to \fa\git W_\fa \]
	the quotient of the cover $\ol{\fc}\to \fa\git W_\fa$. We will refer to the map $p$ as the (generic) spectral cover of $\fa\git W_\fa$. We note that $p$ corresponds to the inclusion 
	\[k[\fa\git W_{\fa}]\hookrightarrow  k[\fa\git W_{\fa}][y]/(y^{n}+a_{2}y^{n-1}+...+a_{n}).\]
	The map 
	\[ k[\fa\git W_{\fa}][y]\rightarrow  k[\fa\git W_{\fa}][y]/(y^{n}+a_{2}y^{n-1}+...+a_{n})\]
	gives a $\bbG_{m}$-equivariant embedding 
	\[\overline{\fc}/i\hookrightarrow \fc\times \bA^{1},\]
	where $\bbG_{m}$ acts on $\bA^{1}$ with weight two.
	
	Finally we want to note that there is a particularly nice description of the set $U$ of Proposition \ref{proposition: gluing pattern U(n,n)} via the spectral cover.
	
	\begin{proposition}
		\label{Spectral description of $U$}
		The set $U$ of Proposition \ref{proposition: gluing pattern U(n,n)} is $k[\fa\git W_{\fa}][a_{2n}^{-1}]$, that is to say it is the complement of the vanishing locus of $a_{2n}$.
	\end{proposition}
	
	The interpretation of this in terms of the spectral cover is that the vanishing locus of $a_{2n}$ is precisely the image in $\fa\git W_{\fa}$ of the intersection $((\fa\git W_{\fa})\times \{0\})\times_{(\fa\git W_{\fa})\times \bA^{1}}\overline{\fc}/i$.
	
	\begin{proof}
		This is immediate from the definition of $U$ in proposition \ref{proposition: gluing pattern U(n,n)}.
	\end{proof}

	Recall that the regular centralizer group scheme $I_{K}^\reg\rightarrow \fp^{\reg}$ descends to a smooth group scheme $J$ on the GIT quotient $\fp^{\reg}\git (\GL_n\times \GL_n)\simeq \fa\git W_\fa$ since the symmetric pair is $\theta$-quasisplit. We give a description of $J$ using the spectral cover above.
	
	\begin{proposition}
		\label{proposition: U(n,n) spectral cover}
		There is a natural map $J\to \mathrm{Res}_{\fa\git W_\fa}^{\ol{\fc}/i}(\bG_m)$ where $\mathrm{Res}_{\fa\git W_\fa}^{\ol{\fc}/i}(\bbG_{m})$ denotes the Weil restriction of $\bG_m$ along the map $p\colon \ol{\fc}/i\to \fc$. This map is an isomorphism. 
	\end{proposition}
	\begin{proof}
		Note that one has the description of regular centralizers of the adjoint action of $G$ on $\fg$ as the Weil restriction of $\bG_m$ along the spectral cover $\ol{\fc}_{\GL_{2n}}\to \ft\git W$. In particular, restricting to $\fa\git W_\fa\subset \ft\git W$, it follows that there is an isomorphism
		\[
		J\xrightarrow{\sim} \mathrm{Res}^{\ol{\fc}}_{\fa\git W_\fa}(\bG_m)^i
		\]
		of $J$ with the $i$-invariant locus in $\mathrm{Res}^{\ol{\fc}}_{\fa\git W_\fa}(\bG_m)^i$. Since
		\[
		\mathrm{Res}^{\ol{\fc}}_{\fa\git W_\fa}(\bG_m)^i = \mathrm{Res}^{\ol{\fc}/i}_{\fa\git W_\fa}\left(\mathrm{Res}^{\ol{\fc}}_{\ol{\fc}/i}(\bG_m)\right)^i
		\]
		it therefore suffices to show that 
		\[
		\mathrm{Res}^{\ol{\fc}}_{\ol{\fc}/i}(\bG_m)^i \simeq \bG_m.
		\]
		Note that we have a map
		\[
		\xi\colon \bG_m\to \mathrm{Res}^{\ol{\fc}}_{\ol{\fc}/i}(\bG_m)^i.
		\]
		Namely, an $S$-point $S\to \bG_m\times \ol{\fc}/i$ has image given by the base change
		\[
		\xymatrix{
			S\ar[r] & \bG_m\times \ol{\fc}/i \\
			S\times_{\ol{\fc}/i}\ol{\fc}\ar[u]\ar[r] & \bG_m\times \ol{\fc}\ar[u]
		}
		\]
		We claim this map is an isomorphism. Let $\cD\subset \ol{\fc}/i$ be the ramification locus of the map $\ol{\fc}\to \ol{\fc}/i$, i.e. the image of the fixed point locus of $i$ in $\ol{\fc}/i$. For $x\in (\ol{\fc}/i)\setminus \cD$, we have the stalk
		\[
		\mathrm{Res}^{\ol{\fc}}_{\ol{\fc}/i}(\bG_m)_x = \bG_m\times \bG_m
		\]
		with $i$ acting by swapping the two factors. As the preimage of such an $x$ is two points, it is easy to see that the map $\xi_x\colon \bG_m\to \bG_m\times \bG_m$ is the diagonal map. Hence, $\xi$ is an isomorphism away from $\cD$. For $x\in \cD$, we have the fiber 
		\[
		\mathrm{Res}^{\ol{\fc}}_{\ol{\fc}/i}(\bG_m)_x = \bG_m\times \bG_a
		\]
		with $i$ acting by $(y,z)\mapsto (y,-z)$. The map $\xi_x\colon \bG_m\to \bG_m\times \bG_a$ is the inclusion into the first factor, and we conclude that the map $\xi$ is an isomorphism.
	\end{proof}

	\subsection{Applications to the Hitchin Fibration for \texorpdfstring{$(\GL_{2n},\GL_n\times\GL_n)$}{(GL2n,GLn x GLn)}}
	
	The regular quotient for the case of $(\GL_{2n},\GL_n\times \GL_n)$ was computed in Example \ref{example: structure of regular quotient for U(n,n) case}. Now that we have constructed spectral covers, we give an alternate description. This recovers the work of Schaposnik on spectral covers \cite{schaposnik2013spectral}.
	
	We will make use of the notation from Section \ref{sec: general structure}; in particular, we fix a smooth projective curve $C$ of genus at least 2 and a line bundle $D$ on $C$ of degree at lest $2g$. For any $S$ point $a:S\times C\to \fc_D$, we define the spectral cover at $a$ to be the base change
	\[
	\xymatrix{
		\ol{C}_a\ar[d]\ar[r] & (\ol{\fc}/i)_D\ar[d] \\
		S\times C\ar[r] & \fc_D
	}
	\]
	In particular, we will set $\ol{C}$ to be the base change along the evaluation map $\cA\times C\to \fc_D$. We can realize $\ol{C}_a$ as a subvariety of the total space $\mathrm{Tot}(2D)$ by considering the vanishing locus of the characteristic polynomial equation.
	
	It is immediate from the Weil restriction description of Proposition \ref{proposition: U(n,n) spectral cover} that we have the following description of fibers.
	
	\begin{cor}
		\label{cor: Mreg as a torsor in U(n,n) case}
		The space $\cM^\reg$ is a torsor over $\cA^\reg$ under the abelian group scheme $\Pic(\phi^*\ol{C}/\cA^{\reg})$.
	\end{cor}

	\begin{cor}
		\label{cor: Dns for U(n,n)}
		For any point $a\in \cA(S)$, the image of the map $a\colon S\times C\to \fc_D$ meets the non-separated divisor $Z = \mathfrak{D}^{\mathrm{ns}}$ exactly at the image of the zero section of $\ol{C}_a$ in $\mathrm{Tot}(2D)$.
	\end{cor}
	\begin{proof}
		As the non-separated locus was found to be given exactly by the coordinate axes $\delta^* = 0$ in $\fa$, the claim is immediate.
	\end{proof}
	
	\begin{remark}
		Corollary \ref{cor: Dns for U(n,n)} is related to mirror symmetry in \cite{me2}.
	\end{remark}
	
	We note that Proposition \ref{prop: Areg to A over Aheart} and Corollary \ref{cor: Areg to A} apply to this case, with $Z = \fD_{\mathrm{ns}}$ a divisor on which $\bG_m$ acts by exponent $2n$. We summarize this in the following.
	
	\begin{prop}
		\label{prop: structure of regular quotient for U(n,n)}
		The map  $\phi\colon \cA^{\reg}\to \cA$ is an \'etale map, with degree $2^{2n\deg(D)}$ over $\cA^\diamondsuit$ and with explicit structure as described in Proposition \ref{prop: Areg to A over Aheart} and Corollary \ref{cor: Areg to A}.
	\end{prop}

	\bibliographystyle{amsalpha}
	\bibliography{bib.bib}

\newcommand{\etalchar}[1]{$^{#1}$}
\providecommand{\bysame}{\leavevmode\hbox to3em{\hrulefill}\thinspace}
\providecommand{\MR}{\relax\ifhmode\unskip\space\fi MR }
\providecommand{\MRhref}[2]{%
  \href{http://www.ams.org/mathscinet-getitem?mr=#1}{#2}
}
\providecommand{\href}[2]{#2}
\begin{thebibliography}{BGPMiR20}

\bibitem[AOV08]{AOV}
Dan Abramovich, Martin Olsson, and Angelo Vistoli, \emph{Tame stacks in positive characteristic}, Ann. Inst. Fourier (Grenoble) \textbf{58} (2008), no.~4, 1057--1091. \MR{2427954}

\bibitem[Aut25]{stacks2022stacks}
Stacks~Project Authors, \emph{Stacks project}.

\bibitem[BGPG03]{bradlow2003surface}
Steven~B Bradlow, Oscar Garc{\'\i}a-Prada, and Peter~B Gothen, \emph{Surface group representations and {$U(p,q)$}-{H}iggs bundles}, Journal of Differential Geometry \textbf{64} (2003), no.~1, 111--170.

\bibitem[BGPMiR03]{garcia3}
Steven~B. Bradlow, Oscar Garc\'ia-Prada, and Ignasi Mundet~i Riera, \emph{Relative {H}itchin-{K}obayashi correspondences for principal pairs}, Q. J. Math. \textbf{54} (2003), no.~2, 171--208. \MR{1989871}

\bibitem[BGPMiR20]{garcia2}
Olivier Biquard, Oscar Garc\'ia-Prada, and Ignasi Mundet~i Riera, \emph{Parabolic {H}iggs bundles and representations of the fundamental group of a punctured surface into a real group}, Adv. Math. \textbf{372} (2020), 107305, 70. \MR{4129012}

\bibitem[Bou15]{bouthier2015dimension}
Alexis Bouthier, \emph{Dimension des fibres de {S}pringer affines pour les groupes}, Transform. Groups \textbf{20} (2015), no.~3, 615--663.

\bibitem[Bra18]{branco2018higgs}
Lucas~C Branco, \emph{Higgs bundles, {L}agrangians and mirror symmetry}, Ph.D. thesis, 2018, arXiv preprint arXiv:1803.01611.

\bibitem[BS19]{baraglia2019cayley}
David Baraglia and Laura~P. Schaposnik, \emph{Cayley and {L}anglands type correspondences for orthogonal {H}iggs bundles}, Trans. Amer. Math. Soc. \textbf{371} (2019), no.~10, 7451--7492.

\bibitem[Chi22]{chi2018geometry}
Jingren Chi, \emph{Geometry of {K}ottwitz-{V}iehmann varieties}, J. Inst. Math. Jussieu \textbf{21} (2022), no.~1, 1--65.

\bibitem[CN20]{chen-ngo}
Tsao-Hsien Chen and Bao~Chau Ngo, \emph{On the {H}itchin morphism for higher-dimensional varieties}, Duke Math. J. \textbf{169} (2020), no.~10, 1971--2004. \MR{4118645}

\bibitem[Con]{conrad}
Brian Conrad, \emph{Reductive groups over fields (notes by {T}ony {F}eng on lectures by {B}rian {C}onrad)}.

\bibitem[Cor88]{corlette}
Kevin Corlette, \emph{Flat {$G$}-bundles with canonical metrics}, Journal of differential geometry \textbf{28} (1988), no.~3, 361--382.

\bibitem[DG{\etalchar{+}}70]{SGA}
Michel Demazure, Alexandre Grothendieck, et~al., \emph{Sch{\'e}mas en groupes: s{\'e}minaire de g{\'e}om{\'e}trie alg{\'e}brique}, Springer-Verlag (1970).

\bibitem[DG02]{donagi2002gerbe}
Ron Donagi and Dennis Gaitsgory, \emph{The gerbe of {H}iggs bundles}, Transform. Groups \textbf{7} (2002), no.~2, 109--153.

\bibitem[Don87]{donaldson1987twisted}
Simon~K Donaldson, \emph{Twisted harmonic maps and the self-duality equations}, Proceedings of the London Mathematical Society \textbf{3} (1987), no.~1, 127--131.

\bibitem[DP12]{donagi-pantev}
R.~Donagi and T.~Pantev, \emph{Langlands duality for {H}itchin systems}, Invent. Math. \textbf{189} (2012), no.~3, 653--735. \MR{2957305}

\bibitem[Gin08]{ginzburg2008variations}
Victor Ginzburg, \emph{Variations on themes of {K}ostant}, Transformation Groups \textbf{13} (2008), no.~3, 557--573.

\bibitem[Gir71]{giraud1971cohomologie}
Jean Giraud, \emph{Cohomologie non ab\'{e}lienne}, Die Grundlehren der mathematischen Wissenschaften, Band 179, Springer-Verlag, Berlin-New York, 1971.

\bibitem[GP20]{garcia4}
Oscar Garc\'ia-Prada, \emph{Higgs bundles and higher {T}eichm\"uller spaces}, Handbook of {T}eichm\"uller theory. {V}ol. {VII}, IRMA Lect. Math. Theor. Phys., vol.~30, Eur. Math. Soc., Z\"urich, [2020] \copyright 2020, pp.~239--285. \MR{4321175}

\bibitem[GPGiR12]{garcia2009hitchin}
Oscar Garc\'ia-Prada, Peter~B. Gothen, and Ignasi~Mundet i~Riera, \emph{The {H}itchin-{K}obayashi correspondence, {H}iggs pairs and surface group representations}, 2012, arXiv:0909.4487.

\bibitem[GPPN23]{gppn}
Oscar Garc\'{\i}a-Prada and Ana Pe\'{o}n-Nieto, \emph{Abelianization of {H}iggs bundles for quasi-split real groups}, Transform. Groups \textbf{28} (2023), no.~1, 285--325.

\bibitem[Gro61]{grothendieck}
Alexander Grothendieck, \emph{Techniques de construction et théorèmes d'existence en géométrie algébrique iv : Les schémas de hilbert}, Séminaire Bourbaki, Séminaire Bourbaki, Vol. 6 (1960–1961), no. 221, Secrétariat mathématique, 1961, pp.~1--28.

\bibitem[GWZ20]{GWZ}
Michael Groechenig, Dimitri Wyss, and Paul Ziegler, \emph{Mirror symmetry for moduli spaces of {H}iggs bundles via p-adic integration}, Invent. Math. \textbf{221} (2020), no.~2, 505--596. \MR{4121158}

\bibitem[Hit87a]{hitchin1987self}
N.~J. Hitchin, \emph{The self-duality equations on a {R}iemann surface}, Proc. London Math. Soc. (3) \textbf{55} (1987), no.~1, 59--126.

\bibitem[Hit87b]{hitchin}
Nigel Hitchin, \emph{Stable bundles and integrable systems}, Duke Math. J. \textbf{54} (1987), no.~1, 91--114.

\bibitem[Hit92]{hitchin-teichmuller}
N.~J. Hitchin, \emph{Lie groups and {T}eichm\"uller space}, Topology \textbf{31} (1992), no.~3, 449--473. \MR{1174252}

\bibitem[HLM24]{me2}
Thomas Hameister, Zhilin Luo, and Benedict Morrissey, \emph{Relative {D}olbeault geometric {L}anglands via the regular quotient}, 2024, arXiv preprint, arXiv:2409.15691.

\bibitem[HS14]{hitchin2014nonabelianization}
Nigel Hitchin and Laura~P. Schaposnik, \emph{Nonabelianization of {H}iggs bundles}, J. Differential Geom. \textbf{97} (2014), no.~1, 79--89. \MR{3229050}

\bibitem[HT03]{hausel-thaddeus}
Tam\'as Hausel and Michael Thaddeus, \emph{Mirror symmetry, {L}anglands duality, and the {H}itchin system}, Invent. Math. \textbf{153} (2003), no.~1, 197--229. \MR{1990670}

\bibitem[Jan98]{jantzen1998representations}
Jens~Carsten Jantzen, \emph{Representations of {L}ie algebras in prime characteristic}, Representation theories and algebraic geometry, Springer, 1998, pp.~185--235.

\bibitem[Kno90]{knop_german}
Friedrich Knop, \emph{Weylgruppe und {M}omentabbildung}, Invent. Math. \textbf{99} (1990), no.~1, 1--23. \MR{1029388}

\bibitem[Kno94]{knop-invariant}
\bysame, \emph{The asymptotic behavior of invariant collective motion}, Invent. Math. \textbf{116} (1994), no.~1-3, 309--328. \MR{1253195}

\bibitem[Kno96]{knop1996automorphisms}
\bysame, \emph{Automorphisms, root systems, and compactifications of homogeneous varieties}, J. Amer. Math. Soc. \textbf{9} (1996), no.~1, 153--174. \MR{1311823}

\bibitem[KR71]{KR}
B.~Kostant and S.~Rallis, \emph{Orbits and representations associated with symmetric spaces}, Amer. J. Math. \textbf{93} (1971), 753--809. \MR{311837}

\bibitem[KW07]{kapustin-witten}
Anton Kapustin and Edward Witten, \emph{Electric-magnetic duality and the geometric {L}anglands program}, Commun. Number Theory Phys. \textbf{1} (2007), no.~1, 1--236. \MR{2306566}

\bibitem[Les20]{spencer_annals}
Spencer Leslie, \emph{The endoscopic fundamental lemma for unitary {F}riedberg-{J}acquet periods}, 2020, arXiv preprint, arXiv:1911.07907, To appear in \emph{Annals of Mathematics}.

\bibitem[Les21]{leslie}
Spencer Leslie, \emph{An analogue of the {G}rothendieck-{S}pringer resolution for symmetric spaces}, Algebra Number Theory \textbf{15} (2021), no.~1, 69--107. \MR{4226983}

\bibitem[Les24]{sl_endoscopy}
Spencer Leslie, \emph{Symmetric varieties for endoscopic groups}, 2024, arXiv preprint, arXiv: 2401.09156.

\bibitem[Lev07]{levy}
Paul Levy, \emph{Involutions of reductive {L}ie algebras in positive characteristic}, Adv. Math. \textbf{210} (2007), no.~2, 505--559. \MR{2303231}

\bibitem[MFK94]{mumford}
David Mumford, John Fogarty, and Frances Kirwan, \emph{Geometric invariant theory}, 3rd ed., Ergebnisse der Mathematik und ihrer Grenzgebiete, vol.~34, Springer-Verlag, Berlin, 1994.

\bibitem[MN]{morrissey2022reg}
Benedict Morrissey and Bao~Chau Ngo, \emph{Reqular quotients and {H}itchin type fibrations}, Unpublished.

\bibitem[Ngo10]{ngo2010lemme}
Bao~Chau Ngo, \emph{Le lemme fondamental pour les alg\`ebres de {L}ie}, Publ. Math. Inst. Hautes \'{E}tudes Sci. (2010), no.~111, 1--169.

\bibitem[NS64]{NS}
M.~S. Narasimhan and C.~S. Seshadri, \emph{Stable bundles and unitary bundles on a compact {R}iemann surface}, Proc. Nat. Acad. Sci. U.S.A. \textbf{52} (1964), 207--211. \MR{170350}

\bibitem[OS03]{olsson-starr}
Martin Olsson and Jason Starr, \emph{Quot functors for {D}eligne-{M}umford stacks}, vol.~31, 2003, Special issue in honor of Steven L. Kleiman, pp.~4069--4096. \MR{2007396}

\bibitem[Pan05]{panyu}
Dmitri~I. Panyushev, \emph{On invariant theory of {$\theta$}-groups}, J. Algebra \textbf{283} (2005), no.~2, 655--670. \MR{2111215}

\bibitem[PN13]{peon2013higgs}
Ana Pe{\'o}n-Nieto, \emph{Higgs bundles, real forms and the {H}itchin fibration}, Ph.D. thesis, Universidad Aut{\'o}noma de Madrid, 2013.

\bibitem[Ric82]{richardson}
R.~W. Richardson, \emph{Orbits, invariants, and representations associated to involutions of reductive groups}, Invent. Math. \textbf{66} (1982), no.~2, 287--312.

\bibitem[Ric17]{riche2017kostant}
Simon Riche, \emph{Kostant section, universal centralizer, and a modular derived {S}atake equivalence}, Math. Z. \textbf{286} (2017), no.~1-2, 223--261. \MR{3648498}

\bibitem[Sch13]{schaposnik2013spectral}
Laura~P. Schaposnik, \emph{Spectral data for {G}-{H}iggs bundles}, Ph.D. thesis, 2013, Thesis (D.Phil.)--University of Oxford (United Kingdom). \MR{3389247}

\bibitem[Sch15]{schaposnik2015spectral}
\bysame, \emph{Spectral data for {$U(m,m)$}-{H}iggs bundles}, Int. Math. Res. Not. IMRN (2015), no.~11, 3486--3498. \MR{3373057}

\bibitem[Sch18]{schaposnik2014introduction}
\bysame, \emph{An introduction to spectral data for {H}iggs bundles}, The geometry, topology and physics of moduli spaces of {H}iggs bundles, Lect. Notes Ser. Inst. Math. Sci. Natl. Univ. Singap., vol.~36, World Sci. Publ., Hackensack, NJ, 2018, pp.~65--101. \MR{3837869}

\bibitem[Sek84]{sekiguchi}
Jiro Sekiguchi, \emph{The nilpotent subvariety of the vector space associated to a symmetric pair}, Publ. Res. Inst. Math. Sci. \textbf{20} (1984), no.~1, 155--212. \MR{736100}

\bibitem[Sim90]{simpson}
Carlos~T Simpson, \emph{Nonabelian {H}odge theory}, Proceedings of the International Congress of Mathematicians, vol.~1, 1990, pp.~747--756.

\bibitem[SV17]{sakellaridis2017periods}
Yiannis Sakellaridis and Akshay Venkatesh, \emph{Periods and harmonic analysis on spherical varieties}, Ast\'{e}risque (2017), no.~396, viii+360. \MR{3764130}

\bibitem[Vus74]{vust}
Thierry Vust, \emph{Op\'eration de groupes r\'eductifs dans un type de c\^ones presque homog\`enes}, Bull. Soc. Math. France \textbf{102} (1974), 317--333. \MR{366941}

\bibitem[Wan24]{griffin}
X.~Griffin Wang, \emph{Multiplicative {H}itchin fibration and fundamental lemma}, 2024, arXiv preprint, arXiv:2402.19331.

\end{thebibliography}

\end{document}